\documentclass[a4paper]{article}

\usepackage[top=1in, bottom=1.5in, left=1.1in, right=1.8in]{geometry}
\usepackage{amsfonts}
\usepackage{amsmath}
\usepackage{amsthm,amssymb}
\usepackage{CJK,graphicx}
\usepackage{amscd}
\usepackage{amssymb}
\usepackage{mathrsfs}
\usepackage[all,cmtip]{xy}
\usepackage{lmodern}
\usepackage[Symbol]{upgreek}
\usepackage{bm}
\usepackage{eucal}
\usepackage[nopar]{lipsum}
\usepackage{rotating}
\usepackage{tikz}
\usepackage{calc}
\usepackage{tikz-cd}
\usepackage{mathabx}
\usepackage{tikz-3dplot}
\usepackage{hyperref}
\usetikzlibrary{calc}
\usetikzlibrary{decorations.markings,arrows}
\usetikzlibrary{shapes.geometric}
\usepackage{lipsum}
\usepackage{makecell}

\newsavebox\CBox
\newcommand\hcancel[2][0.5pt]{%
	\ifmmode\sbox\CBox{$#2$}\else\sbox\CBox{#2}\fi%
	\makebox[0pt][l]{\usebox\CBox}%
	\rule[0.5\ht\CBox-#1/2]{\wd\CBox}{#1}}

\newtheorem{theorem}{Theorem}
\newtheorem{corollary}[theorem]{Corollary}

\newtheorem{lemma}[theorem]{Lemma}
\newtheorem{proposition}[theorem]{Proposition}

\newtheorem{remark}[theorem]{Remark}
\newtheorem{example}[theorem]{Example}

\numberwithin{equation}{section}

\begin{document}
	
\title{\textbf{Plumbings of lens spaces and crepant resolutions of compound $A_n$ singularities}}\author{Bilun Xie \and Yin Li}
\newcommand{\Addresses}{{
		\bigskip
		\footnotesize
		Bilun Xie, \textsc{Department of Mathematics, Uppsala University, 753 10 Uppsala, Sweden}\par\nopagebreak\textit{E-mail address}: \texttt{bilun.xie.3601@student.uu.se}
		
		\medskip
		Yin Li, \textsc{Department of Mathematics, Uppsala University, 753 10 Uppsala, Sweden}\par\nopagebreak
		\textit{E-mail address}: \texttt{yin.li@math.uu.se}

}}
\date{}\maketitle

\begin{abstract}
For many compound $A_n$ ($cA_n$) singularities $R_f=\mathbb{C}[u,v,x,y]/(uv-f(x,y))$ with crepant resolutions $Y_f$, their mirrors are affine $A_n$ plumbings $\ring{W}_f$ of $3$-dimensional lens spaces along circles. We prove two versions of homological mirror symmetry for these Stein $3$-folds.
\begin{itemize}
	\item[(i)] The uncompleted version: there is an equivalence $D^\mathit{perf}\mathcal{W}(\ring{W}_f)\simeq D^b\mathit{Coh}(\ring{Y}_f)$ between the derived wrapped Fukaya category and the bounded derived category of coherent sheaves on some divisor complement $\ring{Y}_f=Y_f\setminus D$.
	\item[(ii)] The completed version: there is an equivalence $D^\mathit{perf}\widehat{\mathcal{W}}(\ring{W}_f)\simeq D^b\mathit{Coh}(\widehat{Y}_f)$, where $\widehat{\mathcal{W}}(\ring{W}_f)$ is the completion of $\mathcal{W}(\ring{W}_f)$ with respect to the word-length filtration of Hamiltonian chords, and $\widehat{Y}_f$ is the complete local version of $Y_f$.
\end{itemize}
As an application of (i), we show that certain infinitely generated subgroup of the pure braid group $\mathit{PBr}_{n+2}$ split injects into the compactly supported symplectic mapping class group of $\ring{W}_f$ as long as $R_f$ is isolated, generalizing the work of Keating-Smith \cite{kss} in the case of a conifold smoothing. Applying categorical localization to (ii), we obtain an equivalence between the (uncompleted) derived wrapped Fukaya category of the corresponding (non-affine) $A_n$ plumbing $W_f$ of lens spaces along circles and the relative singularity category of $Y_f$. This generalizes the result of Smith-Wemyss \cite{sw} in the case of double bubble plumbings and partially answers their realization question. 
\end{abstract}

\section{Introduction}

\subsection{Background}\label{section:bg}

The symplectic topologies of cotangent bundles of smooth manifolds and their plumbings (at points) have been studied extensively in the literature, see for example \cite{ase,jas,ekle,etlep,klw}. It is natural to consider the next simplest case, namely the plumbings of cotangent bundles of manifolds which intersect with each other cleanly along smooth submaifolds of dimension $\geq1$, where some foundational results have been established by Abouzaid \cite{maa2}. In \cite{sw}, Smith and Wemyss studied the plumbings of two copies of $T^\ast S^3$ along a circle that is unknotted in both core spheres. These $6$-dimensional Weinstein manifolds $W_k$ are known as \textit{double bubble plumbings} and their geometries depend on the identifications of the normal bundle of the intersecting unknot, which are parametrized by $k\in\mathbb{Z}$. Under mirror symmetry, the double bubble plumbings correspond to certain local Calabi-Yau $3$-folds containing two floppable $(-1,-1)$-curves meeting at a point. The wrapped Fukaya categories $\mathcal{W}(W_k;\mathbb{K})$ of $W_k$ over an arbitrary field $\mathbb{K}$ are computed in \cite{alp} as localizations of a fiberwise version of the partially wrapped Fukaya category introduced earlier by Abouzaid-Auroux \cite{aah}. The same computational techniques actually generalize to many $A_n$ plumbings of cotangent bundles of $3$-dimensional lens spaces along circles, and they correspond, under mirror symmetry, to crepant resolutions of certain $cA_n$ singularities.

In this paper, we study a different class of plumbings of $3$-dimensional lens spaces along circles, in the sense that the plumbing diagram is cyclic instead of linear. More precisely, consider the total space of a Morse-Bott fibration
\begin{equation}\label{eq:MBf}
\ring{\pi}:\ring{W}_f\rightarrow\mathbb{C}^\ast
\end{equation}
over the cylinder. The smooth fibers of $\ring{\pi}$ are isomorphic to $(\mathbb{C}^\ast)^2$, and there are $n+1$ singular fibers located at the $(n+1)$-th roots of unity, i.e. $c_i=\exp(\frac{2\pi i}{n+1})$, where $i=0,\cdots,n$. The fibers over the critical values $c_i$ are all isomorphic to $(\mathbb{C}\vee\mathbb{C})\times\mathbb{C}^\ast$, so the critical loci of $\ring{\pi}$ are disjoint unions of $n+1$ copies of $\mathbb{C}^\ast$, although the corresponding vanishing cycles may be different. The fibration $\ring{\pi}$ has (outer) monodromies both near infinity and around the origin, and we assume that the latter is trivial. Since the monodromy maps around the large circle near infinity are fibered Dehn twists \cite{ww}, which act trivially on the homology of the fiber, we can fix a global basis $a,b$ of $H_1(T^2;\mathbb{Z})$. We assume that with respect to the basis $a,b$, the vanishing cycles at the critical values $c_i$, $i=0,\dots,n$ are given respectively by
\begin{equation}\label{eq:kl}
k_ia\pm l_ib \textrm{ for }k_i,l_i\in\mathbb{Z}_{\geq0}\textrm{ such that at least one of }k_i\textrm{ and }l_i\textrm{ equals }1.
\end{equation}
We can equip $\ring{W}_f$ with a symplectic form $\omega_{\ring{W}}$ so that the smooth fibers of $\ring{\pi}$ are symplectomorphic to $T^\ast T^2$ and symplectic parallel transport is globally defined away from the critical values of $\ring{\pi}$, see Section \ref{section:plumbing}. Consider the arc $\alpha_i\subset\mathbb{C}^\ast$ connecting the critical values $c_i$ and $c_{i+1\textrm{ mod }n}$ and avoiding the other critical values of $\ring{\pi}$. Parallel transporting the zero section of the fiber $T^\ast T^2$ along the arc $\alpha_i$ defines a closed exact Lagrangian submanifold $Q_i\subset\ring{W}_f$, and we can realize $\ring{W}_f$ as the affine $A_n$ plumbing of the cotangent bundles $T^\ast Q_i$ along circles $Z_i=Q_i\cap Q_{i+1\textrm{ mod }n+1}$, where $i=0,\cdots,n$, see Section \ref{section:plumbing}. Since $Q_i$ has Heegaard genus $\leq1$, we conclude that it is diffeomorphic either to $S^1\times S^2$ or a lens space $L(p,q)$ for some $(p,q)=1$, where $(p,q)\in\mathbb{N}^2$ depends explicitly on the slopes $(k_i,\pm l_i)$ and $(k_{i+1},\pm l_{i+1})$ of the vanishing cycles.

\begin{figure}
\centering
\begin{tikzpicture}
\draw (0,0) circle [radius=2.5];
\draw [dashed] (0,0) circle [radius=1.25];
\draw (1.25,0) node[circle,fill,inner sep=1pt] {};
\draw (-1.25,0) node[circle,fill,inner sep=1pt] {};
\draw (0.625,1.0825) node[circle,fill,inner sep=1pt] {};
\draw (0.625,-1.0825) node[circle,fill,inner sep=1pt] {};
\draw (-0.625,1.0825) node[circle,fill,inner sep=1pt] {};
\draw (-0.625,-1.0825) node[circle,fill,inner sep=1pt] {};
\draw (1.5,0) node {$c_0$};
\draw (-1.5,0) node {$c_3$};
\draw (0.8,1.2) node {$c_1$};
\draw (-0.8,1.2) node {$c_2$};
\draw (-0.8,-1.2) node {$c_4$};
\draw (0.8,-1.2) node {$c_5$};

\draw [orange] (0,0)--(2.165,1.25);
\draw [orange] (0,0)--(0,2.5);
\draw [orange] (0,0)--(0,-2.5);
\draw [orange] (0,0)--(-2.165,-1.25);
\draw [orange] (0,0)--(-2.165,1.25);
\draw [orange] (0,0)--(2.165,-1.25);
\draw (0,0) node {$\times$};
\draw [orange] (2.4,1.4) node {$\gamma_1$};
\draw [orange] (2.4,-1.4) node {$\gamma_0$};
\draw [orange] (0,2.7) node {$\gamma_2$};
\draw [orange] (0,-2.7) node {$\gamma_5$};
\draw [orange] (-2.4,1.4) node {$\gamma_3$};
\draw [orange] (-2.4,-1.4) node {$\gamma_4$};

\draw [teal] (1.25,0)--(0.625,1.0825);
\draw [teal] (0.625,1.0825)--(-0.625,1.0825);
\draw [teal] (-0.625,1.0825)--(-1.25,0);
\draw [teal] (-1.25,0)--(-0.625,-1.0825);
\draw [teal] (-0.625,-1.0825)--(0.625,-1.0825);
\draw [teal] (0.625,-1.0825)--(1.25,0);
\end{tikzpicture}
\caption{Base of the Morse-Bott fibration $\ring{\pi}$ when $n=5$. The arcs $\alpha_i$ (in teal) connecting the critical values $c_i$ and $c_{i+1\textrm{ mod }n+1}$ are matching paths of the Lagrangians $Q_0,\cdots,Q_n$, while the cotangent fibers $L_0,\cdots,L_n$ of $T^\ast Q_0,\cdots,T^\ast Q_n$ fiber over the arcs $\gamma_0,\cdots,\gamma_5$ (in orange).} \label{fig:base}
\end{figure}

\begin{example}\label{example:SYZ}
Let $n=1$, $(k_0,\pm l_0)=(1,0)$ and $(k_1,\pm l_1)=(0,1)$, then $\ring{W}_f$ is the plumbing of two $T^\ast S^3$'s along a Hopf link. It can also be realized as the complete intersection $\ring{W}_f\subset\mathbb{C}^\ast\times\mathbb{C}^4$ defined by the equations
\begin{equation}
\left\{\begin{array}{ll}
u_1v_1=z+1,\\
u_2v_2=z-1,
\end{array}\right.
\end{equation}
where $u_1,u_2,v_1,v_2\in\mathbb{C}$ and $z\in\mathbb{C}^\ast$. This log Calabi-Yau $3$-fold is one of the most studied examples in mirror symmetry, see for example \cite{aakl}, Section 11. It carries a Lagrangian torus fibration and the SYZ mirror construction yields $\ring{Y}_f=Y_f\setminus D$, where $Y_f$ is the resolved conifold $\mathcal{O}_{\mathbb{P}^1}(-1)\oplus\mathcal{O}_{\mathbb{P}^1}(-1)$ and $D$ is the pullback of the divisor $\{xy=0\}$ on the conifold $\{uv=(1+x)(1+y)\}\subset\mathbb{C}^4$.
\end{example}

\begin{example}\label{example:double}
Let $n=2$, $(k_0,\pm l_0)=(1,0)$, $(k_1,\pm l_1)=(0,1)$ and $(k_2,\pm l_2)=(1,\pm k)$ for some $k\in\mathbb{Z}_{\geq0}$, then $\ring{W}_f=\ring{W}_k$ is an open dense subset of the double bubble plumbing $W_k$ obtained by removing a smooth fiber from the Morse-Bott fibration $\pi:W_k\rightarrow\mathbb{C}$ described in \cite{sw}, Section 4.1. In this case, $Q_1$ and $Q_2$ are both diffeomorphic to $S^3$. As one of the two (mutually diffeomorphic) Morse-Bott surgeries of $Q_1$ and $Q_2$, $Q_0$ is diffeomorphic to $S^1\times S^2$ when $k=0$ and diffeomorphic to the lens space $L(k,1)$ when $k\geq1$.\footnote{The reader should notice the notation difference with \cite{alp}, where the Lagrangian cores spheres in the plumbing $W_k$ are denoted by $Q_0$ and $Q_1$. Our notation here is compatible with \cite{sw}.}
\end{example}

\subsection{Homological mirror symmetry}\label{section:hms}

We can encode the geometric data of the Morse-Bott fibration $\ring{\pi}:\ring{W}_f\rightarrow\mathbb{C}^\ast$ in the polynomial factorization
\begin{equation}\label{eq:f}
f(x,y)=f_0(x,y)\cdots f_n(x,y),
\end{equation}
where
\begin{equation}\label{eq:f'}
f_i(x,y)=y^{k_i}\pm x^{l_i}.
\end{equation}
This explains the notation $\ring{W}_f$ used in Section \ref{section:bg} for the affine $A_n$ plumbing of $T^\ast Q_i$'s with the prescribed data. By our assumption (\ref{eq:kl}), each $\{f_i(x,y)=0\}\subset\mathbb{K}^2$ defines a smooth affine algebraic curve, so we obtain a $cA_n$ singularity
\begin{equation}\label{eq:Rf}
R_f=\frac{\mathbb{K}[u,v,x,y]}{uv-f(x,y)},
\end{equation}
which admits a crepant resolution (cf. \cite{sks} or Lemma \ref{lemma:resolution} below)
\begin{equation}\label{eq:resol}
\phi:Y_f\rightarrow\mathrm{Spec}(R_f),
\end{equation}
where $Y_f$ is a normal Calabi-Yau variety and the exceptional loci $C:=C_1\cup\cdots\cup C_n$ consist of a chain of transversely intersecting rational curves. Note that the ordering of the polynomials $f_i(x,y)$ in the factorization (\ref{eq:f}) of $f(x,y)$ uniquely determines the specific crepant resolution $\phi$, and different resolutions are related by flops. The $cA_n$ singularity (\ref{eq:Rf}) is in general not isolated. For it to be isolated, one needs to impose the assumption
\begin{equation}\label{eq:distinct}
(k_i,\pm l_i)\neq(k_j,\pm l_j)\textrm{ for }0\leq i\neq j\leq n.
\end{equation}
Symplectically, this is equivalent to requiring that none of the Lagrangian submanifolds $Q_0,\cdots,Q_n$ in the plumbing $\ring{W}_f$, nor the Morse-Bott surgeries (along intersecting circles) among any subcolletions of of them, are diffeomorphic to $S^1\times S^2$.

The wrapped Fukaya category $\mathcal{W}(\ring{W}_f;\mathbb{K})$ of $\ring{W}_f$ can be computed via counting holomorphic sections of the Morse-Bott fibration $\ring{\pi}:\ring{W}_f\rightarrow\mathbb{C}^\ast$, see Section \ref{section:computation}. On the other hand, as we have observed in Example \ref{example:SYZ}, there is a natural choice of a divisor $D\subset Y_f$ given by pulling back the divisor $\{xy=0\}$ on the singular affine variety $\mathrm{Spec}(R_f)$ along $\phi$. Define $\ring{Y}_f:=Y_f\setminus D$ to be the complement. Note that the exceptional curves $C_1,\cdots,C_n\subset Y_f$ are disjoint from $D$, therefore they appear in $\ring{Y}_f$.

The first version of homological mirror symmetry that we shall prove in this paper is the following.

\begin{theorem}\label{theorem:main}
Let $\mathbb{K}$ be any field, and suppose that the pairs $(k_i,\pm l_i)\in\mathbb{Z}_{\geq0}\times\mathbb{Z}$ for $i=0,\cdots,n$ are as in (\ref{eq:kl}). Then there is an equivalence
\begin{equation}\label{eq:HMS-un}
D^\mathit{perf}\mathcal{W}(\ring{W}_f;\mathbb{K})\simeq D^b\mathit{Coh}(\ring{Y}_f)
\end{equation}
between the derived wrapped Fukaya category of $\ring{W}_f$ and the bounded derived category of coherent sheaves on $\ring{Y}_f$.
\end{theorem}

\begin{remark}
In the simplest case when $\ring{W}_f$ is the plumbing of two $T^\ast S^3$'s along a Hopf link as in Example \ref{example:SYZ}, the mirror equivalence (\ref{eq:HMS-un}) is proved by Chan-Pomerleano-Ueda \cite{cpul}. As in \cite{cpul}, we can actually write down an explicit functor inducing the equivalence (\ref{eq:HMS-un}), see Lemma \ref{lemma:functor}.
\end{remark}

Let $\mathcal{F}(\ring{W}_f;\mathbb{K})$ be the Fukaya category of compact exact Lagrangians, which embeds as a full $A_\infty$-subcategory inside the wrapped Fukaya category $\mathcal{W}(\ring{W}_f;\mathbb{K})$, and let $D^b\mathit{Coh}_C(\ring{Y}_f)\subset D^b\mathit{Coh}(\ring{Y}_f)$ be the full subcategory of complexes whose cohomology sheaves are (set-theoretically) supported on the exceptional loci $C=\bigcup_{i=1}^nC_i\subset\ring{Y}_f$. Theorem \ref{theorem:main} implies the following result on the homological mirror symmetry for compact cores.

\begin{corollary}\label{corollary:core-A}
With the same assumptions as in Theorem \ref{theorem:main}, there is a fully faithful embedding
\begin{equation}
D^b\mathit{Coh}_C(\ring{Y}_f)\hookrightarrow D^\mathit{perf}\mathcal{F}(\ring{W}_f;\mathbb{K}),
\end{equation}
whose image is generated by the Lagrangians $Q_0,\cdots,Q_n\subset\ring{W}_f$.
\end{corollary}

\begin{remark}
In \cite{iua}, Theorem 28, it is proved that there is a fully faithful embedding
\begin{equation}\label{eq:2D}
D^b\mathit{Coh}_C(X_n)\hookrightarrow D^\mathit{perf}\mathcal{F}(\ring{A}_n;\mathbb{K})
\end{equation}
from the bounded derived category of coherent sheaves supported on the exceptional curves of the minimal resolution $X_n\rightarrow\mathrm{Spec}\left(\frac{\mathbb{K}[u,v,x]}{uv-x^{n+1}}\right)$ of the $A_n$ surface singularity into the derived Fukaya category of compact Lagrangians in the affine $A_n$ plumbing
\begin{equation}
\ring{A}_n:=\left.\left\{(x,y,z)\in\mathbb{C}^2\times\mathbb{C}^\ast\right\vert xy=z^{n+1}-1\right\}
\end{equation}
of $T^\ast S^2$'s. Corollary \ref{corollary:core-A} above extends this result to dimension $3$.
\end{remark}

To state our second version of the homological mirror symmetry for $\ring{W}_f$, first note that since the fibration $\ring{\pi}:\ring{W}_f\rightarrow\mathbb{C}^\ast$ carries a fiberwise Hamiltonian $T^2$-action, $\mathcal{W}(\ring{W}_f;\mathbb{K})$ is an $A_\infty$-category over the Laurent polynomial ring $\mathbb{K}[s_1^{\pm1},s_2^{\pm1}]$, where the variables $s_1$ and $s_2$ arise as Seidel elements (see Section \ref{section:Seidel}) in the unit group $\mathit{SH}^0(\ring{W}_f;\mathbb{K})^\times\subset\mathit{SH}^0(\ring{W}_f;\mathbb{K})$ of the degree $0$ symplectic cohomology associated to the Hamiltonian $T^2$-action. We denote by $\mathcal{W}(\ring{W}_f;\mathbb{K})[\![s_1+1,s_2+1]\!]$ the completion of $\mathcal{W}(\ring{W}_f;\mathbb{K})$ at $s_1=s_2=-1$. After changing the variables $x=s_1+1$ and $y=s_2+1$, we will simply write
\begin{equation}\label{eq:comp1}
\mathcal{W}(\ring{W}_f;\mathbb{K})[\![x,y]\!]
\end{equation}
for the wrapped Fukaya category completed at $-1$ of the Seidel elements. On the other hand, since $\ring{W}_f$ is a Weinstein manifold, the morphism spaces of the wrapped Fukaya category $\mathcal{W}(\ring{W}_f;\mathbb{K})$ carry filtrations by word-length of the concatenations of Hamiltonian chords. This gives rise to a second completion of $\mathcal{W}(\ring{W}_f;\mathbb{K})$, which is also an $A_\infty$-category
\begin{equation}
\widehat{\mathcal{W}}(\ring{W}_f;\mathbb{K})
\end{equation}
over $\mathbb{K}[\![x,y]\!]$, and can be regarded as a further completion of (\ref{eq:comp1}). Taking a Legendrian surgery perspective, it is not hard to see that $\widehat{\mathcal{W}}(\ring{W}_f;\mathbb{K})$ is equivalently described by the \textit{completed Chekanov-Eliashberg dg algebra} introduced by Ekholm-Lekili \cite{ekle}, see Section \ref{section:CE}. 

On the B-side, we consider the complete local version
\begin{equation}
\widehat{R}_f=\frac{\mathbb{K}[\![u,v,x,y]\!]}{uv-f(x,y)}
\end{equation}
of the $cA_n$ singularity $R_f$, with the crepant resolution
\begin{equation}
\widehat{\phi}:\widehat{Y}_f\rightarrow\mathrm{Spec}(\widehat{R}_f)
\end{equation}
being a Noetherian variety.

\begin{theorem}\label{theorem:HMS1}
With the same assumptions as in Theorem \ref{theorem:main}, there is an equivalence
\begin{equation}\label{eq:HMS1}
D^\mathit{perf}\widehat{\mathcal{W}}(\ring{W}_f;\mathbb{K})\simeq D^b\mathit{Coh}(\widehat{Y}_f),
\end{equation}
where $D^\mathit{perf}\widehat{\mathcal{W}}(\ring{W}_f;\mathbb{K})$ is the derived category of perfect $A_\infty$-modules over $\widehat{\mathcal{W}}(\ring{W}_f;\mathbb{K})$ and $D^b\mathit{Coh}(\widehat{Y}_f)$ is the bounded derived category of coherent sheaves on $\widehat{Y}_f$.
\end{theorem}

\begin{remark}\label{remark:MS}
Note that in the equivalence (\ref{eq:HMS1}), no divisor is removed from $\widehat{Y}_f$, just as no divisor is removed on the left-hand side of the fully faithful embedding (\ref{eq:2D}). This can be explained as follows. The completed wrapped Fukaya category $\widehat{\mathcal{W}}(\ring{W}_f;\mathbb{K})$ is related to the $A_\infty$-subcategory of the compact Fukaya category $\mathcal{F}(\ring{W}_f;\mathbb{K})$ split-generated by $Q_0,\cdots,Q_n$ via $A_\infty$-Koszul duality, see Section \ref{section:CE}. On the B-side, a parallel fact is proved by Ballard, see \cite{mbd}, Theorem 1 and Theorem 3.2, although when applying his result to Noetherian schemes, additional care needs to be taken (cf. \cite{mbd}, Remark 3.8). If one assumes Ballard's result holds for $\widehat{Y}_f$, then Theorem \ref{theorem:HMS1} would give rise to a fully faithful embedding
\begin{equation}
D^b\mathit{Coh}_C(\widehat{Y}_f)\hookrightarrow D^\mathit{perf}\mathcal{F}(\ring{W}_f;\mathbb{K}),
\end{equation}
from the bounded derived category of coherent sheaves supported on the exceptional curves $C_1\cup\cdots\cup C_n\subset\widehat{Y}_f$ to the derived compact Fukaya category.
\end{remark}

We discuss some consequences of Theorem \ref{theorem:HMS1}. Let $D_n\subset\mathbb{C}^\ast$ be the subset consisting of the $n+1$ distinct points $c_0,\cdots,c_n\in\mathbb{C}^\ast$, one can define the relative wrapped Fukaya category $\mathcal{W}(T^\ast S^1,D_n)$, which is an $A_\infty$-category linear over the power series ring $\widehat{R}=\mathbb{K}[\![t_0,\cdots,t_n]\!]$. It is proved in \cite{evle} that after the base change
\begin{equation}
\widehat{R}\rightarrow\mathbb{K}[\![x,y]\!],\textrm{ }t_i\mapsto f_i(x,y),i=0,\cdots,n,
\end{equation}
we have an equivalence
\begin{equation}
D^\mathit{perf}\widehat{\mathcal{W}}(T^\ast S^1,D_n)\otimes_{\widehat{R}}\mathbb{K}[\![x,y]\!]\simeq D^b\mathit{Coh}(\widehat{Y}_f),
\end{equation}
where $\widehat{\mathcal{W}}(T^\ast S^1,D_n)$ is again the completion of $\mathcal{W}(T^\ast S^1,D_n)$ with respect to the word-length filtration on its morphism spaces. Combining with Theorem \ref{theorem:HMS1} we get the equivalence
\begin{equation}
D^\mathit{perf}\widehat{\mathcal{W}}(T^\ast S^1,D_n)\otimes_{\widehat{R}}\mathbb{K}[\![x,y]\!]\simeq D^\mathit{perf}\widehat{\mathcal{W}}(\ring{W}_f;\mathbb{K}).
\end{equation}
In fact, during the process of establishing Theorem \ref{theorem:HMS1}, we will prove a stronger result (see also Proposition \ref{proposition:precise} for its precise statement).

\begin{proposition}\label{proposition:eq}
Assuming (\ref{eq:kl}), there is an equivalence
\begin{equation}\label{eq:eq}
D^\mathit{perf}\mathcal{W}(T^\ast S^1,D_n)\otimes_{\widehat{R}}\mathbb{K}[\![x,y]\!]\simeq D^\mathit{perf}\mathcal{W}(\ring{W}_f;\mathbb{K})[\![x,y]\!],
\end{equation}
where the right-hand side is the derived category of perfect modules over the $A_\infty$-category (\ref{eq:comp1}).
\end{proposition}

In \cite{lese}, the authors studied the Fukaya categories of algebraic torus fibrations (i.e. $(\mathbb{C}^\ast)^n$-fibrations over Liouville manifolds carrying fiberwise Hamiltonian $T^n$-actions) and conjectured the relations between the wrapped Fukaya category of the total space and that of its base. In particular, our Morse-Bott fibration $\ring{\pi}:\ring{W}_f\rightarrow\mathbb{C}^\ast$ is an algebraic torus fibration in the sense of \cite{lese}, and Proposition \ref{proposition:eq} confirms \cite{lese}, Conjecture E\footnote{In \cite{lese}, Conjecture E, Lekili-Segal considered only the simplest case when the algebraic torus fibration has rank $n=1$, but the same heuristic argument can be extended to higher rank cases.} in this case. For other interesting instances of their conjecture, see \cite{lllm}.

After filling in the origin of the base of the Morse-Bott fibration $\ring{\pi}:\ring{W}_f\rightarrow\mathbb{C}^\ast$ with a smooth fiber isomorphic to $(\mathbb{C}^\ast)^2$ (recall that the monodromy of $\ring{\pi}$ around the origin is assumed to be trivial), we get a new Morse-Bott fibration
\begin{equation}
\pi:W_f\rightarrow\mathbb{C},
\end{equation}
see Figure \ref{fig:base1} for a description of its base (after deformation). Equipping $W_f$ with an appropriate symplectic structure $\omega_W$, it is symplectically equivalent to an $A_n$ plumbing of the cotangent bundles $T^\ast Q_1\cdots,T^\ast Q_n$ along circles. Following \cite{ksws}, we shall refer to these Stein $3$-folds as \textit{multi-bubble plumbings}. Localizing on both sides of the equivalence (\ref{eq:HMS1}) with respect to the idempotents associated to the cotangent fiber $L_0$ of $Q_0$ and the structure sheaf $\mathcal{O}_{\widehat{Y}_f}$, respectively, we get the equivalence
\begin{equation}\label{eq:HMS2-}
D^\mathit{perf}\widehat{\mathcal{W}}(W_f;\mathbb{K})\simeq D^b\mathit{Coh}(\widehat{Y}_f)/\langle\mathcal{O}_{\widehat{Y}_f}\rangle,
\end{equation}
where $\widehat{\mathcal{W}}(W_f;\mathbb{K})$ is the completion with respect to the word-length filtration of the wrapped Fukaya category $\mathcal{W}(W_f;\mathbb{K})$, and $D^b\mathit{Coh}(\widehat{Y}_f)/\langle\mathcal{O}_{\widehat{Y}_f}\rangle$ is known as the \textit{relative singularity category} associated to the crepant resolution $\hat{\phi}:\widehat{Y}_f\rightarrow\mathrm{Spec}(\widehat{R}_f)$ in the sense of Kalck-Yang \cite{ky}. More interestingly, the following \textit{uncompleted} version of (\ref{eq:HMS2-}) holds.

\begin{theorem}\label{theorem:WAn}
Assume (\ref{eq:kl}) holds. Suppose in addition that the Lagrangian submanifold $Q_0\subset\ring{W}_f$ is not diffeomorphic to $S^1\times S^2$, or in other words $(k_0,\pm l_0)\neq(k_n,\pm l_n)$, then we have an equivalence
\begin{equation}\label{eq:HMS2}
D^\mathit{perf}\mathcal{W}(W_f;\mathbb{K})\simeq D^b\mathit{Coh}(Y_f)/\langle\mathcal{O}_{Y_f}\rangle.
\end{equation}
between the derived wrapped Fukaya category of the multi-bubble plumbing and the relative singularity category associated to the crepant resolution $\phi:Y_f\rightarrow\mathrm{Spec}(R_f)$.
\end{theorem}

Theorem \ref{theorem:WAn} partially answers the realization question raised by Smith-Wemyss, see \cite{sw}, Section 7. This also generalizes the homological mirror symmetry for the double bubble plumbing $W_k$ proved in \cite{alp} and \cite{sw}, which holds for any $k\geq1$. 

\begin{remark}
The Lagrangian submanifold $Q_0\subset W_f$ is obtained by performing iterated Morse-Bott surgeries of the cleanly intersecting Lagrangian submanifolds $Q_1,\cdots,Q_n$. In this sense, the assumption $(k_0,\pm l_0)\neq(k_n,\pm l_n)$ in Theorem \ref{theorem:WAn} is a natural generalization of the condition that $k\neq0$ in the case of double bubble plumbings $W_k$.
\end{remark}

The reason that an algebraic, rather than complete local version of homological mirror symmetry as in Theorem \ref{theorem:WAn} would hold is that the equivalence
\begin{equation}
D^\mathit{perf}\mathcal{W}(\mathbb{C},D_n)\otimes_R\mathbb{K}[x,y]\simeq D^b\mathit{Coh}(Y_f)/\langle\mathcal{O}_{Y_f}\rangle
\end{equation}
between the derived relative wrapped Fukaya category of $(\mathbb{C},D_n)$ (after base change) and the relative singularity category proved in \cite{evle} already holds before passing to the completion, where $D_n\subset\mathbb{C}$ is the set of points $\{c_0,\cdots,c_n\}$ and $R=\mathbb{K}[t_0,\cdots,t_n]$. In the proof of the equivalence (\ref{eq:HMS2-}), the only place that requires taking the completion is \cite{lese}, Conjecture E (cf. Proposition \ref{proposition:eq} above). However, after passing from $T^\ast S^1$ to $\mathbb{C}$, we have an uncompleted analogue of the equivalence (\ref{eq:eq}), see also Remark \ref{remark:Seidel}.

\begin{proposition}\label{proposition:eq2}
Under the same assumptions as in Theorem \ref{theorem:WAn}, there is an equivalence
\begin{equation}
D^\mathit{perf}\mathcal{W}(\mathbb{C},D_n)\otimes_R\mathbb{K}[x,y]\simeq D^\mathit{perf}\mathcal{W}(W_f;\mathbb{K}).
\end{equation}
\end{proposition}

Just as in the affine case, Theorem \ref{theorem:WAn} would imply a version of homological mirror symmetry for compact cores in the multi-bubble plumbing $W_f$.

\begin{corollary}\label{corollary:core}
Under the same assumption as in Theorem \ref{theorem:WAn}, we have a fully faithful embedding
\begin{equation}\label{eq:core}
D^b\mathit{Coh}_C(Y_f)\hookrightarrow D^\mathit{perf}\mathcal{F}(W_f;\mathbb{K}),
\end{equation}
which is an equivalence if the Lagrangian submanifolds $Q_1,\cdots,Q_n\subset W_f$ split-generate the compact Fukaya category $\mathcal{F}(W_f;\mathbb{K})$.
\end{corollary}

\begin{remark}
It is proved in \cite{sw}, Section 4.6 that $Q_1,Q_2$ split-generate $\mathcal{F}(W_k;\mathbb{K})$ when $k=1$ or $k=p$ is an odd prime and $\mathrm{char}(\mathbb{K})=p$. In general, one can equip the lens spaces $Q_i$ with non-trivial local systems to produce mutually orthogonal objects in the Fukaya category $\mathcal{F}(W_k;\mathbb{K})$, so (\ref{eq:core}) is in general not an equivalence.
\end{remark}

\subsection{Symplectic mapping class group}\label{section:MCG}

The homological mirror symmetry equivalence proved in Theorem \ref{theorem:main} has interesting applications in the study of the symplectic mapping class groups of the Stein $3$-folds $\ring{W}_f$. Let $\mathit{Symp}_c(\ring{W}_f)$ be the group of compactly supported symplectomorphisms, it is proved by Keating-Smith (\cite{kss}, Theorem 1.1) that when $\ring{W}_f$ is the plumbing of two copies of $T^\ast S^3$ along a Hopf link as in Example \ref{example:SYZ}, there is a split-injection
\begin{equation}
\mathbb{Z}^{\ast\infty}\rightarrow\pi_0\mathit{Symp}_c(\ring{W}_f)
\end{equation}
from the free group on countably many generators to the symplectic mapping class group. The group $\mathbb{Z}^{\ast\infty}$ arises as a subgroup of the pure braid group $\mathit{PBr}_3$, which acts on the mapping class group $\pi_0\mathit{Symp}_\mathit{ex}(\ring{W}_f)$, where $\mathit{Symp}_\mathit{ex}(\ring{W}_f)$ is the group of exact symplectomorphisms of the affine $A_1$ plumbing $\ring{W}_f$. In general, there is a map
\begin{equation}
\mathit{PBr}_{n+2}\rightarrow\pi_0\mathit{Symp}_\mathit{ex}(\ring{W}_f)
\end{equation}
as long as the $cA_n$ singularity $R_f$ is isolated, or equivalently (\ref{eq:distinct}) holds. Inside $\mathit{PBr}_{n+2}$ there is a distinguished subgroup $\mathit{PBr}_{n+2}^c$, defined as
\begin{equation}
\mathit{PBr}_{n+2}^c:=\ker\left(\mathit{PBr}_{n+2}\xrightarrow{\varphi}\mathbb{Z}^{n+1}\right),
\end{equation}
where the map $\varphi$ forgets $n$ of the first $n+1$ strands, note that any choice of the $n$ strands gives rise to a map to $\mathit{PBr}_2\cong\mathbb{Z}$. When $n=1$, we have $\mathit{PBr}_3^c\cong\mathbb{Z}^{\ast\infty}$, and $\mathit{PBr}_{n+2}^c$ is always an infinitely generated subgroup of $\mathit{PBr}_{n+2}$ for any $n\geq1$. Generalizing Keating-Smith's result, we prove the following.

\begin{theorem}\label{theorem:MCG}
Assume (\ref{eq:distinct}) holds. There is a split injection
\begin{equation}
\mathit{PBr}_{n+2}^c\rightarrow\pi_0\mathit{Symp}_c(\ring{W}_f).
\end{equation}
In particular, the symplectic mapping class group $\pi_0\mathit{Symp}_c(\ring{W}_f)$ is infinitely generated.
\end{theorem}

Note that even for the $cA_1$ singularities $R_f=\mathbb{K}[u,v,x,y]/\left(uv-x(x+y^k)\right)$ with $k\geq2$, in which case $\ring{W}_f$ is the plumbing of two copies of $T^\ast L(k,1)$ along a Hopf link\footnote{Here by ``Hopf link" we mean the core of the two Heegaard tori in the decomposition of $L(k,1)$.}, our result is new. The proof of Theorem \ref{theorem:MCG} is similar to that of \cite{kss}, Theorem 1.1, which relies on Toda's computation \cite{yt1,yt2} of the autoequivalence group of $D^b\mathit{Coh}_C(\ring{Y}_f)$. Our Theorem \ref{theorem:main} now enables us to exploit Toda's result to its full generality.

By attaching Weinstein $2$-handles to the Stein $3$-folds $\ring{W}_f$, we obtain the following, which differs with the case of compactly supported smooth mapping class groups.

\begin{corollary}
There exist infinitely many $3$-dimensional Stein domains $(\overline{X},\partial\overline{X})$, with both $\overline{X}$ and $\partial\overline{X}$ simply-connected, and $\pi_0\mathit{Symp}_c(\overline{X})$ infinitely generated.
\end{corollary}

On the other hand, if we consider the graded symplectic mapping class group $\pi_0\mathit{Symp}^\mathit{gr}(\ring{W}_f)$, then Toda's result \cite{yt1,yt2} combined with Theorem \ref{theorem:main} implies the following.

\begin{theorem}\label{theorem:MCG1}
Under the assumption (\ref{eq:distinct}), there is a split injection
\begin{equation}
\mathit{PBr}_{n+2}\rightarrow\pi_0\mathit{Symp}^\mathit{gr}(\ring{W}_f).
\end{equation}
\end{theorem}

\begin{remark}
In \cite{evle}, an action of the affine pure braid group $\mathit{PBr}_{n+1}^\mathit{aff}$ is obtained on the relative wrapped Fukaya category $D^\mathit{perf}\mathcal{W}(T^\ast S^1,D_n)$. By Proposition \ref{proposition:eq}, $\mathit{PBr}_{n+1}^\mathit{aff}$ also acts on $D^\mathit{perf}\mathcal{W}(\ring{W}_f;\mathbb{K})$, and such an action factors through the graded symplectic mapping class group $\pi_0\mathit{Symp}^\mathit{gr}(\ring{W}_f)$. Since $\mathit{PBr}_{n+1}^\mathit{aff}$ can be realized as a subgroup of $\mathit{PBr}_{n+2}$, it follows from Theorem \ref{theorem:MCG1} that the representation
\begin{equation}
\mathit{PBr}_{n+1}^\mathit{aff}\rightarrow\pi_0\mathit{Symp}^\mathit{gr}(\ring{W}_f)
\end{equation}
is faithful when (\ref{eq:distinct}) holds. This is mirror to the injectivity of the map
\begin{equation}\label{eq:aff-Br}
\mathit{PBr}_{n+1}^\mathit{aff}\rightarrow\mathrm{Auteq}D^b\mathit{Coh}(\widehat{Y}_f)
\end{equation}
when the $cA_n$ singularity $\widehat{R}_f$ is isolated. In fact, the injectivity of (\ref{eq:aff-Br}), with $\mathit{PBr}_{n+1}^\mathit{aff}$ replaced with the fundamental group of the complement of the complexification of an affine hyperplane arrangement, is expected to be true for all $3$-fold flopping contractions, see \cite{mwa}, Section 1.9. The faithfulness of the affine braid group action on the wrapped Fukaya categories of other Stein $3$-folds has been studied in \cite{gtwc}.
	
If one considers the multi-bubble plumbing $W_f$, similar considerations as above give rise to a map
\begin{equation}
\mathit{PBr}_{n+1}\rightarrow\pi_0\mathit{Symp}^\mathit{gr}(W_f).
\end{equation}
It can be shown that this map is a split injection, see \cite{ksws}, Section 5. On the B-side, the faithfulness of the (generalized) pure braid group action on $D^b\mathit{Coh}(\widehat{Y}_f)$ is established for all $3$-fold flopping contractions by Hirano-Wemyss \cite{hwf}.
\end{remark}

We end the introduction by providing an outline of the paper. In Section \ref{section:pre}, we recall some preliminary facts about plumbings of cotangent bundles and crepant resolutions of $cA_n$ singularities. In Section \ref{section:computation}, we compute the wrapped Fukaya category of $\ring{W}_f$ over any field $\mathbb{K}$. This is then used in Section \ref{section:HMS} to prove the results concerning homological mirror symmetry claimed in Section \ref{section:hms}. In Section \ref{section:Symp} we study the compactly supported symplectic mapping class group of the Stein manifolds $\ring{W}_f$ using homological mirror symmetry, and prove Theorem \ref{theorem:MCG}.

\section*{Acknowledgements}
Many of the results presented in this paper are part of the master thesis of BX \cite{bxh}. We thank Uppsala University for providing the excellent research environment. YL would also like to thank Johan Asplund for the collaboration of \cite{alp} and Yank{\i} Lekili for useful discussions related to \cite{lese}, Conjecture E in 2024.

\section{Preliminaries}\label{section:pre}

We collect in this section some preliminary facts that will be relevant to the main contents of this paper.

\subsection{Plumbings of cotangent bundles}\label{section:plumbing}

Let $Q_1$ and $Q_2$ be two closed $n$-dimensional smooth manifolds which intersect along a smooth $d$-dimensional submanifold $Z=Q_1\cap Q_2$ for some $d<n$. Given any identification
\begin{equation}\label{eq:idn}
\eta:\nu_{Z/Q_1}\xrightarrow{\simeq}\nu_{Z/Q_2}
\end{equation}
between the normal bundles of $Z$, we can construct a Stein manifold $W_\eta=T^\ast Q_1\#_{Z,\eta}T^\ast Q_2$, called the \textit{plumbing} of the cotangent bundles of $Q_1$ and $Q_2$ along $Z$ with respect to the identification $\eta$. More precisely, let $D^\ast Q_i$, $i=1,2$ be the disc cotangent bundles, then by the Weinstein neighborhood theorem there exist open neighborhoods $U(Z)_i$ of $D^\ast Z\subset D^\ast Q_i$, which are symplectomorphic to disc subbundles of $\nu_{Z/Q_i}\times\mathbb{C}$ over $Z$. Using the identification
\begin{equation}
\eta\otimes\sqrt{-1}:U(Z)_1\xrightarrow{\simeq}U(Z)_2,
\end{equation}
we can glue the Weinstein neighborhoods together to obtain a Stein domain $\overline{W}_\eta$. The plumbing $W_\eta$ is the completion of $\overline{W}_\eta$. Repeating this construction, one can construct plumbings of cotangent bundles of cleanly intersecting smooth manifolds according to any plumbing diagram.

When $Q_1$ and $Q_2$ in the above are two closed $3$-manifolds intersecting cleanly along a circle $Z=Q_1\cap Q_2\cong S^1$, then since the normal bundles $\nu_{Z/Q_1}$ and $\nu_{Z/Q_2}$ are trivial, the possible isomorphisms (\ref{eq:idn}) form a torsor for $\mathbb{Z}$, and we may therefore write $\eta=k\in\mathbb{Z}$. 

More specifically, if both $Q_1$ and $Q_2$ are diffeomorphic to $S^3$, it is proved in \cite{sw}, Lemma 2.4 that any one of the two Morse-Bott surgeries $K\subset T^\ast Q_1\#_{Z,k}T^\ast Q_2$ of $Q_1$ and $Q_2$ along $Z$ satisfies $H^\ast(K;\mathbb{Z})\cong H^\ast(S^1\times S^2;\mathbb{Z})$ when $k=0$ and $H^\ast(K;\mathbb{Z})\cong\mathbb{Z}/|k|$ when $k\neq0$. In particular, when $Z$ is an unknot in both $Q_1$ and $Q_2$, then $K\subset W_k$ is diffeomorphic either to $S^1\times S^2$ when $k=0$ or a lens space $L(|k|,1)$ when $k\neq0$. The same argument as in the proof of \cite{sw}, Lemma 2.4 actually extends to the case when $Q_1$ and $Q_2$ are lens spaces, and we have the following.

\begin{proposition}
Let $Q_1$ and $Q_2$ be lens spaces, then their Morse-Bott surgery $K$ of $Q_1$ and $Q_2$ in the plumbing $T^\ast Q_1\#_{Z,k}T^\ast Q_2$ is either a homology $S^1\times S^2$ or a homology lens space.
\end{proposition}
\begin{proof}
For the reader's convenience, we briefly sketch the argument. Using the surgery exact triangle associated to clean intersections \cite{mw1}, we see that in the compact Fukaya category $\mathcal{F}(T^\ast Q_1\#_{Z,k}T^\ast Q_2;\mathbb{Z})$ with integer coefficients, there is a quasi-isomorphism
\begin{equation}
K\simeq\mathit{Cone}(Q_1\xrightarrow{\kappa}Q_2)
\end{equation}
between the Morse-Bott surgery $K$ of $Q_1$ and $Q_2$ and the mapping cone associated to the morphism $\kappa\in\mathit{HF}^1(Q_1,Q_2)$, which is a multiple of the unique degree $1$ generator corresponding to the fundamental class of $Z$. Since the mapping cone is an indecomposable object in $\mathcal{F}(T^\ast Q_1\#_{Z,k}T^\ast Q_2;\mathbb{Z})$, $\kappa\in\mathit{HF}^1(Q_1,Q_2)$ is primitive. Since $K$ is exact, the endomorphism space $\hom_\mathcal{F}(K,K)$ of $\mathit{Cone}(Q_1\xrightarrow{\kappa}Q_2)$ in the Fukaya category $\mathcal{F}(T^\ast Q_1\#_{Z,k}T^\ast Q_2;\mathbb{Z})$ is isomorphic to $H^\ast(K;\mathbb{Z})$, so we only need to compute $\hom_\mathcal{F}(K,K)$, which is determined by the Massey product $\mu^3(e,f,e)$ in a minimal model of the Fukaya $A_\infty$-algebra $\bigoplus_{1\leq i,j\leq2}\mathit{CF}^\ast(Q_i,Q_j)$ over $\mathbb{Z}$, where $e\in\mathit{HF}^1(Q_1,Q_2)$ and $f\in\mathit{HF}^2(Q_2,Q_1)$ are generators. Since $\mu^3(e,f,e)=\lambda\cdot f^\vee$ for some $\lambda\in\mathbb{Z}$ and $f^\vee\in\mathit{HF}^2(Q_1,Q_2)$, the result follows.
\end{proof}

In general, the Morse-Bott surgery $K$ of two lens spaces $Q_1$ and $Q_2$ can fail to be a lens space. However, as we will see blow, this is not going to happen for the cases treated in this paper.

We now restrict ourselves to the specific situation of the Stein $3$-folds $\ring{W}_f$ introduced in the introduction, which is defined as the total space of a Morse-Bott fibration $\ring{\pi}:\ring{W}_f\rightarrow\mathbb{C}^\ast$, with data specified by the factorization (\ref{eq:f}). We equip it with a plurisubhmaronic function $\rho:\ring{W}_f\rightarrow(0,\infty)$ as follows. Start with the function
\begin{equation}\label{eq:pf}
\rho_F(z_1,z_2)=|z_1|^2+|z_2|^2+\frac{1}{|z_1|^2}+\frac{1}{|z_2|^2}
\end{equation}
on a smooth fiber of $\ring{\pi}$, where $z_1,z_2$ are complex coordinates on $(\mathbb{C}^\ast)^2$. One can view the Morse-Bott fibration $\ring{\pi}:\ring{W}_f\rightarrow\mathbb{C}^\ast$ as a fiber connected sums of $n+1$ elementary Morse-Bott fibrations $\pi_i:E_i\rightarrow\mathbb{C}$, $i=0,\cdots,n$, such that each $\pi_i$ has a unique singular fiber over $c_i$. After taking the completion, each $\pi_i$ is equivalent to the standard fibration $\mathbb{C}^2\times\mathbb{C}^\ast\rightarrow\mathbb{C}$ given by $xy$, where $(x,y,z)\in\mathbb{C}^2\times\mathbb{C}^\ast$ are complex coordinates. Since the coordinate $z_1$ corresponds to either $x\neq0$ or $y\neq0$, while $z_2$ corresponds to $z$, it is clear that the function $\rho_F$ can be extended to the singular fiber of $\pi_i$, giving a function on $E_i$. One can then use the biholomorphic coordinate transformations between different $E_i$'s to get a function on $\ring{W}_f$, which we still denote by $\rho_F:\ring{W}_f\rightarrow(0,\infty)$ by abuse of notations. Define the function $\rho_B:\mathbb{C}^\ast\rightarrow(0,\infty)$ on the base to be $2\left(|z_3|^2+\frac{1}{|z_3|^2}\right)$, where $z_3$ is the complex coordinate on $\mathbb{C}^\ast$ and the constant $2$ is added for computational convenience. The plurisubharmonic function $\rho:\ring{W}_f\rightarrow(0,\infty)$ is defined as
\begin{equation}
\rho=\rho_F+\ring{\pi}^\ast\rho_B.
\end{equation}
We equip $\ring{W}_f$ with the K\"{a}hler form $\omega_{\ring{W}}:=-dd^\mathbb{C}\rho$. Let $Z$ be the Liouville vector field that is dual to the $1$-form $\theta_{\ring{W}}=-d^\mathbb{C}\rho$, and denote by $\phi_Z^t$ its time-$t$ flow. Recall that the \textit{skeleton} of $\ring{W}_f$ is the subset $\mathrm{Sk}(\ring{W}_f):=\bigcap_{t>0}\phi_Z^{-t}(\overline{\ring{W}}_f)$, where $\overline{\ring{W}}_f\subset\ring{W}_f$ is the corresponding Liouville domain. Since $\ring{W}_f$ is Weinstein, $\mathrm{Sk}(\ring{W}_f)$ is a Lagrangian subset. To compute it, first observe the following.

\begin{lemma}\label{lemma:parallel}
Symplectic parallel transport is globally defined for paths away from critical values of $\ring{\pi}:\ring{W}_f\rightarrow\mathbb{C}^\ast$, and it is compatible with the Hamiltonian $T^2$-action rotating the fibers $\ring{\pi}$.
\end{lemma}
\begin{proof}
This is similar to \cite{sw}, Lemma 4.1, see also Lemma 4.5 loc. cit. The global parallel transport is obtained by viewing $\ring{W}_f$ as a fiber sum of elementary Morse-Bott fibrations $\pi_i:E_i\rightarrow\mathbb{C}$ and patching together locally defined parallel transport maps, which are $T^2$-equivariant. Here we use the fact that the monodromy of $\ring{\pi}$ around the origin is trivial.
\end{proof} 

The following is a generalization of \cite{kss}, Lemma 3.3.

\begin{lemma}\label{lemma:plumbing}
The Weinstein manifold $\ring{W}_f$ is symplectically equivalent to the affine $A_n$ plumbing of the cotangent bundles $T^\ast Q_0,\cdots,T^\ast Q_n$ along circles.
\end{lemma}
\begin{proof}
The structure of the Morse-Bott fibration $\ring{\pi}:\ring{W}_f\rightarrow\mathbb{C}^\ast$ induces in a natural way the splitting of the Liouville vector field $Z=Z_F+Z_B$ away from the critical loci of $\ring{\pi}$.

When restricted to a smooth fiber $F$ of $\ring{\pi}$, the vector field $Z_F$ is dual to the Liouville $1$-form $i^\ast\theta_{\ring{W}}$, where $i:F\hookrightarrow\ring{W}_f$ is the inclusion of the fiber. By our construction of the K\"{a}hler potential $\rho$, it follows that
\begin{equation}
i^\ast\theta_{\ring{W}}=-i^\ast\left(d^\mathbb{C}\left(|z_1|^2+|z_2|^2+\frac{1}{|z_1|^2}+\frac{1}{|z_2|^2}\right)\right).
\end{equation}
Thus $Z_F$ vanishes whenever $|z_1|=|z_2|=1$, and its flow scales $\log|z_i|$ over positive time for $i=1,2$.

Symplectic parallel transport is well-defined by Lemma \ref{lemma:parallel}, and it is compatible with the fiberwise Hamiltonian $T^2$-action on $\ring{\pi}:\ring{W}_f\rightarrow\mathbb{C}^\ast$, preserving $|z_1|$ and $|z_2|$. In particular, any point in $\ring{W}_f$ which does not satisfy $|z_1|=|z_2|=1$ flows off to infinity in the fiber direction, so the Lagrangian skeleton $\mathrm{Sk}(\ring{W}_f)$ is contained in the subset $\left\{|z_1|=|z_2|=1\right\}\subset\ring{W}_f$.

To determine $\mathrm{Sk}(\ring{W}_f)\subset\ring{W}_f$, it remains to analyze the Liouville flow on $\left\{|z_1|=|z_2|=1\right\}$. $T^2$-equivariance implies that the symplectic form $\omega_{\ring{W}}$ restricted to $\left\{|z_1|=|z_2|=1\right\}$ is the pullback of a symplectic form $\omega_B$ on the base $\mathbb{C}^\ast$. By construction, the K\"{a}hler potential $\rho$ restricted to $\left\{|z_1|=|z_2|=1\right\}$ is given by
\begin{equation}\label{eq:potential}
\rho(z_3)=2\sum_{i=0}^n|z_3-c_i|+2\left(|z_3|^2+\frac{1}{|z_3|^2}\right).
\end{equation}
On $\left\{|z_1|=|z_2|=1\right\}$, $Z=Z_B$ is the gradient vector field of $\rho(z_3)$ with respect to the restriction of $\omega_{\ring{W}}$ to $\left\{|z_1|=|z_2|=1\right\}$. This is the same as the pullback of the gradient flow of $\rho(z_3)$ on $\mathbb{C}^\ast$ with respect to the metric associated to $\omega_B$. See Figure \ref{fig:gradient} for a plot of the vector field $\nabla\rho(z_3)$ when $n=3$.

Recall that the $c_i$'s are $n+1$ roots of unity. Direct computation shows that the critical points of (\ref{eq:potential}) satisfy the equation
\begin{equation}\label{eq:crit}
\sum_{i=0}^n\frac{z_3-c_i}{|z_3-c_i|}=2z_3\left(\frac{1}{|z_3|^4}-1\right),
\end{equation}
from which it is clear that the set critical points of $\rho(z_3)$ is symmetric under the rotation by multiples of $\frac{2\pi}{n+1}$. Although the closed forms of the critical points are not available, by analyzing the equations (\ref{eq:potential}) and (\ref{eq:crit}), one can show that there exist $n+1$ saddle points $s_0,\cdots,s_n\in\mathbb{C}^\ast$ with angular coordinates $\arg(c_i)<\arg(s_i)<\arg(c_{i+1\textrm{ mod }n+1})$ and modulus $\frac{1}{2}<|s_i|<1$ for $i=0,\cdots,n$. After a small perturbation of $\omega_B$, there are unique gradient flow lines from $s_i$ to $c_i$ and $c_{i+1}$, whose concatenation $\alpha_i\subset\mathbb{C}^\ast$ gives the matching path for the Lagrangian core $Q_i$. This proves that $\mathrm{Sk}(\ring{W}_f)=\bigcup_{i=0}^nQ_i$, therefore $\ring{W}_f$ is the plumbing of $T^\ast Q_0,\cdots,T^\ast Q_n$ along circles.
\end{proof}

\begin{figure}
	\centering
	\includegraphics[width=0.7\linewidth]{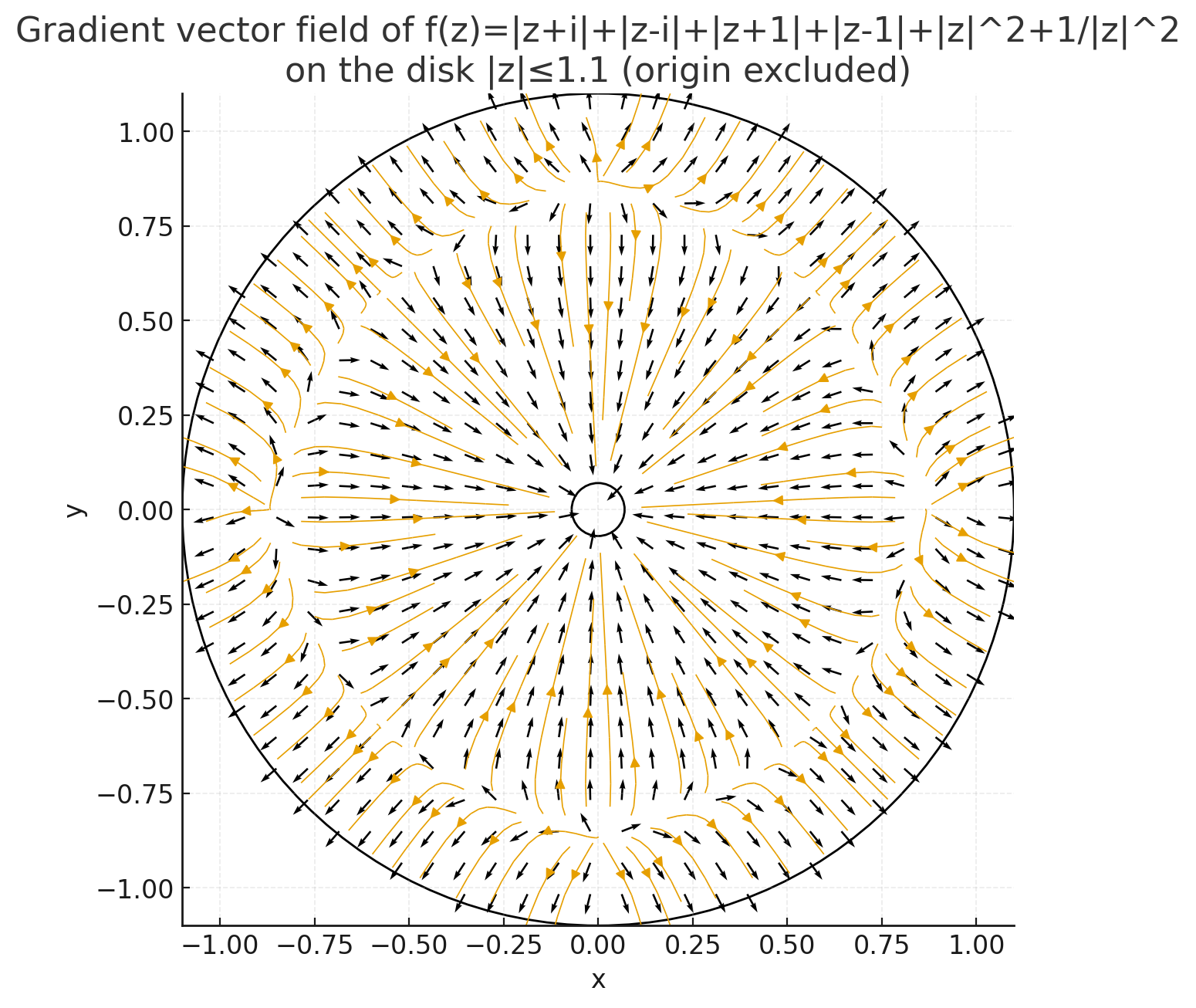}
	\caption{Gradient vector field and flow lines (in orange) of the function (\ref{eq:potential}) when $n=3$ (picture created by ChatGPT)}
	\label{fig:gradient}
\end{figure}

Let $Q_i,Q_{i+1\textrm{ mod }n+1}\subset\ring{W}_f$ be two adjacent Lagrangian lens spaces in the affine $A_n$ plumbing. In this specific setting, up to Hamiltonian isotopy the Morse-Bott surgery $K\subset\ring{W}_f$ is one of the two (non-isotopic) matching cycles associated to the paths connecting the critical values $c_i$ to $c_{i+2\textrm{ mod }n+1}$ of $\ring{\pi}$, which is necessarily diffeomorphic to $L(p,q)$ for some coprime $(p,q)\in\mathbb{Z}_{\geq0}^2$ (recall that $S^1\times S^2=L(0,1)$). The same is true for Morse-Bott surgeries of more than two Lagrangian matching cycles.

\begin{remark}
The circle $Z_i=Q_i\cap Q_{i+1}$ is in general not an unknot in $Q_i$ (or $Q_{i+1}$), in the sense that it does not bound an embedded disc. It is an unknot in $Q_i$ (or $Q_{i+1}$), however, when $Q_i$ (or $Q_{i+1}$) is diffeomorphic to $S^3$.
\end{remark}

Consider the partial compactification $W_f$ of $\ring{W}_f$ by filling in a smooth fiber $(\mathbb{C}^\ast)^2$ above the origin of the Morse-Bott fibration $\ring{\pi}:\ring{W}_f\rightarrow\mathbb{C}$. Without changing the complex structure on $W_f$, we can deform the Morse-Bott fibration $\pi:W_f\rightarrow\mathbb{C}$ obtained by filling in the origin of $\ring{\pi}$ so that the new critical values $c_0',\cdots,c_n'$ now lie on the real line. Equip $W_f$ with a symplectic form $\omega_W$ coming from the K\"{a}hler potential $\rho'=\rho_F+\pi^\ast\rho_\mathbb{C}$, where $\rho_F:W_f\rightarrow(0,\infty)$ is the fiberwise potential defined in the same way as above, which takes the form (\ref{eq:pf}) when restricted to a smooth fiber of $\pi:W_f\rightarrow\mathbb{C}$, and $\rho_\mathbb{C}=2|z_3|^2$, we see that $W_f$ is a Weinstein manifold. Similar argument as in the proof of Lemma \ref{lemma:plumbing} implies the following.

\begin{lemma}\label{lemma:plumbing1}
The Weinstein manifold $W_f$ is symplectically equivalent to the $A_n$ plumbing of the cotangent bundles $T^\ast Q_1,\cdots,T^\ast Q_n$ along circles.
\end{lemma}
\begin{proof}
Running the same argument as in the proof of Lemma \ref{lemma:plumbing}, we are reduced to studying the restriction of the K\"{a}hler potential $\rho'=\rho_F+\pi^\ast\rho_\mathbb{C}$ to the subset $\{|z_1|=|z_2|=1\}\subset W_f$, where $z_1$ and $z_2$ are complex coordinates on the fiber of $\pi:W_f\rightarrow\mathbb{C}$, which is given explicitly by
\begin{equation}
\rho'(z_3)=2\sum_{i=0}^n|z_3-c_i'|+2|z_3|^2.
\end{equation}
In this case, the gradient flow lines of $\rho'$ shrink to line segments connecting $c_0',\cdots,c_n'\in\mathbb{R}$ when $t\rightarrow-\infty$.
\end{proof}

\subsection{Crepant resolutions of $cA_n$ singularities}\label{section:resolution}

In this subsection we work in the algebraic setting and consider the crepant resolution of the $cA_n$ singularity
\begin{equation}
R_f=\frac{\mathbb{K}[u,v,x,y]}{uv-f(x,y)},
\end{equation}
where $f(x,y)=f_0(x,y)\cdots f_n(x,y)$ is defined as in (\ref{eq:f'}). We assume $\mathbb{K}$ to be an arbitrary field. When $\mathrm{char}(\mathbb{K})=0$, everything discussed here is well-known, see \cite{sks}. The story in the complete local case is similar.

\begin{lemma}\label{lemma:resolution}
Assume that (\ref{eq:distinct}) holds. The affine variety $\mathrm{Spec}(R_f)$ admits a specific crepant resolution $\phi:Y_f\rightarrow\mathrm{Spec}(R_f)$ containing a chain of rational curves $C_1,\cdots,C_n$ such that $C_i\cap C_{i+1}$ transversely at a unique point for $i=0,\cdots,n-1$ and $C_i\cap C_j=\emptyset$ if $|i-j|>1$. Moreover, each $C_i$ is either a $(-1,-1)$-curve or a $(0,-2)$-curve.
\end{lemma}
\begin{proof}
One can argue similarly as in \cite{sw}, Lemma 3.1. In our case, however, we shall use in addition the fact that
\begin{equation}\label{eq:cA1}
\varphi_k:Z_k\rightarrow\mathrm{Spec}\left(\frac{\mathbb{K}[u,v,x,y]}{uv-x(x+y^k)}\right),
\end{equation}	
which is the blow-up of the ideal $(u,x)$, is a projective birational morphism with a single $(-2,0)$-curve as the exceptional locus when $k>1$, regardless of $\mathrm{char}(\mathbb{K})$. This is proved in the same way as in the case $k=1$, which is done in \cite{sw}, Lemme 3.1, by applying Zariski's main theorem and writing down the \v{C}ech complex for $\mathcal{O}_{Z_k}$, since the partial resolution $\varphi_k$ is obtained by blowing up the same ideal.
	
Without loss of generality, we can assume $n\geq2$. To construct the specific crepant resolution $\phi$, we first blow up the ideal $(u,f_0)$ to obtain $\phi_1:Y_1\rightarrow\mathrm{Spec}(R_f)$, where the variety $Y_1$ is covered by two open charts
\begin{equation}
\mathrm{Spec}\left(\mathbb{K}[u,x,S]\right)\textrm{ and }\mathrm{Spec}\left(\frac{\mathbb{K}[v,x,y,T]}{\left(vT-f_1(x,y)\cdots f_n(x,y)\right)}\right),
\end{equation}
where $S=\frac{f_0}{u}$ and $T=\frac{u}{f_0}$ patch to define a $\mathbb{P}^1\subset Y_1$. Since $\mathrm{Spec}(R_f)$ is normal, Zariski's main theorem combined with \v{C}ech cohomology calculations show that the morphism $\phi_1$ is crepant, i.e. $R(\phi_1)_\ast\mathcal{O}_{Y_1}=\mathcal{O}_{\mathrm{Spec}(R_f)}$.

One then blows up the ideals $(u,f_i)$ for $i=1,\cdots,{n-1}$ to get a sequence of birational morphisms
\begin{equation}
Y_f\xrightarrow{\phi_n}Y_{n-1}\xrightarrow{\phi_{n-1}}\cdots\xrightarrow{\phi_1}\mathrm{Spec}(R_f),
\end{equation}
whose composition $\phi=\phi_n\circ\cdots\circ\phi_1:Y_f\rightarrow\mathrm{Spec}(R_f)$ gives the desired crepant resolution of the $cA_n$ singularity $R_f$. Note that the crepancy of $\phi$ follows from the crepancy of each individual birational morphism $\phi_i$, where $1\leq i\leq n$, which is proved using Zariski's main theorem and \v{C}ech cohomology calculations.

The partial resolution $Y_{n-1}$ is covered by two affine charts 
\begin{equation}
\mathrm{Spec}\left(\mathbb{K}[u,x,S']\right)\textrm{ and }
\mathrm{Spec}\left(\frac{\mathbb{K}[v,x,y,T']}{vT'-f_{n-1}(x,y)f_n(x,y)}\right),
\end{equation}
where $S'=\frac{f_{n-2}}{u}$ and $T'=\frac{u}{f_{n-2}}$. In $Y_{n-1}$, the coordinates $S'$ and $T'$ patch together to an exceptional curve $\mathbb{P}^1\subset Y_{n-1}$, which passes through the node $\left\{vT=f_{n-1}(x,y)f_n(x,y)\right\}$ in the second chart. Arguing by induction, we see that in the full resolution $Y_f$ we have a chain of transversely intersecting $\mathbb{P}^1$'s, which we denote by $C_1,\cdots,C_n$.

To find the normal bundles of $C_1,\cdots,C_n\subset Y_f$, note that the last of which, $C_n$, comes from resolution of the $cA_1$ singularity $\frac{\mathbb{K}[v,x,y,T']}{vT'-f_{n-1}(x,y)f_n(x,y)}$, which is isolated by our assumption. Since any isolated $cA_1$ singularity is equivalent to $\frac{\mathbb{K}[u,v,x,y]}{uv-x(x+y^k)}$ for some $k\geq1$, $C_n$ is a $(-1,-1)$-curve when $k=1$, and a $(0,-2)$-curve when $k>1$. Contracting $C_n$ results in a $cA_{n-1}$ singularity $\frac{\mathbb{K}[u,v,x,y]}{(uv-f_1(x,y)\cdots f_{n-1}(x,y))}$. We can then argue inductively to determine the normal bundles of $C_1,\cdots,C_{n-1}$.
\end{proof}

\begin{remark}
Taking into account of the data on the mirror side, one can prove that the exceptional curve $C_i\subset Y_f$ is a $(-1,-1)$-curve if the corresponding spherical Lagrangian submanifold $Q_i\subset\ring{W}_f$ is diffeomorphic to $S^3$, otherwise it is a $(0,-2)$-curve. 

For simplicity and for our applications in mind, we have assumed in the above that the $cA_n$ singularity $R_f$ is isolated. When it is not isolated, since in our case the singularity $R_f$ is of a specific form, a similar argument as above still works and produces a crepant resolution $Y_f\rightarrow\mathrm{Spec}(R_f)$ (cf. the $n=0$ case of \cite{sw}, Lemma 3.1).
\end{remark}

\subsection{Seidel representation}\label{section:Seidel}

Let $\mathbb{K}$ be an arbitrary field. For a closed symplectic manifold $X$ satisfying certain monotonicity conditions, Seidel defined in \cite{pspi} a group homomorphism
\begin{equation}
\mathcal{R}_X:\pi_1\left(\mathit{Ham}(X)\right)\rightarrow\mathit{QH}^0(X;\mathbb{K})^\times
\end{equation}
from the fundamental group of the space Hamiltonian diffeomorphisms of $X$ to invertible elements in even degree quantum cohomology. The map $\mathcal{R}_X$ is known as the \textit{Seidel representation}, and it defined via counting holomorphic sections of a symplectic fiber bundle $E\rightarrow S^2$ constructed by patching together two copies of trivial bundles $D^2\times X$, whose transition function is determined by some fixed element $g\in\pi_1\left(\mathit{Ham}(X)\right)$. An element $\mathcal{R}_X(g)\in\mathit{QH}^0(X;\mathbb{K})^\times$ is known as a \textit{Seidel element} associated to $g$.

There is an analogue of the map $\mathcal{R}_X$ above for Liouville manifolds, whose construction is due to Oh-Tanaka \cite{otc}, which takes the form
\begin{equation}\label{eq:SeiRep}
\mathcal{R}_M:\pi_1\left(\mathit{Aut}(M)\right)\rightarrow\mathit{SH}^0(M;\mathbb{K})^\times,
\end{equation}
where $\mathit{Aut}(M)$ is the topological group of graded Liouville automorphisms of $M$. Instead of mimicking Seidel's original construction, \cite{otc} produces a coherent action of $\mathit{Aut}(M)$ on the wrapped Fukaya category $\mathcal{W}(M;\mathbb{K})$ by studying the universal Liouville bundle $E\mathit{Aut}(M)\rightarrow B\mathit{Aut}(M)$, therefore simultaneously produces higher homotopical analogues of the map $\mathcal{R}_M$. As above, associated to each $g\in\pi_1\left(\mathit{Aut}(M)\right)$ there is a Seidel element $\mathcal{R}_M(g)\in\mathit{SH}^0(M;\mathbb{K})^\times$.

In general, if the Liouville manifold $M$ carries a Hamiltonian $S^1$-action $g$, the action will not be compatible with the Reeb flow on the ideal boundary $\partial_\infty M$, therefore the orbit of this action does not automatically give the Seidel element $\mathcal{R}(g)$. However, there are some exceptions. For example, when $M$ is not Liouville, but some non-exact convex symplectic manifolds like the total space of negative line bundles, the Hamiltonian $S^1$-action rotating the fibers of the line bundle is compatible with the Reeb flow on circle bundle, and the Seidel element is represented by a Hamiltonian $S^1$-orbit, see \cite{arf}. The situation treated in this paper is another example, where the fiberwise Hamiltonian circle actions of $\ring{\pi}:\ring{W}_f\rightarrow\mathbb{C}^\ast$ give rise to orbits of a particular wrapping Hamiltonian $H_W:\ring{W}_f\rightarrow[0,\infty)$ (see Section \ref{section:setup} below) that is adapted to the Morse-Bott fibration $\ring{\pi}$. Although $H_W$ differs from the quadratic Hamiltonians that are usually used to define the symplectic cohomology $\mathit{SH}^\ast(\ring{W}_f;\mathbb{K})$, its cofinality (cf. Lemma \ref{lemma:cofinal} in the open string case, for the closed string case, see \cite{jzp}, Lemma 5.3.7 for a similar situation) would imply that the Hamiltonian Floer cohomology of $H_W$ coincides with $\mathit{SH}^\ast(\ring{W}_f;\mathbb{K})$. In this paper, we will actually use an alternative argument, see Lemma \ref{lemma:vertical}.

\subsection{Legendrian surgery}\label{section:CE}

The completion of the wrapped Fukaya category and the general fact that the $A_\infty$-category of compact cores in a Weinstein manifold can be recovered from the wrapped Fukaya category are used in several places in this paper. These are originally considered from a Legendrian surgery perspective by Ekholm-Lekili \cite{ekle}. For the reader's convenience, we briefly recall some of their results in this subsection. We will work in the slightly more general set up when the skeleton consists of Lagrangian submanifolds intersecting with each other cleanly, although transverse intersection is assumed throughout in \cite{ekle}.

Recall that any $2n$-dimensional Weinstein manifold $W$ can be constructed from a subcritical Weinstein domain $\overline{W}_0$, by attaching Weinstein handles of index $n$ at the contact boundary $\partial\overline{W}_0$ along a link of Legendrian spheres. Fix a finite set $\Gamma$. For each $v\in\Gamma$, let $\overline{Q}_v\subset\overline{W}_0$ be an oriented connected Lagrangian submanifold whose boundary $\partial\overline{Q}_v\subset\partial\overline{W}_0$ is a Legendrian sphere $\Lambda_v$. For simplicity, we also assume that $\overline{Q}_v$ is \textit{Spin} for each $v\in\Gamma$, which is always true for our purposes since in this paper we are dealing with Weinstein $6$-manifolds. Moreover, we assume that the Lagrangian submanifolds $\{\overline{Q}_v\}_v\in\Gamma$ intersect with each other cleanly, and the intersections are contained in the interior of $\overline{W}_0$. It follows that $\Lambda=\bigsqcup_{v\in\gamma}\Lambda_v$ is a link of Legendrian spheres in $\partial\overline{W}_0$. Attaching $n$-handles to $\overline{W}_0$ along $\Lambda$, we get a Weinstein domain $\overline{W}$. By the works \cite{beee,cdrgg,ekle,gps}, it is now well-known that the wrapped Fukaya category $\mathcal{W}(W;\mathbb{K})$ can be computed through Legendrian surgery. In fact, let $\mathcal{R}$ be the set of Reeb chords with ends on $\Lambda$. As a $\mathbb{K}$-module, the Chekanov-Eliashberg dg algebra of $\Lambda\subset\partial\overline{W}_0$ is given by
\begin{equation}
\mathit{CE}^\ast(\Lambda):=\bigoplus_{i=0}^\infty\mathbb{K}\langle\mathcal{R}\rangle^{\otimes i}.
\end{equation}
The differential on $\mathit{CE}^\ast(\Lambda)$ counts anchored holomorphic disks with boundary punctures in the symplectization $\mathbb{R}\times\partial\overline{W}_0$. The boundary components of these holomorphic curves are mapped to $\mathbb{R}\times\Lambda$, and the asymptotics at the punctures are given by the Reeb chords in $\mathcal{R}$. See \cite{beee}, Section 4.1. It is clear that $\mathit{CE}^\ast(\Lambda)$ is a dg algebra over the semisimple ring $\Bbbk:=\bigoplus_{v\in\Gamma}\mathbb{K}e_v$, with $e_w\mathcal{R}e_v$ being the set of Reeb chords from starting $\Lambda_w$ and ending at $\Lambda_v$. The Lagrangian filling $\bigcup_{v\in\Gamma}\overline{Q}_v\subset\overline{W}_0$ of the Legendrian link $\Lambda$ induces an augmentation
\begin{equation}
\varepsilon_Q:\mathit{CE}^\ast(\Lambda)\rightarrow\Bbbk,
\end{equation}
making the Chekanov-Eliashberg dg algebra an augmented dg algebra.

Let $\{L_v\}_{v\in\Gamma}$ be the Lagrangian cocore discs of the Weinstein $n$-handles attached along $\Lambda\subset\partial\overline{W}_0$, which generate the wrapped Fukaya category $\mathcal{W}(W;\mathbb{K})$ \cite{cdrgg,gps}, so there is an equivalence
\begin{equation}\label{eq:generation}
D^\mathit{perf}\mathcal{W}(W;\mathbb{K})\simeq D^\mathit{perf}(\mathcal{W}_W)
\end{equation}
between the derived wrapped Fukaya category of $W$ and the derived category of perfect modules over the wrapped Fukaya $A_\infty$-algebra
\begin{equation}
\mathcal{W}_W:=\bigoplus_{v,w\in\Gamma}\mathit{CW}^\ast(L_v,L_w).
\end{equation}
By \cite{beee,ekle}, there is a surgery quasi-isomorphism
\begin{equation}\label{eq:surgery}
\mathcal{W}_W\simeq\mathit{CE}^\ast(\Lambda)
\end{equation}
between $A_\infty$-algebras over $\Bbbk$, which reduces the computation of the wrapped Fukaya category to the analysis of Morse flow trees \cite{tem}.

Equivalently, we can write $\mathit{CE}^\ast(\Lambda)$ as
\begin{equation}
\mathit{CE}^\ast(\Lambda)=\Bbbk\oplus\bigoplus_{i=1}^\infty\overline{\mathit{LC}}_\ast(\Lambda)[-1]^{\otimes_\Bbbk i},
\end{equation}
where by $\mathit{LC}_\ast(\Lambda)$ we mean the graded $\Bbbk$-bimodule generated by $\mathcal{R}$, with grading shifted down by $1$, and $\overline{\mathit{LC}}_\ast(\Lambda)\subset\mathit{LC}_\ast(\Lambda)$ is the submodule without the idempotents. Using the same holomorphic curves as in the definition of the differential on $\mathit{CE}^\ast(\Lambda)$ and the augmentation $\varepsilon_Q$, one can equip $\mathit{LC}_\ast(\Lambda)$ with the structure of an $A_\infty$-coalgebra. Now the dg algebra $\mathit{CE}^\ast(\Lambda)$ can be written as a cobar construction
\begin{equation}\label{eq:cobar}
\mathit{CE}^\ast(\Lambda)=\Omega\mathit{LC}_\ast(\Lambda).
\end{equation}
Define the \textit{Legendrian $A_\infty$-algebra} of $\Lambda$ to be the graded $\Bbbk$-linear dual
\begin{equation}
\mathit{LA}^\ast(\Lambda):=\mathit{LC}_\ast(\Lambda)^\#.
\end{equation}
Let $Q_v\subset W$ be the closed exact Lagrangian submanifold obtained as the union of $\overline{Q}_v$ and the Lagrangian core discs of the corresponding Weinstein $n$-handle. Consider the Fukaya $A_\infty$-algebra
\begin{equation}
\mathcal{Q}_W:=\bigoplus_{v,w\in\Gamma}\mathit{CF}^\ast(Q_v,Q_w),
\end{equation}
which is also augmented by the trivial projection to the idempotents.

\begin{proposition}\label{proposition:qis1}
There is a quasi-isomorphism $\mathcal{Q}_W\simeq\mathit{LA}^\ast(\Lambda)$ between $A_\infty$-algebras over $\Bbbk$.
\end{proposition}
\begin{proof}
This is essentially the same as \cite{ekle}, Theorem 63, except that in our case the Lagrangian submanifolds $\{Q_v\}_{v\in\Gamma}$ are allowed to have clean, instead of transverse intersections. When the intersections are transverse, an $A_\infty$-morphism $\varphi:\mathcal{Q}_W\rightarrow\mathit{LA}^\ast(\Lambda)$ can be defined by counting holomorphic discs with $m+1$ boundary punctures so that the boundary components are mapped to parallel copies of $\bigcup_{v\in\Gamma}\overline{Q}_v$, such that the $m$ positive punctures are asymptotic to Lagrangian intersection points, and the unique negative puncture is asymptotic to a Reeb chord in $\mathcal{R}$. Recall that all of the intersections among the Lagrangian submanifolds $\{Q_v\}_{v\in\Gamma}$ appear already in the subcritical Weinstein domain $\overline{W}_0$. In our case, to define the map $\varphi:\mathcal{Q}_W\rightarrow\mathit{LA}^\ast(\Lambda)$, one can either perturb the $Q_v$'s so that the intersections become transverse (which does not affect $\mathit{LA}^\ast(\Lambda)$ up to quasi-isomorphism), or one can work directly with the Morse-Bott model, with punctured holomorphic discs replaced with \textit{punctured pearly trees}, see Figure \ref{fig:pearl}. These are similar to the pearly trees of holomorphic discs considered in \cite{psh,nso}, except that at the output we have a single puncture asymptotic to the Reeb chords in $\mathcal{R}$. The fact that $\varphi$ is a quasi-isomorphism follows from the long exact sequence for the wrapped Floer cohomologies $\mathit{HW}^\ast(\overline{Q}_v,\overline{Q}_w)$ (computed in $\overline{W}_0$) induced by the subdivision of the wrapped Floer cochain complex into the high energy part and the low energy part.
\end{proof}

\begin{figure}
\centering
\begin{tikzpicture}
\draw (0,0.5)--(1.5,0.5);
\draw (0,-0.5)--(1.5,-0.5);
\draw (1.5,-0.5) arc [start angle=-90, end angle=90, radius=0.5];
\draw [decoration={markings, mark=at position 0.5 with {\arrow{<}}},postaction={decorate}] (1.933,0.25)--(2.799,0.75);
\draw [decoration={markings, mark=at position 0.5 with {\arrow{<}}},postaction={decorate}] (1.933,-0.25)--(3.232,-1);
\draw (3.232,1) circle [radius=0.5];
\draw (3.665,-1.25) circle [radius=0.5];
\draw [decoration={markings, mark=at position 0.5 with {\arrow{<}}},postaction={decorate}] (3.665,1.25)--(4.531,1.75);
\draw (4.964,2) circle [radius=0.5];
\draw [decoration={markings, mark=at position 0.5 with {\arrow{<}}},postaction={decorate}] (5.464,2)--(6.464,2);
\draw (6.464,2) node[circle,fill,inner sep=1pt] {};
\draw [decoration={markings, mark=at position 0.5 with {\arrow{<}}},postaction={decorate}] (3.665,0.75)--(6.464,0.25);
\draw (6.464,0.25) node[circle,fill,inner sep=1pt] {};
\draw [decoration={markings, mark=at position 0.5 with {\arrow{<}}},postaction={decorate}] (4.098,-1)--(6.464,-0.5);
\draw (6.464,-0.5) node[circle,fill,inner sep=1pt] {};
\draw [decoration={markings, mark=at position 0.5 with {\arrow{<}}},postaction={decorate}] (4.098,-1.5)--(6.464,-2);
\draw (6.464,-2) node[circle,fill,inner sep=1pt] {};
\draw (3.7,-2.2) node {$\overline{Q}_{v_0}$};
\draw (5.3,-1.25) node {$\overline{Q}_{v_1}$};
\draw (4,0) node {$\overline{Q}_{v_2}$};
\draw (5.5,1.2) node {$\overline{Q}_{v_3}$};
\draw (3.3,2) node {$\overline{Q}_{v_4}$};
\draw (-0.8,0) node {$\Lambda_{v_0}\sqcup\Lambda_{v_4}$};
\end{tikzpicture}
\caption{A punctured pearly tree contributing to the map $\varphi:\mathcal{Q}_W\rightarrow\mathit{LA}^\ast(\Lambda)$, where $m=4$. The Morse flow lines are connected by the holomorphic discs with boundaries on the (clean) intersections of the Lagrangian submanifolds $\overline{Q}_{v_i}$, the inputs on the right-hand side are Morse critical points, while the unique output on the right-hand side is a puncture with asymptotic conditions given by Reeb chords on $\Lambda_{v_0}\sqcup\Lambda_{v_4}$.} \label{fig:pearl}
\end{figure}

Combining (\ref{eq:surgery}), (\ref{eq:cobar}) and Proposition \ref{proposition:qis1}, we obtain the \textit{generalized Eilenberg-Moore equivalence}
\begin{equation}\label{eq:GEM}
R\hom_{\mathcal{W}_W}(\Bbbk,\Bbbk)\simeq\mathcal{Q}_W,
\end{equation}
where in the above $\Bbbk$ is regarded as a right $\mathcal{W}_W$-module.

However, it is in general not true that the bar construction $\mathrm{B}\mathit{LA}^\ast(\Lambda)^\#$ of the graded linear dual of the Legendrian $A_\infty$-algebra of $\Lambda$ gives the Chekanov-Eliashberg dg algebra. Instead, we get from this the completed tensor algebra
\begin{equation}
\widehat{\mathit{CE}}^\ast(\Lambda):=\Bbbk\left\langle\!\left\langle\overline{\mathit{LC}}_\ast(\Lambda)[-1]\right\rangle\!\right\rangle,
\end{equation}
which we call the \textit{completed Chekanov-Eliashberg dg algebra}. In other words, since $\mathit{CE}^\ast(\Lambda)$ is generated as a $\Bbbk$-module by words of Reeb chords in $\mathcal{R}$, $\widehat{\mathit{CE}}^\ast(\Lambda)$ is the completion of $\mathit{CE}^\ast(\Lambda)$ with respect to the word length filtration. By the surgery quasi-isomorphism (\ref{eq:surgery}) and the generation result (\ref{eq:generation}), we obtain an equivalence
\begin{equation}
D^\mathit{perf}\widehat{\mathcal{W}}(W;\mathbb{K})\simeq D^\mathit{perf}\left(\widehat{\mathit{CE}}^\ast(\Lambda)\right)
\end{equation}
between the completed derived wrapped Fukaya category and the derived category of perfect modules over the completed Chekanov-Eliashberg dg algebra.

\section{Computing the wrapped Fukaya category}\label{section:computation}

In this section, we compute the wrapped Fukaya category $\mathcal{W}(\ring{W}_f;\mathbb{K})$ over any field $\mathbb{K}$. The computation actually works over $\mathbb{Z}$ (cf. \cite{alp}, Remark 2.15).

\subsection{The setup}\label{section:setup}

Our computational strategy is similar to that of \cite{alp} in the case of double bubble plumbings, making use of the Morse-Bott fibration $\ring{\pi}:\ring{W}_f\rightarrow\mathbb{C}^\ast$ and counting its holomorphic sections with boundaries on certain Lagrangian submanifolds with are ``admissible" with respect to $\ring{\pi}$. 

We first specify the objects in the wrapped Fukaya category $\mathcal{W}(\ring{W}_f;\mathbb{K})$ to work with. Consider the arcs $\gamma_0,\cdots,\gamma_n\subset\mathbb{C}^\ast$ in Figure \ref{fig:base}, they intersect the unit circle at the points
\begin{equation}
\zeta_i:=\exp\left(\frac{(4i+1)\pi}{2n+2}\sqrt{-1}\right),\textrm{ }i=0,\cdots,n.
\end{equation}
Since $\ring{\pi}^{-1}(\zeta_i)$ is isomorphic to $(\mathbb{C}^\ast)^2$, we can pick the Lagrangian $(\mathbb{R}_+)^2$ in the fiber and denote it by $\ell_i$. Parallel transporting $\ell_i\subset\ring{\pi}^{-1}(\zeta_i)$ along the arc $\gamma_i$ defines a Lagrangian submanifold $L_i\subset\ring{W}_f$ diffeomorphic to $\mathbb{R}^3$. Note that under $\ring{\pi}$, the compact cores $Q_0,\cdots, Q_n\subset\ring{W}_f$ project to arcs (i.e. gradient flow lines of (\ref{eq:potential}) in the proof of Lemma \ref{lemma:plumbing}) connecting $z_i$ to $z_{i+1\textrm{ mod }n+1}$ when $i=0,\cdots,n$. In particular, $L_i$ intersects $Q_i$ transversely at a unique point in $\ring{\pi}^{-1}(\zeta_i)$, and is disjoint from $Q_j$ for $j\neq i$. In fact, analyzing the Liouville vector field $Z=Z_F+Z_B$ in Lemma \ref{lemma:plumbing}, it is not hard to see that the Lagrangian submanifolds $L_0,\cdots,L_n$ are Hamiltonian isotopic to the cotangent fibers of $T^\ast Q_0,\cdots,T^\ast Q_n$, respectively. In particular, they generate the wrapped Fukaya category of the plumbing $\mathcal{W}(\ring{W}_f;\mathbb{K})$, see \cite{cdrgg,gps}. Our goal here is then to compute the endomorphism $A_\infty$-algebra
\begin{equation}\label{eq:WFA}
\ring{\mathcal{W}}_f:=\bigoplus_{0\leq i,j\leq n}\mathit{CW}^\ast(L_i,L_j).
\end{equation}

Next, we specify our wrapping Hamiltonian used to define the Floer cochain complexes $\mathit{CW}^\ast(L_i,L_j)$, where $0\leq i,j\leq n$. To do this, we start with the fact (used in Section \ref{section:plumbing}) that $\ring{\pi}:\ring{W}_f\rightarrow\mathbb{C}^\ast$ is obtained by patching together elementary Morse-Bott fibrations $\pi_i:E_i\rightarrow\mathbb{C}$ with a unique singular fiber isomorphic to $(\mathbb{C}\vee\mathbb{C})\times\mathbb{C}^\ast$ over the critical value $c_i$. Since $\pi_i$ can be identified with the product of the standard Lefschetz fibration $\pi_\mathit{std}:\mathbb{C}^2\rightarrow\mathbb{C}$ and $\mathbb{C}^\ast$, we start with the function $h:T^\ast S^1\rightarrow[0,\infty)$ which vanishes near the zero section, and equals $r^2$ when the radial coordinate $r$ on the cylinder is large enough. Since $h$ vanishes near the critical point, it extends to a fiberwise Hamiltonian $\tilde{h}:\mathbb{C}^2\rightarrow[0,\infty)$ by parallel transport. Denote by $\pi_{\mathbb{C}^2}:E_i\rightarrow\mathbb{C}^2$ and $\pi_{\mathbb{C}^\ast}:E_i\rightarrow\mathbb{C}^\ast$ the natural projections, define
\begin{equation}
H_{F,i}:=\pi_{\mathbb{C}^2}^\ast\tilde{h}+\pi_{\mathbb{C}^\ast}^\ast h.
\end{equation}
This gives a fiberwise Hamiltonian on the total space $E_i$ of the elementary Morse-Bott fibration $\pi_i$. One then uses the fiberwise biholomorphism of $(\mathbb{C}^\ast)^2$ to patch the locally defined fiberwise Hamiltonians $H_{F,i}$ together to a globally defined fiberwise Hamiltonian $H_F:\ring{W}_f\rightarrow[0,\infty)$. On the other hand, define $H_B=r^2$ to be the quadratic Hamiltonian on the base $(\mathbb{R}\times S^1,rd\theta)$. Finally, the wrapping Hamiltonian $H_W:\ring{W}_f\rightarrow[0,\infty)$ is defined to be a small Morse perturbation of
\begin{equation}\label{eq:HW}
\ring{\pi}^\ast H_B+H_F.
\end{equation}

\begin{lemma}\label{lemma:cofinal}
The wrapping of the Lagrangian submanifolds $L_i\subset\ring{W}_f$ induced by the Hamiltonian $H_W$ is cofinal for $i=0,\cdots,n$.
\end{lemma}
\begin{proof}
As we have seen in the proof of Lemma \ref{lemma:plumbing}, the cylindrical end $[r_0,\infty)\times\partial\overline{\ring{W}}_f$ for $r_0\gg1$ of $\ring{W}_f$ has a product structure induced by the Morse-Bott fibration $\ring{\pi}:\ring{W}_f\rightarrow\mathbb{C}^\ast$, so are the cylindrical ends $[1,\infty)\times\partial\overline{L}_i$, $0\leq i\leq n$ of our Lagrangian submanifolds, where $\overline{L}_i:=L_i\cap\overline{\ring{W}}_f$, since by construction they are fibered over the arcs $\gamma_i\subset\mathbb{C}^\ast$, $i=0,\cdots,n$. On the other hand, our wrapping Hamiltonian $H_W:\ring{W}_f\rightarrow[0,\infty)$ defined by (\ref{eq:HW}) has product form with respect to the product structure on the cylindrical end of $\ring{W}_f$. The lemma now follows from \cite{gps}, Proposition 7.4, which ensures the cofinality of the product wrapping.
\end{proof}

It follows from Lemma \ref{lemma:cofinal} that we have a quasi-isomorphism between the Fukaya $A_\infty$-algebra defined using $H_W$ and the Fukaya $A_\infty$-algebra defined using a wrapping Hamiltonian that is quadratic with respect to the usual cylindrical end $[r_0,\infty)\times\partial\overline{\ring{W}}_f$ coming from the Weinstein structure on the plumbing $\ring{W}_f$. Because of this, we do not distinguish between the wrapped Floer cochain complexes defined using the product type wrapping Hamiltonian $H_W$ adapted to the Morse-Bott fibration $\ring{\pi}$ and a usual quadratic wrapping Hamiltonian, and will simply write $\mathit{CW}^\ast(L_i,L_j)$ for both of the wrapped Floer cochain complexes.

An $\omega_{\ring{W}}$-compatible almost complex structure $J$ on $\ring{W}_f$ is called \textit{admissible} if the fibration $\ring{\pi}$ is $J$-holomorphic outside of some small disjoint discs $\Delta_0,\cdots,\Delta_n\subset\mathbb{C}^\ast$ centered at $c_0,\cdots,c_n$, respectively. We denote by $\mathcal{J}(\ring{\pi})$ the space of admissible almost complex structures on $\ring{W}_f$. For any $J\in\mathcal{J}(\ring{\pi})$, it follows from \cite{alp}, Lemma 2.2 that there exists a $J$-plurisubharmonic function $\hbar:\ring{W}_f\rightarrow[0,\infty)$ such that
\begin{itemize}
	\item $\hbar$ restricted to each fiber of $\ring{\pi}$ is proper,
	\item $d\hbar(X_{H_F})=0$ outside of $\ring{W}_f^\mathrm{in}$, a closed subset of $\ring{W}_f$ whose intersection with every fiber of $\ring{\pi}$ is a sublevel set of $\hbar$,
	\item $d\hbar(\xi^\#)=d^\mathbb{C}\hbar(\xi^\#)=0$ outside of $\ring{W}_f^\mathrm{in}\cup\bigcup_{i=0}^n\pi^{-1}(\Delta_i)$ for the horizontal lift $\xi^\#$ of some vector $\xi$ on $\mathbb{C}^\ast$,
	\item $\mathcal{L}_{X_{H_F}}(d^\mathbb{C}\hbar)=\mathcal{L}_{\xi^\#}(d^\mathbb{C}\hbar)=0$,
	\item $d^\mathbb{C}\hbar(X_{H_F})\geq0$.
\end{itemize}
The existence of such a function $\hbar$ ensures that the (perturbed) $J$-holomorphic curves defining the $A_\infty$-structure of the wrapped Fukaya category $\mathcal{W}(\ring{W}_f;\mathbb{K})$ are contained in a sublevel set of $\hbar:\ring{W}_f\rightarrow[0,\infty)$, see \cite{aah}, Proposition 3.11. On the other hand, we have a maximum principle with respect to the projection $\ring{\pi}:\ring{W}_f\rightarrow\mathbb{C}^\ast$, which ensures that all the Floer trajectories with Lagrangian boundary conditions on $L_0,\cdots,L_n$ are projected to a compact subset of the base $\mathbb{C}^\ast$, see \cite{aah}, Proposition 3.10 (which deals with the case when the fibration on the symplectic manifold is over $\mathbb{C}$, but the same argument applies to our situation, where the base is $\mathbb{C}^\ast$). Thus we have a well-defined $A_\infty$-category $\mathcal{W}(\ring{W}_f;\mathbb{K})$ over any field $\mathbb{K}$.

\subsection{The case of a single Lagrangian}\label{section:single}

We first focus on the wrapped Floer cochain complex $\mathit{CW}^\ast(L_i,L_i)$ of a single Lagrangian submanifold $L_i\subset\ring{W}_f$, where $0\leq i\leq n$, defined in Section \ref{section:setup} using the wrapping Hamiltonian $H_W:\ring{W}_f\rightarrow[0,\infty)$. 

\begin{proposition}\label{proposition:single}
For any $i=0,\cdots,n$, we have an isomorphism
\begin{equation}\label{eq:Laurent}
\mathit{HW}^\ast(L_i,L_i)\cong\mathbb{K}[z_1^{\pm1},z_2^{\pm1},z_3^{\pm1}].
\end{equation}
between $\mathbb{K}$-modules, where the variables $z_1,z_2,z_3$ have degree $0$.
\end{proposition}
\begin{proof}
Denote by $L_i^{(1)}:=\phi_{H_W}^1(L_i)$ the wrapping of $L_i$ under the time-$1$ flow of the Hamiltonian vector field $X_{H_W}$. Under the projection $\ring{\pi}$, the Lagrangians $L_i$ and $L_i^{(1)}$ define arcs $\gamma_i,\gamma_i^{(1)}\subset\mathbb{C}^\ast$. Since the Hamiltonian $H_B:\mathbb{C}^\ast\rightarrow[0,\infty)$ on the base of $\ring{\pi}:\ring{W}_f\rightarrow\mathbb{C}^\ast$ is quadratic and the arc $\gamma_i\subset T^\ast S^1$ can be identified with the cotangent fiber, we get the intersection points $\gamma_i\cap\gamma_i^{(1)}=\{x_r\}_{r\in\mathbb{Z}}$ in the base cylinder $\mathbb{C}^\ast$, see Figure \ref{fig:wrapping}. By our choice of the wrapping Hamiltonian $H_W:\ring{W}_f\rightarrow[0,\infty)$, above each intersection point $x_r$, the fiberwise Lagrangians $\ell_{i,r},\ell_{i,r}^{(1)}:=\phi_{H_F}^1(\ell_{i,r})\subset\ring{\pi}^{-1}(x_r)\cong(\mathbb{C}^\ast)^2$, which are obtained by parallel transporting $\ell_i,\ell_i^{(1)}:=\phi_{H_F}^1(\ell_i)\subset\ring{\pi}^{-1}(\zeta_i)$ along the path on $\gamma_i$ connecting $\zeta_i$ to $x_r$, intersect at the points $\{y_{p,q,r}\}_{(p,q)\in\mathbb{Z}^2}$. It follows that the wrapped Floer cochain complex can be identified with
\begin{equation}\label{eq:id}
\mathit{CW}^\ast(L_i,L_i)\cong\bigoplus_{(p,q,r)\in\mathbb{Z}^3}\mathbb{K}\cdot\vartheta_{p,q,r}^i,
\end{equation}
where $\vartheta_{0,0,r}^i$ corresponds to $x_r\in\mathbb{C}^\ast$ and $\vartheta_{p,q,0}^i$ corresponds to the intersection point $y_{p,q}$ in the fiber $\ring{\pi}^{-1}(\zeta_i)$. The grading on the fibered Lagrangians $L_i,L_i^{(1)}\subset\ring{W}_f$ are induced from the choices of gradings on the fiberwise Lagrangians $\ell_i,\ell_i^{(1)}\subset\ring{\pi}^{-1}(\zeta_i)\cong(\mathbb{C}^\ast)^2$ and the base arcs $\gamma_i,\gamma_i^{(1)}\subset\mathbb{C}^\ast$. It follows that if we use the standard grading on the fiberwise Lagrangians, then each generator $\vartheta_{p,q,r}^i$ is supported in degree $0$. Thus the Floer differential vanishes automatically. After identifying the generator $\vartheta_{p,q,r}^i$ with the monomial $z_1^pz_2^qz_3^r$, we obtain the isomorphism (\ref{eq:Laurent}).
\end{proof}

\begin{figure}
	\centering
	\begin{tikzpicture}
	\filldraw [draw=black, color={black!15}] (7.5,0) cos (7,-1.25) sin (5,-1.25) sin (6,0);
	\filldraw [draw=black,color={black!6}] (7,-1.25) sin (7.5,0) cos (8,-1.25) sin (8.5,-2.5)--(8,-2.5) cos (7,-1.25);
	\filldraw [draw=black,color={black!15}] (8,-2.5) cos (9,-1.25) sin (8.5,-2.5)--(8,-2.5);
	\draw (0,0) to (10,0);
	\draw (0,-2.5) to (10,-2.5);
	\draw (-0.5,-1.25) to (10.5,-1.25);
	\draw [dash dot] (-0.5,0) to (0,0);
	\draw [dash dot] (10,0) to (10.5,0);
	\draw [dash dot] (-0.5,-2.5) to (0,-2.5);
	\draw [dash dot] (10,-2.5) to (10.5,-2.5);
	\draw (10.8,-1.25) node {\small$L_i$};
	\draw [dotted] (0,0) arc(90:270:0.3 and 1.25);
	\draw (0,-2.5) arc(-90:90:0.3 and 1.25);
	\draw [dotted] (10,0) arc(90:270:0.3 and 1.25);
	\draw (10,-2.5) arc(-90:90:0.3 and 1.25);
	\draw [blue,dotted] (-0.5,-2) sin (0,-2.5);
	\draw [blue] (0,-2.5) cos (1,-1.25);
	\draw [blue] (1,-1.25) sin (2,0);
	\draw [blue,dotted] (2,0) cos (3,-1.25);
	\draw [blue,dotted] (3,-1.25) sin (4,-2.5);
	\draw [blue] (4,-2.5) cos (5,-1.25);
	\draw [blue] (5,-1.25) sin (6,0);
	\draw [blue,dotted] (6.,0) cos (7,-1.25);
	\draw [blue,dotted] (7,-1.25) sin (8,-2.5);
	\draw [blue] (8,-2.5) cos (9,-1.25);
	\draw [blue] (9,-1.25) sin (10,0);
	\draw [blue,dotted] (10,0) cos (10.5,-0.5);
		
	\draw [red,dotted] (-0.5,0) cos (0,-1.25);
	\draw [red,dotted] (0,-1.25) sin (0.5,-2.5);
	\draw [red] (0.5,-2.5) cos (1,-1.25);
	\draw [red] (1,-1.25) sin (1.5,0);
	\draw [red,dotted] (1.5,0) cos (2,-1.25);
	\draw [red,dotted] (2,-1.25) sin (2.5,-2.5);
	\draw [red] (2.5,-2.5) cos (3,-1.25);
	\draw [red] (3,-1.25) sin (3.5,0);
	\draw [red,dotted] (3.5,0) cos (4,-1.25);
	\draw [red,dotted] (4,-1.25) sin (4.5,-2.5);
	\draw [red] (4.5,-2.5) cos (5,-1.25);
	\draw [red] (5,-1.25) sin (5.5,0);
	\draw [red,dotted] (5.5,0) cos (6,-1.25);
	\draw [red,dotted] (6,-1.25) sin (6.5,-2.5);
	\draw [red] (6.5,-2.5) cos (7,-1.25);
	\draw [red] (7,-1.25) sin (7.5,0);
	\draw [red,dotted] (7.5,0) cos (8,-1.25);
	\draw [red,dotted] (8,-1.25) sin (8.5,-2.5);
	\draw [red] (8.5,-2.5) cos (9,-1.25);
	\draw [red] (9,-1.25) sin (9.5,0);
	\draw [red,dotted] (9.5,0) cos (10,-1.25);
	\draw [red,dotted] (10,-1.25) sin (10.5,-2.5);
	\draw [violet] (1,-1.25) node[circle,fill,inner sep=1pt] {};
	\draw [violet] (9,-1.25) node[circle,fill,inner sep=1pt] {};
	\draw [violet] (5,-1.25) node[circle,fill,inner sep=1pt] {};
	\draw [violet] (7,-1.25) node[circle,fill,inner sep=1pt] {};
	\draw [violet] (1.3,-1.5) node {\small$x_{-1}$};
	\draw [violet] (5.2,-1.5) node {\small$x_0$};
	\draw [violet] (9.2,-1.5) node {\small$x_1$};
	\draw [violet] (7.3,-1.5) node {\small$\tilde{x}_1$};
	\draw [blue] (10.9,-0.5) node {\small$L_i^{(1)}$};
	\draw [red] (10.9,-2.5) node {\small$L_i^{(2)}$};
	\end{tikzpicture}
	\caption{Wrapping in the cylinder \label{fig:wrapping}}
\end{figure}

\begin{paragraph}{Convention.}
We shall use monomials in $z_3$ to denote the generators corresponding to the intersection points $x_r\in\gamma_i\cap\gamma_i^{(1)}$ on the base $\mathbb{C}^\ast$, and use monomials in the variables $z_1,z_2$ to denote the generators in the fibers of $\ring{\pi}:\ring{W}_f\rightarrow\mathbb{C}^\ast$.
\end{paragraph}
\bigskip

Due to the existence of singular fibers of $\ring{\pi}:\ring{W}_f\rightarrow\mathbb{C}^\ast$, the triangle product
\begin{equation}\label{eq:prod-H}
\mu^2:\mathit{HW}^\ast(L_i,L_i)\otimes\mathit{HW}^\ast(L_i,L_i)\rightarrow\mathit{HW}^\ast(L_i,L_i)
\end{equation}
on the wrapped Floer cohomology of $L_i$ can potentially differ from that on the Laurent polynomial ring, therefore giving rise to a deformation of $\mathbb{K}[z_1^{\pm1},z_2^{\pm1},z_3^{\pm1}]$ as a $\mathbb{K}$-algebra. We shall determine it explicitly in the following proposition.

\begin{proposition}\label{proposition:single2}
There is a $\mathbb{K}$-algebra isomorphism
\begin{equation}\label{eq:dring}
\mathit{HW}^\ast(L_i,L_i)\cong\frac{\mathbb{K}\left[z_1^{\pm1},z_2^{\pm1},(x_r)_{r\in\mathbb{Z}}\right]}{\left(x_{-1}\cdot x_1-f(z_1,z_2),\textrm{ }x_r\cdot x_{r'}-x_{r+r'}\textrm{ when }rr'\geq0\right)},
\end{equation}
where $f(z_1,z_2)=f_0(z_1,z_2)\cdots f_n(z_1,z_2)$ is the polynomial (\ref{eq:f}) and all the generators are supported in degree $0$. Here, we regard $x_r\in\mathit{CW}^\ast(\gamma_i,\gamma_i)$ in the base $\mathbb{C}^\ast$ as a generator of $\mathit{CW}^\ast(L_i,L_i)$ by identifying it with $\vartheta_{0,0,r}^i$ in (\ref{eq:id}).
\end{proposition}
\begin{proof}
Let $L_i^{(1)},L_i^{(2)}\subset\ring{W}_f$ be the Lagrangian submanifolds obtained as wrappings of $L_i$ under the time-$1$ and time-$2$ flows of $X_{H_W}$, respectively. Consider the cochain level map
\begin{equation}\label{eq:prod}
\mu^2:\mathit{CW}^\ast\left(L_i^{(1)},L_i\right)\otimes\mathit{CW}^\ast\left(L_i^{(2)},L_i^{(1)}\right)\rightarrow\mathit{CW}^\ast\left(L_i^{(2)},L_i\right)
\end{equation}
underlying the product (\ref{eq:prod-H}). It is defined by counting (perturbed) $J$-holomorphic triangles $u:\mathbb{D}\setminus\{p_0,p_1,p_2\}\rightarrow\left(\ring{W}_f,L_i\cup L_i^{(1)}\cup L_i^{(2)}\right)$, where $\mathbb{D}$ is the closed unit disc and $p_0,p_1,p_2\in\partial\mathbb{D}$ are ordered counterclockwisely, with the connected boundary components between $p_0,p_1$, $p_1,p_2$ and $p_2,p_0$ mapped to the Lagrangian submanifolds $L_i,L_i^{(1)}$ and $L_i^{(2)}$, respectively. It follows from the open mapping principle that the image of $u$ is either contained in some smooth fiber of $\ring{\pi}:\ring{W}_f\rightarrow\mathbb{C}^\ast$, or it is a $J$-holomorphic section of $\ring{\pi}$ over some triangle formed by the arcs $\gamma_i,\gamma_i^{(1)},\gamma_i^{(2)}\subset\mathbb{C}^\ast$ for some $J\in\mathcal{J}(\ring{\pi})$, which are projections of the Lagrangians $L_i,L_i^{(1)},L_i^{(2)}$ under $\ring{\pi}$, respectively. Figures \ref{fig:wrapping} and \ref{fig:triangle} are two equivalent descriptions of the wrappings in the base, where the shaded regions are examples of a triangle formed by $\gamma_i,\gamma_i^{(1)}$ and $\gamma_i^{(2)}$, over which a $J$-holomorphic section of $\ring{\pi}$ appears. Since the smooth fibers of $\ring{\pi}$ are symplectomorphic to $T^\ast T^2$ and our wrapping Hamiltonian $H_W:\ring{W}_f\rightarrow[0,\infty)$ has product form in the fiber direction, the counting of the $J$-holomorphic triangles in the fiber always gives $1$ if it is non-trivial. This enables us to identify the fiberwise Floer cochain complexes $\mathit{CW}^\ast\left(\ell_{i,r},\ell_{i,r}^{(1)}\right)$ over some fixed $x_r\in\gamma_i\cap\gamma_i^{(1)}$ together with their triangle products as Laurent polynomial rings $\mathbb{K}[z_1^{\pm1},z_2^{\pm1}]$, and counting $J$-holomorphic sections of $\ring{\pi}$ over a fixed triangle with vertices $x_r\in\gamma_i\cap\gamma_i^{(1)}$, $x_{r'}\in\gamma_i^{(1)}\cap\gamma_i^{(2)}$ and $\tilde{x}_{r+r'}\in\gamma_i\cap\gamma_i^{(2)}$ in $\mathbb{C}^\ast$ gives rise to a map
\begin{equation}\label{eq:fiberwise}
\mathbb{K}[z_1^{\pm1},z_2^{\pm1}]\otimes\mathbb{K}[z_1^{\pm1},z_2^{\pm1}]\rightarrow\mathbb{K}[z_1^{\pm1},z_2^{\pm1}].
\end{equation}
In our case, depending on the positions of the generators on the base, there are two possibilities, shown in Figures \ref{fig:triangle} and \ref{fig:triangle2}, respectively. We remark that the intersection points $x_r\in\gamma_i\cap\gamma_i^{(1)}$ are labeled so that $x_0=\zeta_i$, the points $x_r$ with $r>0$ satisfy $|x_r|>1$ and $|x_{r+1}|>|x_r|$, while $x_r$ with $r<0$ satisfy $0<|x_r|<1$ and $|x_{r+1}|<|x_r|$.

When the intersections $x_r\in\gamma_i\cap\gamma_i^{(1)}$ and $x_{r'}\in\gamma_i^{(1)}\cap\gamma_i^{(2)}$ satisfy $rr'\geq0$, the triangle formed by the arcs $\gamma_i,\gamma_i^{(1)},\gamma_i^{(2)}$ does not contain any of the critical values of $\ring{\pi}$, see Figure \ref{fig:triangle}. It follows that we can shrink such a triangle to a single point by deforming the arcs $\gamma_i,\gamma_i^{(1)},\gamma_i^{(2)}$ using a compactly supported isotopy in $\mathbb{C}^\ast$, which reduces the counting to that inside a single fiber of $\ring{\pi}$, see for example \cite{aah}, Proposition 5.14. This shows that there is precisely one $J$-holomorphic section of $\ring{\pi}$ over the triangle and the map (\ref{eq:fiberwise}) is nothing but usual multiplication between Laurent polynomials. This establishes the relation $x_r\cdot x_{r'}=x_{r+r'}$ when $rr'\geq0$ in the isomorphism (\ref{eq:dring}). Note that we have identified $\tilde{x}_{r+r'}$ with $x_{r+r'}$.

When the intersections $x_r\in\gamma_i\cap\gamma_i^{(1)}$ and $x_{r'}\in\gamma_i^{(1)}\cap\gamma_i^{(2)}$ satisfy $rr'<0$, the (immersed) triangles formed by the arcs $\gamma_i,\gamma_i^{(1)},\gamma_i^{(2)}$ contain all the critical values of $\ring{\pi}$ (although some critical values may be multiply covered), see Figure \ref{fig:triangle2}. Since the general case can be recovered from the products $x_{-1}\cdot x_1$ and the relation $x_r\cdot x_{r'}=x_{r+r'}$ when $rr'\geq0$ established above, we only need to do the computation in the case when $r=-1$, $r'=1$. In this case, every critical value of $\ring{\pi}$ is covered once by the projection of the image of $u$ to $\mathbb{C}^\ast$. The map (\ref{eq:fiberwise}) is given by the composition of the product of Laurent polynomials (which comes from the triangle products of generators in the fiber direction) and the multiplication with $g(z_1,z_2)\in\mathbb{K}[z_1^{\pm1},z_2^{\pm1}]$ determined by the weights of the count of $J$-holomorphic sections of $\ring{\pi}:\ring{W}_f\rightarrow\mathbb{C}^\ast$, i.e.
\begin{equation}\label{eq:g}
\mathbb{K}[z_1^{\pm1},z_2^{\pm1}]\otimes\mathbb{K}[z_1^{\pm1},z_2^{\pm1}]\xrightarrow{\cdot}\mathbb{K}[z_1^{\pm1},z_2^{\pm1}]\xrightarrow{\cdot g(z_1,z_2)}\mathbb{K}[z_1^{\pm1},z_2^{\pm1}].
\end{equation}
The fact that $g(z_1,z_2)=f(z_1,z_2)$ is established in Lemma \ref{lemma:poly} below. This proves the relation $x_{-1}\cdot x_1=f(z_1,z_2)$ in (\ref{eq:dring}), since $x_0$ corresponds to the idempotent $e_i\in\mathit{CW}^0(L_i,L_i)$.
\end{proof}

\begin{figure}
	\centering
	\begin{tikzpicture}
	\filldraw [draw=black, color={black!15}] (5.5,-1.25)--(6.5,0)--(7,0)--(6.5,-1.25)--(5.5,-1.25);
	\filldraw [draw=black, color={black!15}] (6.5,-2.5)--(7.5,-1.25)--(7,-2.5);
	\draw (0,0) to (11,0);
	\draw (0,-2.5) to (11,-2.5);
	\draw (0,-1.25) to (11,-1.25);
	\draw (11.2,-1.25) node {\small$L_i$};
	\draw (5.5,-0.625) node[circle,fill,inner sep=1pt] {};
	\draw (5.5,-1.875) node[circle,fill,inner sep=1pt] {};
	\draw (5.3,-0.625) node {\small$c_i$};
	\draw (5.8,-1.875) node {\small$c_\ast$};
	\draw [red] (11.4,0) node {\small$L_i^{(2)}$};
	\draw [blue] (11.4,-1.875) node {\small$L_i^{(1)}$};
	\draw [dashed] (5.5,0)--(5.5,-2.5);
	
	\draw [red] (0,-2.5) to (1,0);
	\draw [red] (1,-2.5) to (2,0);
	\draw [red] (2,-2.5) to (3,0);
	\draw [red] (3,-2.5) to (4,0);
	\draw [red] (4,-2.5) to (5,0);
	\draw [red] (5,-2.5) to (6,0);
	\draw [red] (6,-2.5) to (7,0);
	\draw [red] (7,-2.5) to (8,0);
	\draw [red] (8,-2.5) to (9,0);
	\draw [red] (9,-2.5) to (10,0);
	\draw [red] (10,-2.5) to (11,0);
	
	\draw [blue] (0,-0.625) to (0.5,0);
	\draw [blue] (0.5,-2.5) to (2.5,0);
	\draw [blue] (2.5,-2.5) to (4.5,0);
	\draw [blue] (4.5,-2.5) to (6.5,0);
	\draw [blue] (6.5,-2.5) to (8.5,0);
	\draw [blue] (8.5,-2.5) to (10.5,0);
	\draw [blue] (10.5,-2.5) to (11,-1.875);
	
	\draw [violet] (5.5,-1.25) node[circle,fill,inner sep=1pt] {};
	\draw [violet] (6.5,-1.25) node[circle,fill,inner sep=1pt] {};
	\draw [violet] (7.5,-1.25) node[circle,fill,inner sep=1pt] {};
	\draw [violet] (5.7,-1.45) node {\small$x_0$};
	\draw [violet] (6.7,-1.45) node {\small$\tilde{x}_1$};
	\draw [violet] (7.7,-1.45) node {\small$x_1$};
	\end{tikzpicture}
    \caption{The triangle product of generators above $x_0$ and $x_1$, where the upper edge and the lower edge of the strip are identified. The shaded triangle corresponds precisely to the one in Figure \ref{fig:wrapping}. $c_\ast$ represents the critical values other than $c_i$.}\label{fig:triangle}
\end{figure}

\begin{figure}
	\centering
	\begin{tikzpicture}
	\filldraw [draw=black, color={black!15}] (3.5,-1.25)--(4.5,0)--(6,0)--(5.5,-1.25)--(3.5,-1.25);
	\filldraw [draw=black, color={black!15}] (6.5,0)--(4.5,-2.5)--(6,-2.5)--(7,0)--(6.5,0);
	\filldraw [draw=black, color={black!15}] (6.5,-2.5)--(7,-2.5)--(7.5,-1.25)--(6.5,-2.5);
	\draw (0,0) to (11,0);
	\draw (0,-2.5) to (11,-2.5);
	\draw (0,-1.25) to (11,-1.25);
	\draw (11.2,-1.25) node {\small$L_i$};
	\draw (5.5,-0.625) node[circle,fill,inner sep=1pt] {};
	\draw (5.5,-1.875) node[circle,fill,inner sep=1pt] {};
	\draw (5.3,-0.625) node {\small$c_i$};
	\draw (5.8,-1.875) node {\small$c_\ast$};
	\draw [red] (11.4,0) node {\small$L_i^{(2)}$};
	\draw [blue] (11.4,-1.875) node {\small$L_i^{(1)}$};
	\draw [dashed] (5.5,0)--(5.5,-2.5);
	
	\draw [red] (0,-2.5) to (1,0);
	\draw [red] (1,-2.5) to (2,0);
	\draw [red] (2,-2.5) to (3,0);
	\draw [red] (3,-2.5) to (4,0);
	\draw [red] (4,-2.5) to (5,0);
	\draw [red] (5,-2.5) to (6,0);
	\draw [red] (6,-2.5) to (7,0);
	\draw [red] (7,-2.5) to (8,0);
	\draw [red] (8,-2.5) to (9,0);
	\draw [red] (9,-2.5) to (10,0);
	\draw [red] (10,-2.5) to (11,0);
	
	\draw [blue] (0,-0.625) to (0.5,0);
	\draw [blue] (0.5,-2.5) to (2.5,0);
	\draw [blue] (2.5,-2.5) to (4.5,0);
	\draw [blue] (4.5,-2.5) to (6.5,0);
	\draw [blue] (6.5,-2.5) to (8.5,0);
	\draw [blue] (8.5,-2.5) to (10.5,0);
	\draw [blue] (10.5,-2.5) to (11,-1.875);
	
	\draw [violet] (5.5,-1.25) node[circle,fill,inner sep=1pt] {};
	\draw [violet] (7.5,-1.25) node[circle,fill,inner sep=1pt] {};
	\draw [violet] (3.5,-1.25) node[circle,fill,inner sep=1pt] {};
	\draw [violet] (3.8,-1.45) node {\small$x_{-1}$};
	\draw [violet] (5.8,-1.45) node {\small$\tilde{x}_0$};
	\draw [violet] (7.7,-1.45) node {\small$x_1$};
    \end{tikzpicture}
    \caption{The triangle product of generators above $x_{-1}$ and $x_1$}\label{fig:triangle2}
\end{figure}

\begin{lemma}\label{lemma:poly}
The polynomial $g(z_1,z_2)$ in (\ref{eq:g}) equals $f(z_1,z_2)$.
\end{lemma}
\begin{proof}
In order to determine the Laurent polynomial $g(z_1,z_2)\in\mathbb{K}[z_1^{\pm1},z_2^{\pm1}]$, we use a trick of Abouzaid-Auroux \cite{aah}, which replaces the fibered Lagrangians $L_i,L_i^{(1)},L_i^{(2)}\subset\ring{W}_f$ with Lagrangian cylinders defined as follows. For each $z=(z_1,z_2)\in(\mathbb{K}^\ast)^2$, let $\tau_z\subset\ring{\pi}^{-1}(\zeta_i)$ be a product torus in $(\mathbb{C}^\ast)^2$ equipped with a rank one local system over $\mathbb{K}$ whose holonomy around the $j$-th $S^1$-factor is $z_j^{-1}\in\mathbb{K}^\ast$ for $j=1,2$. Note that $\tau_z$ is invariant under parallel transport between the fibers of $\ring{\pi}$, and we can choose the wrapping Hamiltonian $H_W:\ring{W}_f\rightarrow[0,\infty)$ in Section \ref{section:setup} so that $\tau_z$ is invariant under the flow of the fiberwise Hamiltonian vector field $X_{H_F}$. Let $T_z^{(t)}\subset\ring{W}_f$ be the Lagrangian cylinder defined by parallel transporting $\tau_z$ over the curve $\gamma_i^{(t)}$ obtained as the time-$t$ flow of $\gamma_i$ under $X_{H_B}$. Note that $T_z^{(t)}\cong T^2\times\mathbb{R}$ is also equipped with a rank one local system by parallel transporting the one on $\tau_z$.

As a $\mathbb{K}$-module, the wrapped Floer cochain complex $\mathit{CW}^\ast\left(T_z,T_z^{(1)}\right)$ can be identified with $H^\ast(T^2;\mathbb{K})[z_3^{\pm1}]$ by a similar argument as in Proposition \ref{proposition:single}. Consider the triangle product
\begin{equation}\label{eq:pt}
\mu^2:\mathit{CW}^\ast\left(T_z^{(1)},T_z\right)\otimes\mathit{CW}^\ast\left(T_z^{(2)},T_z^{(1)}\right)\rightarrow\mathit{CW}^\ast\left(T_z^{(2)},T_z\right)
\end{equation}
obtained by replacing the Lagrangians $L_i,L_i^{(1)}$ and $L_i^{(2)}$ in (\ref{eq:prod}) with the Lagrangian cylinders $T_z,T_z^{(1)}$ and $T_z^{(2)}$, respectively. It follows from \cite{alp}, Lemma 2.17 that under the identification $\mathit{CW}^\ast\left(T_z,T_z^{(t)}\right)\cong H^\ast(T^2;\mathbb{K})[z_3^{\pm1}]$, the product (\ref{eq:pt}), which is a map
\begin{equation}
H^\ast(T^2;\mathbb{K})[z_3^{\pm1}]\otimes H^\ast(T^2;\mathbb{K})[z_3^{\pm1}]\rightarrow H^\ast(T^2;\mathbb{K})[z_3^{\pm1}]
\end{equation}
can be identified with the composition of the cup product on $H^\ast(T^2;\mathbb{K})$ and the multiplication by $g(z_1,z_2)\in\mathbb{K}$ for any $z=(z_1,z_2)\in(\mathbb{K}^\ast)^2$ in the fiber direction of $\ring{\pi}:\ring{W}_f\rightarrow\mathbb{C}^\ast$, while in the base direction it is given by the usual multiplication of Laurent polynomials in $z_3$. Here $g(z_1,z_2)$ is the Laurent polynomial in (\ref{eq:g}). Thus we have turned the problem of counting $J$-holomorphic sections of $\ring{\pi}$ with boundary on $L_i\cup L_i^{(1)}\cup L_i^{(2)}$ into the equivalent problem of counting $J$-holomorphic sections of $\ring{\pi}$ with boundary on $T_z\cup T_z^{(1)}\cup T_z^{(2)}$. To deal with the latter problem, first note that since all of the singular fibers of $\ring{\pi}$ are isomorphic to $(\mathbb{C}\vee\mathbb{C})\times\mathbb{C}^\ast$, there are altogether $2^{n+1}$ different homotopy classes in $\pi_2\left(\ring{W}_f,T_z\cup T_z^{(1)},T_z^{(2)}\right)$ of $J$-holomorphic sections of $\ring{\pi}$ over the embedded triangle formed by the arcs $\gamma_i,\gamma_i^{(1)},\gamma_i^{(2)}$, which in this case contains all the critical values $c_0,\cdots,c_n$. We claim that for each homotopy class, the moduli space of $J$-holomorphic sections consists of a single orbit under the fiberwise Hamiltonian $T^2$-action, so that the count of these sections passing through any given point of $\tau_z\subset\ring{\pi}^{-1}(\zeta_i)$ is equal to one.

To see this, denote by $Z_i^\pm\cong\mathbb{C}\times\mathbb{C}^\ast$, $i=0,\cdots,n$ the irreducible components of the singular fibers $\ring{\pi}^{-1}(c_i)$. Fix some $i$ and a sign sequence $\alpha=(\alpha_0,\cdots,\alpha_n)\in\{\pm\}^{n+1}$, consider the complement $Y_\alpha:=\ring{W}_f\setminus(\bigcup_{j=0}^nZ_j^{\alpha_j})$. It is not hard to see that $Y_\alpha\cong(\mathbb{C}^\ast)^3$. Let $u:\Delta_i\rightarrow\ring{W}_f$ be a $J$-holomorphic section of $\ring{\pi}$ above the triangle $\Delta_i$ with vertices $x_{-1},x_1$ and $\tilde{x}_0$ (cf. Figure \ref{fig:triangle2}) intersecting $Z_j^{\alpha_j}\subset\ring{\pi}^{-1}(c_j)$ at a single point for each $j=0,\cdots,n$ and disjoint from the other $Z_j^{-\alpha_j}$'s, then $u(\Delta_i)\subset Y_\alpha$. By parametrizing the holomorphic sections of $\ring{\pi}|_{Y_\alpha}:Y_\alpha\rightarrow\mathbb{C}^\ast$ using the complex coordinate $z_3$ on $\Delta_i\subset\mathbb{C}^\ast$, we are reduced to finding holomorphic maps $\Delta_i\rightarrow(\mathbb{C}^\ast)^2$ which satisfy appropriate boundary conditions over $\partial\Delta_i$. It now follows from the standard complex analysis argument (cf. \cite{aah}, Proposition 5.24) that the moduli space of $J$-holomorphic sections $\mathcal{S}(\ring{\pi},\Delta_i,J)$ of $\ring{\pi}$ over $\Delta_i$ in this particular class consists of a single $T^2$-orbit.

It remains to determine the contribution of this unique $T^2$-orbit. The $J$-holomorphic sections of $\ring{\pi}$ contributing to the triangle product (\ref{eq:pt}) is weighted by the holonomies of the local systems on the Lagrangian cylinders $T_z^{(t)}$ for $t=0,1,2$. To determine these weights, we first smooth the boundary of the triangular region $\Delta_i\subset\mathbb{C}^\ast$ formed by the arcs $\gamma_i,\gamma_i^{(1)},\gamma_i^{(2)}$ to a simple closed curve $\tilde{\gamma}_i\subset\mathbb{C}^\ast$, which deforms the union of Lagrangian tori $\bigcup_{z\in\partial\Delta_i}\tau_z$ to a product Lagrangian torus $T_i\cong T^3\subset\ring{W}_f$, so that the holonomy of the local system on $T_i$ along the deformed curve $\tilde{\gamma}_i$ (a smooth curve encircling all the critical values of $\ring{\pi}$, see the left-hand side of Figure \ref{fig:deformation}) is trivial. Note that the local system on $T_i$ is defined by parallel transporting the one on $\tau_z$. We then perform a second deformation of the curve $\tilde{\gamma}_i$, so that the deformed curve $\tilde{\gamma}_i'$ is sufficiently close to the concatenation of the basic loops $\beta_0,\cdots,\beta_n$ such that each $\beta_i$ is a basic loop encircling the critical value $c_i$, see the right-hand side of Figure \ref{fig:deformation}. Since these deformations are isotopies, they do not affect the holonomy $\mathrm{hol}\left(u(\partial\Delta_i)\right)\in\mathbb{K}^\ast$.

Denote by $S_i\subset\mathbb{C}^\ast$ the region bounded by the smooth curve $\tilde{\gamma}_i'$, and by $D_{ij}$ the small closed discs centered at $c_j$, $j=0,\cdots,n$, with $\partial D_{ij}=\beta_j$. Locally near $c_j$ we have an elementary Morse-Bott fibration $\pi_j:\overline{E}_j\rightarrow D_{ij}$ with a unique singular fiber, where $\overline{E}_j\subset E_j\cong\mathbb{C}^2\times\mathbb{C}^\ast$ is defined to be $\ring{\pi}^{-1}(D_{ij})$. Let $\mathcal{S}(\ring{\pi},S_i,J)$ be the moduli space of $J$-holomorphic sections of $\ring{\pi}$ with boundary conditions given by the product torus $T_i'\subset\ring{W}_f$ which fibers over $\tilde{\gamma}_i'$. For a generic $J\in\mathcal{J}(\ring{\pi})$, $\mathcal{S}(\ring{\pi},S_i,J)$ is a smooth manifold, and we write $\mathcal{J}_\mathit{reg}(\ring{\pi})\subset\mathcal{J}(\ring{\pi})$ for the ($C^\infty$ dense) subspace of regular admissible almost complex structures. Let $\mathcal{S}(\pi_j,D_{ij},J_j)$ be the moduli space of $J_j$-holomorphic sections of $\pi_j$, so that the almost complex structure $J_j$ is regular and $\pi_j$ is $J_j$-holomorphic outside of a small disc centered at $c_j\in D_{ij}$. The latter space carries an evaluation map
\begin{equation}
\mathit{ev}_j:\mathcal{S}(\pi_j,D_{ij},J_j)\rightarrow\tau_z
\end{equation}
at fiber over $\ast\in\partial D_{ij}$, where $\tau_z\subset\pi_j^{-1}(\ast)$ is the product torus defined by parallel transporting the one in the fiber $\ring{\pi}^{-1}(\zeta_j)$. We shall use the following fact about the gluing of pseudoholomorphic curves.

\begin{proposition}\label{proposition:fp}
Assume that the almost complex structures $J_j$ are regular and the evaluation maps $\mathit{ev}_j$ and $\mathit{ev}_{j+1\textrm{ mod }n+1}$ are mutually transverse for $j=0,\cdots,n$, then for sufficiently small gluing parameters, there is an almost complex structure $J\in\mathcal{J}_\mathit{reg}(\ring{\pi})$, such that we have the diffeomorphism
\begin{equation}\label{eq:fp}
\begin{split}
\mathcal{S}(\ring{\pi},S_i,J)=&\mathcal{S}(\pi_j,D_{ij},J_j){{}_{\mathit{ev}_j}\times_{\mathit{ev}_{j+1\textrm{ mod }n+1}}}\mathcal{S}(\pi_{j+1\textrm{ mod }n+1},D_{i(j+1\textrm{ mod }n+1)},J_{j+1\textrm{ mod }n+1}) \\
&{{}_{\mathit{ev}_{j+1\textrm{ mod }n+1}}\times_{\mathit{ev}_{j+2\textrm{ mod }n+1}}}\cdots{{}_{\mathit{ev}_{j-2}}\times_{\mathit{ev}_{j-1}}}\mathcal{S}(\pi_{j-1},D_{i(j-1)},J_{j-1}),
\end{split}
\end{equation}
between the moduli space of $J$-holomorphic sections of $\ring{\pi}:\ring{W}_f\rightarrow\mathbb{C}^\ast$ over $S_i$ and the fiber producst of the individual moduli spaces $\mathcal{S}(\pi_j,D_{ij},J_j)$, where the evaluation maps $\mathit{ev}_j$ and $\mathit{ev}_{j+1\textrm{ mod }n+1}$ are at the point $\{\ast\}=\beta_j\cap\beta_{j+1\textrm{ mod }n+1}$. 
\end{proposition}

A similar statement of the proposition above can be found in \cite{psa}, Proposition 2.7, although that is more restrictive in the sense that Seidel was dealing with pseudoholomorphic sections of a Lefschetz fibration. See also \cite{psf}, (17c). However, the same argument as in \cite{psa}, Section 2.4 can be adapted here to prove Proposition \ref{proposition:fp} in its form stated above. See also the discussions in \cite{ww}, Section 4.4, which covers our situation.

Recall that the count of $J$-holomorphic sections in $\mathcal{S}(\ring{\pi},S_i,J)$ is weighted by the holonomies of the local systems on the product Lagrangian torus $T_i'$. By (\ref{eq:fp}), this is the product of the holonomies of the local systems on the product Lagrangian tori $T_{ij}\subset\overline{E}_j$ obtained by parallel transporting $\tau_z\subset\pi_j^{-1}(\zeta_j)$ along the circle $\beta_j$. 

Without loss of generality, we assume from now on that the vanishing cycle of $\ring{\pi}:\ring{W}_f\rightarrow\mathbb{C}^\ast$ at $c_j$ is $k_ja+l_jb$ with $(k_j,l_j)\in\mathbb{Z}_{\geq0}^2$, where $a,b\in H_1(T^2;\mathbb{Z})$ is the global basis fixed in the introduction. For each $j\in\{0,\cdots,n\}$, let $[D_\pm]\in H_2(E_j,T_{ij})\cong\mathbb{Z}^2$ be the two relative homology classes of the sections of $\pi_j:\overline{E}_j\rightarrow D_{ij}$, so that $[D_+]$ is the class of sections intersecting the irreducible component $Z_j^+\subset\pi_j^{-1}(c_j)$ once positively, and $[D_-]$ is the other class of sections. Following the proof of \cite{aah}, Proposition 5.26, we define a $2$-chain $D_{+-}$ on $\overline{E}_j$ relative to $T_{ij}$ as follows. Pick a point $s\in\beta_j$ and a path $\gamma\subset D_{ij}$ connecting $s$ to $c_j$, let $D_{+-}$ be the union of the $S^1$-orbits of the $S^1$-action with weight $(k_j-l_j,0)$, such that when approaching the singular fiber over $c_j$ along $\gamma$, these $S^1$-orbits collapse to points on $Z_j^+\cap Z_j^-\cong\mathbb{C}^\ast$. We can orient $D_{+-}$ so that it equals $[D_-]-[D_+]$ in $H_2(E_j,T_{ij})$, and $\partial D_{+-}$ represents the class $(-k_j,l_j,0)\in H_1(T^3;\mathbb{Z})\cong\mathbb{Z}^3$. By our choice of the local system on $\tau_z$, the holonomy along $\partial D_{+-}$ is $z_2^{k_j}z_1^{-l_j}\in\mathbb{K}^\ast$. Since $\partial D_-$ represents the class $(0,-l_j,1)\in H_1(T^3;\mathbb{Z})$, it follows that the holomorphic sections in classes $[D_-]$ and $[D_+]$ are weighted by $z_1^{l_j}$ and $z_2^{k_j}$, respectively, so the total weights of the $J$-holomorphic sections in $\mathcal{S}(\ring{\pi},S_i,J)$ are given by the product of the polynomials $f_j(z_1,z_2)$, as claimed.
\end{proof}

\begin{figure}
	\centering
	\begin{tikzpicture}
	\draw (0,0) circle [radius=2.5];
	\draw [dashed] (0,0) circle [radius=1.25];
	\draw (1.25,0) node[circle,fill,inner sep=1pt] {};
	\draw (-1.25,0) node[circle,fill,inner sep=1pt] {};
	\draw (0.625,1.0825) node[circle,fill,inner sep=1pt] {};
	\draw (0.625,-1.0825) node[circle,fill,inner sep=1pt] {};
	\draw (-0.625,1.0825) node[circle,fill,inner sep=1pt] {};
	\draw (-0.625,-1.0825) node[circle,fill,inner sep=1pt] {};
	\draw (1.5,0) node {$c_0$};
	\draw (-1.5,0) node {$c_3$};
	\draw (0.8,1.2) node {$c_1$};
	\draw (-0.8,1.2) node {$c_2$};
	\draw (-0.8,-1.2) node {$c_4$};
	\draw (0.8,-1.2) node {$c_5$};
	\draw (0,0) node {$\times$};
	\draw [orange] plot [smooth] coordinates {(-0.6,-0.4) (-0.7,0.1) (-0.5,0.3) (0,0.6) (0.5,0.8) (1,0.8) (1.4,1.4) (0.8,1.8) (0,2) (-0.8,1.8) (-1.4,1.3) (-2,0) (-1.4,-1.3) (0,-2) (0.8,-1.8) (1.4,-1.4) (2,0) (1.8,0.4) (0.8,0.3) (0.5,-0.3) (0,-0.6) (-0.5,-0.5) (-0.6,-0.4)};
	\draw [orange] (2,0.5) node {$\tilde{\gamma}_1$};
	
	\draw (3.5,0) node {$\xrightarrow{\textrm{deform}}$};
	\draw (7,0) circle [radius=2.5];
	\draw [dashed] (7,0) circle [radius=1.25];
	\draw (8.25,0) node[circle,fill,inner sep=1pt] {};
	\draw (5.75,0) node[circle,fill,inner sep=1pt] {};
	\draw (7.625,1.0825) node[circle,fill,inner sep=1pt] {};
	\draw (7.625,-1.0825) node[circle,fill,inner sep=1pt] {};
	\draw (6.375,1.0825) node[circle,fill,inner sep=1pt] {};
	\draw (6.375,-1.0825) node[circle,fill,inner sep=1pt] {};
	\draw (8.5,0) node {$c_0$};
	\draw (5.5,0) node {$c_3$};
	\draw (7.8,1.2) node {$c_1$};
	\draw (6.2,1.2) node {$c_2$};
	\draw (6.2,-1.2) node {$c_4$};
	\draw (7.8,-1.2) node {$c_5$};
	\draw (7,0) node {$\times$};
	\draw [orange] ([shift={(8.25,0)}]-85:0.5) arc[radius=0.5,start angle=-85,end angle=245];
	\draw [orange] ([shift={(7.625,-1.0825)}]-145:0.5) arc[radius=0.5,start angle=-145,end angle=5];
	\draw [orange] ([shift={(7.625,-1.0825)}]35:0.5) arc[radius=0.5,start angle=35,end angle=185];
	\draw [orange] ([shift={(7.625,1.0825)}]-165:0.5) arc[radius=0.5,start angle=-165,end angle=165];
	\draw [orange] ([shift={(6.375,1.0825)}]-145:0.5) arc[radius=0.5,start angle=-145,end angle=5];
	\draw [orange] ([shift={(6.375,1.0825)}]35:0.5) arc[radius=0.5,start angle=35,end angle=185];
	\draw [orange] ([shift={(5.75,0)}]105:0.5) arc[radius=0.5,start angle=105,end angle=255];
	\draw [orange] ([shift={(5.75,0)}]-75:0.5) arc[radius=0.5,start angle=-75,end angle=75];
	\draw [orange] ([shift={(6.375,-1.0825)}]165:0.5) arc[radius=0.5,start angle=165,end angle=315];
	\draw [orange] ([shift={(6.375,-1.0825)}]-15:0.5) arc[radius=0.5,start angle=-15,end angle=135];
	\draw [orange] plot [smooth] coordinates {(8.05,-0.45) (8.1,-0.6) (8.11,-0.7) (8.08,-0.8) (8.02,-0.8)};
	\draw [orange] plot [smooth] coordinates {(8.3,-0.5) (8.2,-0.55) (8.13,-1.05)};
	\draw [orange] plot [smooth] coordinates {(6.85,-1.21) (6.82,-1.3) (7,-1.3) (7.14,-1.2) (7.13,-1.11)};
	\draw [orange] plot [smooth] coordinates {(6.72,-1.44) (6.85,-1.35) (7,-1.37) (7.1,-1.3) (7.22,-1.37)};
	\draw [orange] plot [smooth] coordinates {(5.88,-0.48) (5.8,-0.5) (5.8,-0.7) (5.9,-0.8) (6.02,-0.72)};
	\draw [orange] plot [smooth] coordinates {(5.62,-0.48) (5.65,-0.52) (5.75,-0.8) (5.9,-0.9) (5.88,-0.96)};
	\draw [orange] plot [smooth] coordinates {(5.88,0.49) (5.83,0.51) (5.83,0.7) (5.92,0.83) (5.97,0.8)};
	\draw [orange] plot [smooth] coordinates {(5.88,1.04) (5.88,0.95) (5.7,0.8) (5.7,0.52) (5.61,0.48)};
	\draw [orange] plot [smooth] coordinates {(6.78,1.37) (6.82,1.35) (7,1.25) (7.12,1.15) (7.14,1.22)};
	\draw [orange] plot [smooth] coordinates {(6.87,1.13) (6.87,1.19) (7,1.17) (7.14,0.95)};
	\draw [orange] (8.8,0.5) node {$\tilde{\gamma}_1'$};
	\end{tikzpicture}
	\caption{Deformation of the curve $\tilde{\gamma}_1$} \label{fig:deformation}
\end{figure}

\subsection{Products between two adjacent Lagrangians}\label{section:nearby}

For $i=0,\cdots,n$, consider the Lagrangian submanifolds $L_i,L_{i+1\textrm{ mod }n+1}\subset\ring{W}_f$ defined by parallel transporting the fiberwise Lagrangians $\ell_i\subset\ring{\pi}^{-1}(\zeta_i),\ell_{i+1\textrm{ mod }n+1}\subset\ring{\pi}^{-1}(\zeta_{i+1\textrm{ mod }n+1})$ over the paths $\gamma_i,\gamma_{i+1\textrm{ mod }n+1}\subset\mathbb{C}^\ast$, respectively. It follows from Proposition \ref{proposition:single} that as $\mathbb{K}$-vector spaces (from now on, the index $i+1$ in this subsection will always be understood as $i+1\textrm{ mod }n+1$)
\begin{equation}\label{eq:iden0}
\mathit{CW}^\ast(L_i,L_i)\cong\bigoplus_{(p,q,r)\in\mathbb{Z}^3}\mathbb{K}\cdot\vartheta_{p,q,r}^i,\textrm{ }\mathit{CW}^\ast(L_{i+1},L_{i+1})\cong\bigoplus_{(p,q,r)\in\mathbb{Z}^3}\mathbb{K}\cdot\vartheta_{p,q,r}^{i+1}.
\end{equation}
After wrapping the Lagrangian submanifold $L_i$ using the time-$1$ flow of $X_{H_W}$, we get a well-defined Floer cochain complex $\mathit{CW}^\ast\left(L_i^{(1)},L_{i+1}\right)$ whose generators correspond to the transverse intersection points $L_i^{(1)}\cap L_{i+1}$. Under the projection $\ring{\pi}$, they are mapped to the intersections between the arcs $\gamma_i^{(1)}$ and $\gamma_{i+1}$. We label these intersection points by $\{x_r^-\}_{r\in\mathbb{Z}}\subset\mathbb{C}^\ast$. Above each point $x_r^-$, the Floer complex $\mathit{CW}^\ast\left(L_i^{(1)},L_{i+1}\right)$ has generators corresponding to the transverse intersection points of the fiberwise Lagrangians $\ell_{i,r}^{(1)},\ell_{i+1,r}\subset\ring{\pi}^{-1}(x_r^-)$, where $\ell_{i,r}^{(1)}$ is the time-$1$ flow of $\ell_{i,r}$ under the Hamiltonian vector field $X_{H_F}$. Since $\ring{\pi}^{-1}(x_r^-)\cong T^\ast T^2$ and $\ell_{i,r}^{(1)}$ and $\ell_{i+1,r}$ are Hamiltonian isotopic to cotangent fibers, by our choice of the fiberwise wrapping Hamiltonian $H_F$ in Section \ref{section:setup}, we see that $\ell_{i,r}^{(1)}\cap\ell_{i+1,r}$ is parametrized by $(p,q)\in\mathbb{Z}^2$, and there is a similar identification
\begin{equation}\label{eq:iden1}
\mathit{CW}^\ast\left(L_i^{(1)},L_{i+1}\right)\cong\bigoplus_{(p,q,r)\in\mathbb{Z}^3}\mathbb{K}\cdot\eta_{p,q,r}^i
\end{equation}
as (\ref{eq:iden0}), where the $\eta_{p,q,r}^i$'s have degree $0$. In Figure \ref{fig:prod}, the generators $\eta_{p,q,r}^i$ project under $\ring{\pi}$ to the intersections between the blue arc and the upper edge. Similarly, let $L_{i+1}^{(2)}$ be the time-$2$ flow of $L_{i+1}$ under the Hamiltonian vector field $X_{H_W}$, we have
\begin{equation}
\mathit{CW}^\ast\left(L_{i+1}^{(2)},L_i^{(1)}\right)\cong\bigoplus_{(p,q,r)\in\mathbb{Z}^3}\mathbb{K}\cdot\delta_{p,q,r}^i,
\end{equation}
where the degree $0$ generators $\delta_{p,q,r}^i$ correspond to the fiberwise intersection points that project to $\gamma_{i+1}^{(2)}\cap\gamma_i^{(1)}=\{x_r^+\}_{r\in\mathbb{Z}}$ under $\ring{\pi}$. In Figure \ref{fig:prod}, these correspond to the intersection points between the red arc and the blue arc on the bottom edge. Now we consider the triangle product
\begin{equation}\label{eq:prod1}
\mu^2:\mathit{CW}^\ast\left(L_i^{(1)},L_{i+1}\right)\otimes\mathit{CW}^\ast\left(L_{i+1}^{(2)},L_i^{(1)}\right)\rightarrow\mathit{CW}^\ast\left(L_{i+1}^{(2)},L_{i+1}\right).
\end{equation}
As we have seen in Section \ref{section:single}, the (perturbed) $J$-holomorphic discs contributing to the product are either contained in the fiber or appear as $J$-holomorphic sections of $\ring{\pi}:\ring{W}_f\rightarrow\mathbb{C}^\ast$. In the fiber direction, since both $\ell_i$ and $\ell_{i+1}$ are cotangent fibers of $T^\ast T^2$, the fiberwise Floer complexes $\mathit{CW}^\ast(\ell_i,\ell_{i+1})$ and $\mathit{CW}^\ast(\ell_{i+1},\ell_{i+1})$ together with their triangle products computed in $T^\ast T^2$ can be identified with the Laurent polynomial ring $\mathbb{K}[z_1^{\pm1},z_2^{\pm1}]$. The product (\ref{eq:prod1}) over the fixed triangle in $\mathbb{C}^\ast$ with vertices $x_r^-$, $x_{r'}^+$ and $x_{r+r'}$, where $r,r'\in\mathbb{Z}$, formed by the arcs $\gamma_i^{(1)}$, $\gamma_{i+1}$ and $\gamma_{i+1}^{(2)}$ is therefore identified with a map 
\begin{equation}\label{eq:adgi}
\mathbb{K}[z_1^{\pm1},z_2^{\pm1}]\otimes\mathbb{K}[z_1^{\pm1},z_2^{\pm1}]\rightarrow\mathbb{K}[z_1^{\pm1},z_2^{\pm1}]\xrightarrow{\cdot g_i(z_1,z_2)}\mathbb{K}[z_1^{\pm1},z_2^{\pm1}],
\end{equation}
given by the multiplication of Laurent polynomials in the fiber direction of $\ring{\pi}$, composed with the multiplication of a particular Laurent polynomial $g_i(z_1,z_2)\in\mathbb{K}[z_1^{\pm1},z_2^{\pm1}]$ coming from the weights of $J$-holomorphic sections of $\ring{\pi}$. The expression of $g_i(z_1,z_2)$ depends on the values of $r$ and $r'$, which we will analyze case by case below.

\begin{proposition}\label{propsition:adtrivial}
When $r\ge 0$ and $r'\ge1$ or $r\le-1$ and $r'\le0$, we have $g_i(z_1,z_2)=1$.
\end{proposition}
\begin{proof}
This is similar to the case of $rr'\ge0$ in Proposition \ref{proposition:single2}. In this case, the triangle formed by the arcs $\gamma_i^{(1)}$, $\gamma_{i+1}$ and $\gamma_{i+1}^{(2)}$ does not contain any critical values of $\ring{\pi}$, see Figure \ref{fig:prod1}. It follows that we can shrink such a triangle to a single point by deforming the arcs using a compactly supported isotopy in $\mathbb{C^*}$, which reduces the counting to that inside a smooth fiber of $\ring{\pi}$. This shows that there is precisely one $J$-holomorphic section of $\ring{\pi}$ over this triangle, and the map \ref{eq:adgi} is the usual multiplication of Laurent polynomials. 
\end{proof}

\begin{figure}
	\centering
	\begin{tikzpicture}
	\filldraw [draw=black, color={black!15}] (8.5,1)--(6.5,-1.5)--(8.5,-1.5)--(9.5,1)--(8.5,1);
	\filldraw [draw=black, color={black!15}] (8.5,-1.5)--(9.5,-1.5)--(10.5,1);
	\filldraw [draw=black, color={black!15}] (2.5,-1.5)--(3.5,-1.5)--(4.5,1);
	\draw (0,1) to (11,1);
	\draw (0,-1.5) to (11,-1.5);
	\draw (0,-0.25) to (11,-0.25);
	\draw (11.2,-0.25) node {\small$L_i$};
	\draw (11.3,1) node {\small$L_{i+1}$};
	\draw (5.5,0.375) node[circle,fill,inner sep=1pt] {};
	\draw (5.5,-0.875) node[circle,fill,inner sep=1pt] {};
	\draw [dashed] (5.5,1)--(5.5,-1.5);
	\draw (5.75,0.375) node {\small$c_i$};
	\draw (5.25,-0.875) node {\small$c_\ast$};
	\draw [violet] (4.5,1)  node[circle,fill,inner sep=1pt] {};
	\draw [violet] (10.5,1)  node[circle,fill,inner sep=1pt] {};
	\draw [violet] (2.5,-1.5) node[circle,fill,inner sep=1pt] {};
	\draw [violet] (3.5,-1.5) node[circle,fill,inner sep=1pt] {};
	\draw [violet] (6.5,-1.5) node[circle,fill,inner sep=1pt] {};
	\draw [violet] (8.5,-1.5) node[circle,fill,inner sep=1pt] {};
	\draw [violet] (10.5,1.3) node {\small$x_3^+$};
	\draw [violet] (6.5,1.3) node {\small$x_0^-$};
	\draw [violet] (6.5,-1.8) node {\small$x_0^-$};
	\draw [violet] (8.5,-1.8) node {\small$x_3$};
	\draw [violet] (4.5,1.3) node {\small$x_{0}^+$};
	\draw [violet] (2.5,-1.8) node {\small$x_{-2}^-$};
	\draw [violet] (3.5,-1.8) node {\small$x_{-2}$};
	\draw [blue] (0,0.375) to (0.5,1);
	\draw [blue] (0.5,-1.5) to (2.5,1);
	\draw [blue] (2.5,-1.5) to (4.5,1);
	\draw [blue] (4.5,-1.5) to (6.5,1);
	\draw [blue] (6.5,-1.5) to (8.5,1);
	\draw [blue] (8.5,-1.5) to (10.5,1);
	\draw [blue] (10.5,-1.5) to (11,-0.875);
	\draw [blue] (11.4,-0.875) node {\small$L_i^{(1)}$};
		
	\draw [red] (10.5,-1.5) to (11,-0.25);
	\draw [red] (10.5,1) to (9.5,-1.5);
	\draw [red] (9.5,1) to (8.5,-1.5);
	\draw [red] (8.5,1) to (7.5,-1.5);
	\draw [red] (7.5,1) to (6.5,-1.5);
	\draw [red] (6.5,1) to (5.5,-1.5);
	\draw [red]  (5.5,1) to (4.5,-1.5);
	\draw [red] (4.5,1) to (3.5,-1.5);
	\draw [red] (3.5,1) to (2.5,-1.5);
	\draw [red] (2.5,1) to (1.5,-1.5);
	\draw [red] (1.5,1) to (0.5,-1.5);
	\draw [red] (0.5,1) to (0,-0.25);
	\draw [red] (7.5,1.3) node {\small$L_{i+1}^{(2)}$};
	\end{tikzpicture}
	\caption{The left triangle contributes to the product (\ref{eq:prod1}) between generators over $x_0^+$ and $x_{-2}^-$, the right triangle contributes to the product between generators over $x_3^+$ and $x_0^-$.}\label{fig:prod1}
\end{figure}

\begin{proposition}\label{proposition:gi}
When $r=0$ and $r'\leq0$, we have $g_i(z_1,z_2)=f_i(z_1,z_2)$.
\end{proposition}
\begin{proof}
This is similar to Lemma \ref{lemma:poly}, so the proof will only be sketched. For any $z\in(\mathbb{C}^\ast)^2$, we first replace the Lagrangian submanifolds $L_i^{(1)},L_{i+1},L_{i+1}^{(2)}\subset\ring{W}_f$ with the Lagrangian cylinders $T_{z,i}^{(1)},T_{z,i+1},T_{z,i+1}^{(2)}$ equipped with rank $1$ local systems obtained by parallel transporting the Lagrangian branes $\tau_z\subset\ring{\pi}^{-1}(\zeta_i)$ and $\tau_z\subset\ring{\pi}^{-1}(\zeta_{i+1})$, as specified in the proof of Lemma \ref{lemma:poly}, along the arcs $\gamma_i^{(1)}$, $\gamma_{i+1}$ and $\gamma_{i+1}^{(2)}$, respectively. This transfers the problem of determining the triangle product (\ref{eq:prod1}) to that of computing the product
\begin{equation}\label{eq:p3}
\mu^2:\mathit{CW}^\ast\left(T_{z,i}^{(1)},T_{z,i+1}\right)\otimes\mathit{CW}^\ast\left(T_{z,i+1}^{(2)},T_{z,i}^{(1)}\right)\rightarrow\mathit{CW}^\ast\left(T_{z,i+1}^{(2)},T_{z,i+1}\right).
\end{equation}
By \cite{alp}, Lemma 2.17, the map (\ref{eq:p3}) is then identified with
\begin{equation}
H^\ast(T^2;\mathbb{K})[z_3^{\pm1}]\otimes H^\ast(T^2;\mathbb{K})[z_3^{\pm1}]\rightarrow H^\ast(T^2;\mathbb{K})[z_3^{\pm1}],
\end{equation}
which is the cup product composed with the multiplication by $g_i(z_1,z_2)\in\mathbb{K}$ in the $H^2(T^2;\mathbb{K})$ factor, and the usual polynomial multiplication in the variable $z_3$. To determine the dependence of the weight $g_i(z_1,z_2)$ on $z_1$ and $z_2$, one then deforms the triangle formed by the arcs $\gamma_i^{(1)}$, $\gamma_{i+1}$ and $\gamma_{i+1}^{(2)}$ to a smooth curve $\tilde{\gamma}\subset\mathbb{C}^\ast$. Parallel transporting $\tau_z$ over $\tilde{\gamma}$ then yields a product Lagrangian torus $T_{i,i+1}\subset\ring{W}_f$, equipped with the local system obtained by parallel transporting the one on $\tau_z$. In this case, it follows from our assumption $r=r'=0$ that $\tilde{\gamma}$ encircles a unique critical value $c_i$ of $\ring{\pi}$, see Figure \ref{fig:prod}. It follows from the proof of Lemma \ref{lemma:poly} that there are exactly two $J$-holomorphic sections of $\ring{\pi}:\ring{W}_f\rightarrow\mathbb{C}^\ast$ over the disc bounded by $\tilde{\gamma}$, whose total contribution after weighting by the holonomies along the boundary with respect to the local system on $T_{i,i+1}$ is precisely $z_2^{k_i}\pm z_1^{l_i}$.
\end{proof}

\begin{figure}
	\centering
	\begin{tikzpicture}
	\filldraw [draw=black, color={black!15}] (5.5,1)--(4.5,-1.5)--(6.5,1)--(5.5,1);
	\draw (0,1) to (11,1);
	\draw (0,-1.5) to (11,-1.5);
	\draw (0,-0.25) to (11,-0.25);
	\draw (11.2,-0.25) node {\small$L_i$};
	\draw (11.3,1) node {\small$L_{i+1}$};
	\draw (5.5,0.375) node[circle,fill,inner sep=1pt] {};
	\draw (5.5,-0.875) node[circle,fill,inner sep=1pt] {};
	\draw [dashed] (5.5,1)--(5.5,-1.5);
	\draw (5.7,0.375) node {\small$c_i$};
	\draw (5.25,-0.875) node {\small$c_\ast$};
	\draw [violet] (5.5,1) node[circle,fill,inner sep=1pt] {};
	\draw [violet] (6.5,1) node[circle,fill,inner sep=1pt] {};
	\draw [violet] (4.5,-1.5) node[circle,fill,inner sep=1pt] {};
	\draw [violet] (4.5,-1.8) node {\small$x_0^+$};
	\draw [violet] (5.5,1.2) node {\small$x_0$};
	\draw [violet] (6.5,1.3) node {\small$x_0^-$};
	
	\draw [blue] (0,0.375) to (0.5,1);
	\draw [blue] (0.5,-1.5) to (2.5,1);
	\draw [blue] (2.5,-1.5) to (4.5,1);
	\draw [blue] (4.5,-1.5) to (6.5,1);
	\draw [blue] (6.5,-1.5) to (8.5,1);
	\draw [blue] (8.5,-1.5) to (10.5,1);
	\draw [blue] (10.5,-1.5) to (11,-0.875);
	\draw [blue] (11.4,-0.875) node {\small$L_i^{(1)}$};
	
	\draw [red] (10.5,-1.5) to (11,-0.25);
	\draw [red] (10.5,1) to (9.5,-1.5);
	\draw [red] (9.5,1) to (8.5,-1.5);
	\draw [red] (8.5,1) to (7.5,-1.5);
	\draw [red] (7.5,1) to (6.5,-1.5);
	\draw [red] (6.5,1) to (5.5,-1.5);
	\draw [red]  (5.5,1) to (4.5,-1.5);
	\draw [red] (4.5,1) to (3.5,-1.5);
	\draw [red] (3.5,1) to (2.5,-1.5);
	\draw [red] (2.5,1) to (1.5,-1.5);
	\draw [red] (1.5,1) to (0.5,-1.5);
	\draw [red] (0.5,1) to (0,-0.25);
	\draw [red] (7.5,1.3) node {\small$L_{i+1}^{(2)}$};
	\end{tikzpicture}
    \caption{The triangle product (\ref{eq:prod1}) between generators over $x_0^+$ and $x_0^-$}\label{fig:prod}
\end{figure}

\begin{proposition}\label{proposition:gi'}
When $r=-1, r'\ge1$, we have $g_i(z_1,z_2)=\prod\limits_{j\neq i} f_j(z_1,z_2)$.
\end{proposition}
\begin{proof}
This is similar to to Lemma \ref{lemma:poly} and Proposition \ref{proposition:gi}, except that now the triangle formed by $\gamma_i^{(1)}$, $\gamma_{i+1}$ and $\gamma_{i+1}^{(2)}$ contains all but the critical value $c_i$, $0\leq i\leq n$. See Figure \ref{fig:prod'}. We omit the details.
\end{proof}

\begin{figure}
	\centering
	\begin{tikzpicture}
	\filldraw [draw=black, color={black!15}] (6.5,1)--(4.5,-1.5)--(6.5,-1.5)--(7.5,1)--(6.5,1);
	\filldraw [draw=black, color={black!15}] (6.5,-1.5)--(7.5,-1.5)--(8.5,1);
		
	\draw (0,1) to (11,1);
	\draw (0,-1.5) to (11,-1.5);
	\draw (0,-0.25) to (11,-0.25);
	\draw (11.2,-0.25) node {\small$L_i$};
	\draw (11.3,1) node {\small$L_{i+1}$};
	\draw (5.5,0.375) node[circle,fill,inner sep=1pt] {};
	\draw (5.5,-0.875) node[circle,fill,inner sep=1pt] {};
	\draw [dashed] (5.5,1)--(5.5,-1.5);
	\draw (5.75,0.375) node {\small$c_i$};
	\draw (5.25,-0.875) node {\small$c_\ast$};
	\draw [violet] (8.5,1)  node[circle,fill,inner sep=1pt] {};
	\draw [violet] (4.5,-1.5) node[circle,fill,inner sep=1pt] {};
	\draw [violet] (6.5,-1.5) node[circle,fill,inner sep=1pt] {};
		
	\draw [violet] (8.5,1.3) node {\small$x_2^+$};
	\draw [violet] (6.5,-1.8) node {\small$x_1$};
	\draw [violet] (4.5,-1.8) node {\small$x_{-1}^-$};
		
	\draw [blue] (0,0.375) to (0.5,1);
	\draw [blue] (0.5,-1.5) to (2.5,1);
	\draw [blue] (2.5,-1.5) to (4.5,1);
	\draw [blue] (4.5,-1.5) to (6.5,1);
	\draw [blue] (6.5,-1.5) to (8.5,1);
	\draw [blue] (8.5,-1.5) to (10.5,1);
	\draw [blue] (10.5,-1.5) to (11,-0.875);
	\draw [blue] (11.4,-0.875) node {\small$L_i^{(1)}$};
		
	\draw [red] (10.5,-1.5) to (11,-0.25);
	\draw [red] (10.5,1) to (9.5,-1.5);
	\draw [red] (9.5,1) to (8.5,-1.5);
	\draw [red] (8.5,1) to (7.5,-1.5);
	\draw [red] (7.5,1) to (6.5,-1.5);
	\draw [red] (6.5,1) to (5.5,-1.5);
	\draw [red]  (5.5,1) to (4.5,-1.5);
	\draw [red] (4.5,1) to (3.5,-1.5);
	\draw [red] (3.5,1) to (2.5,-1.5);
	\draw [red] (2.5,1) to (1.5,-1.5);
	\draw [red] (1.5,1) to (0.5,-1.5);
	\draw [red] (0.5,1) to (0,-0.25);
	\draw [red] (7.5,1.3) node {\small$L_{i+1}^{(2)}$};
	\end{tikzpicture}
	\caption{The triangle product (\ref{eq:prod1}) between generators over $x_2^+$ and $x_{-1}^-$}\label{fig:prod'}
\end{figure}

Next, we explain how to compute the triangle products (\ref{eq:prod1}) in the cases when $r\ge0$ and $r'\le0$. We will first regard the generator $\eta_{p,q,r}^i$ as the product between $\vartheta_{0,0,r}^i$ and $\eta_{p,q,0}^{i+1}$ under the map
\begin{equation}
\mu^2:\mathit{CW}^\ast\left(L_{i+1},L_{i+1}\right)\otimes\mathit{CW}^\ast\left(L_i^{(1)},L_{i+1}\right)\rightarrow\mathit{CW}^\ast\left(L_i^{(1)},L_{i+1}\right).
\end{equation}
It follows from Proposition \ref{proposition:gi} and Proposition \ref{proposition:single2} that
\begin{equation}
	\begin{split}
	\mu^2(\eta_{p,q,r}^i,\delta_{p',q',r'}^i)&=\mu^2\left(\mu^2(\vartheta_{0,0,r}^{i+1},\eta_{p,q,0}^i),\delta_{p',q',r'}^i\right) \\
	&=\mu^2\left(\vartheta_{0,0,r}^{i+1},\mu^2(\eta_{p,q,0}^i,\delta_{p',q',r'}^i)\right) \\
	&=f_i(z_1,z_2)\mu^2(\vartheta_{0,0,r}^{i+1},\vartheta_{p+p',q+q',r'}^{i+1}) \\
	&=f_i(z_1,z_2)f(z_1,z_2)^{\min\left\{|r|,|r'|\right\}}\cdot\vartheta_{p+p',q+q',r+r'}^{i+1}.
	\end{split}
\end{equation}

Similarly, we can compute the triangle product (\ref{eq:prod1}) in the cases when $r<-1$ and $r'\ge1$. We regard $\eta_{p,q,r}^i$ as the product between $\vartheta_{0,0,r+1}^i$ and $\eta_{p,q,-1}^{i+1}$ under the map
\begin{equation}
\mu^2:\mathit{CW}^\ast\left(L_{i+1},L_{i+1}\right)\otimes\mathit{CW}^\ast\left(L_i^{(1)},L_{i+1}\right)\rightarrow\mathit{CW}^\ast\left(L_i^{(1)},L_{i+1}\right).
\end{equation}
It follows from Propositions \ref{proposition:single2} and \ref{proposition:gi'} that
\begin{equation}
	\begin{split}
	\mu^2(\eta_{p,q,r}^i,\delta_{p',q',r'}^i)&=\mu^2\left(\mu^2(\vartheta_{0,0,r+1}^{i+1},\eta_{p,q,-1}^i),\delta_{p',q',r'}^i\right) \\
	&=\mu^2\left(\vartheta_{0,0,r+1}^{i+1},\mu^2(\eta_{p,q,-1}^i,\delta_{p',q',r'}^i)\right) \\
	&=\prod\limits_{j\neq i} f_j(z_1,z_2)\cdot\mu^2(\vartheta_{0,0,r+1}^{i+1},\vartheta_{p+p',q+q',r'-1}^{i+1}) \\
	&=\prod\limits_{j\neq i} f_j(z_1,z_2)\cdot f(z_1,z_2)^{\min\left\{|r+1|,|r'-1|\right\}}\cdot\vartheta_{p+p',q+q',r+r'}^{i+1}.
	\end{split}
\end{equation}

\subsection{Concluding the computation}

We start with the following observation, which in particular shows that the product structure $\mu^2$ on the wrapped Fukaya $A_\infty$-algebra $\ring{\mathcal{W}}_f$ is determined by the triangle product for wrapped Floer cochain complexes of two adjacent Lagrangians.

\begin{lemma}\label{lemma:other}
Let $L_i,L_j\subset\ring{W}_f$ be two non-adjacent Lagrangians in the collection $\{L_0,\cdots,L_n\}$, where $i+1<j$ mod $n+1$. Up to the multiplication of a polynomial factor in $\mathbb{K}[z_1,z_2]$, any generator of $\mathit{CW}^\ast(L_i,L_j)$ can be identified with the products of the generators in the Floer cochain complexes $\mathit{CW}^\ast(L_i,L_{i+1}),\cdots,\mathit{CW}^\ast(L_{j-1},L_j)$ of adjacent Lagrangians. 
\end{lemma}
\begin{proof}
Let $L_i^{(2)}$ be the image of $L_i$ under the time-$2$ flow of the Hamiltonian vector field $X_{H_W}$. From our discussions in the previous sections, it should now be clear that we have an identification
\begin{equation}
\mathit{CW}^\ast\left(L_i^{(2)},L_j\right)\cong\bigoplus_{(p,q,r)\in\mathbb{Z}^3}\mathbb{K}\cdot\sigma_{p,q,r}^{i,j}
\end{equation}
between $\mathbb{K}$-vector spaces, where the degree $0$ generators $\sigma_{p,q,r}^{i,j}$ project to $x_r\in\gamma_i^{(2)}\cap\gamma_j$ under $\ring{\pi}:\ring{W}_f\rightarrow\mathbb{C}^\ast$. Consider the triangle product
\begin{equation}\label{eq:prod4}
\mu^2:\mathit{CW}^\ast\left(L_{j-1}^{(1)},L_j\right)\otimes\mathit{CW}^\ast\left(L_i^{(2)},L_{j-1}^{(1)}\right)\rightarrow\mathit{CW}^\ast\left(L_i^{(2)},L_j\right),
\end{equation}
which can be identified with the composition of maps
\begin{equation}
\mathbb{K}[z_1^{\pm1},z_2^{\pm1}]\otimes\mathbb{K}[z_1^{\pm1},z_2^{\pm1}]\xrightarrow{\cdot}\mathbb{K}[z_1^{\pm1},z_2^{\pm1}]\xrightarrow{\cdot g_{ij}(z_1,z_2)}\mathbb{K}[z_1^{\pm1},z_2^{\pm1}],
\end{equation}
where the first one is the usual multiplication between Laurent polynomials coming from the triangle products in the fibers of $\ring{\pi}:\ring{W}_f\rightarrow\mathbb{C}^\ast$ and the second one is the multiplication with a particular Laurent polynomial $g_{ij}(z_1,z_2)$ coming from the counting of $J$-holomorphic sections of $\ring{\pi}$ over the triangle formed by the arcs $\gamma_{j-1}^{(1)}$, $\gamma_j$ and $\gamma_i^{(2)}$. Since for each $r\in\mathbb{Z}$, there exists a triangle in $\mathbb{C}^\ast$ formed by the arcs $\gamma_{j-1}^{(1)}$, $\gamma_j$ and $\gamma_i^{(2)}$ with $x_r$ as one of its vertices, we see that for any $\sigma_{p,q,r}^{i,j}$, the map (\ref{eq:prod}) with output $\sigma_{p,q,r}^{i,j}$ is non-trivial. On the other hand, as we have seen in Sections \ref{section:single} and \ref{section:nearby}, the Laurent polynomial $g_{ij}(z_1,z_2)$ comes from the holonomies of rank $1$ local systems along the boundaries of the $J$-holomorphic sections of $\ring{\pi}$. It from follows our choice of the local systems on the Lagrangian branes $\tau_z\subset\ring{\pi}^{-1}(\zeta_j)$ that negative powers of $z_1$ and $z_2$ do not appear in $g_{ij}(z_1,z_2)$, so we actually have $g_{ij}(z_1,z_2)\in\mathbb{K}[z_1,z_2]$. The proof of the general case now follows by induction.
\end{proof}

The situation for the Floer cochain complexes $\mathit{CW}^\ast(L_i,L_j)$ with $j+1<i$ mod $n+1$ is similar, and is therefore left to the reader.

Since the generators of the wrapped Fukaya $A_\infty$-algebra $\ring{\mathcal{W}}_f$ defined by (\ref{eq:WFA}) are all supported in degree $0$, there are no non-trivial differentials or higher $A_\infty$-structures. Thus our computations of the triangle products in Sections \ref{section:single} and \ref{section:nearby} have already determined the wrapped Fukaya $A_\infty$-algebra $\ring{\mathcal{W}}_f$, as is shown by the following proposition.

\begin{proposition}\label{proposition:conclusion}
The formal $A_\infty$-algebra $\ring{\mathcal{W}}_f$ is isomorphic to the algebra over $\mathbb{K}$ generated by
\begin{equation}
\vartheta_{0,0,0}^i\in\mathit{CW}^\ast(L_i,L_i),\textrm{ }\eta_{0,0,0}^i\in\mathit{CW}^\ast(L_i,L_{i+1}),\textrm{ }\delta_{0,0,0}^i\in\mathit{CW}^\ast(L_{i+1},L_i),
\end{equation}
for $i=0,\cdots,n$ and $z_1^{\pm1},z_2^{\pm1}$, where all generators lie in degree $0$. Moreover, $\vartheta_{0,0,0}^i=e_i\in\mathit{CW}^0(L_i,L_i)$ is an idempotent and the generators $\eta_{0,0,0}^i$ and $\delta_{0,0,0}^i$ satisfy the relations
\begin{equation}\label{eq:rel}
\delta_{0,0,0}^i\cdot\eta_{0,0,0}^i=f_i(z_1,z_2)\vartheta_{0,0,0}^i,\textrm{ }\eta_{0,0,0}^i\cdot\delta_{0,0,0}^i=f_i(z_1,z_i)\vartheta_{0,0,0}^{i+1},
\end{equation}
where $i=0,\cdots,n\textrm{ mod }n+1$.
\end{proposition}
\begin{proof}
By Lemma \ref{lemma:other} above, in order to generate the wrapped Fukaya $A_\infty$-algebra $\ring{\mathcal{W}}_f$ over $\mathbb{K}[z_1^{\pm1},z_2^{\pm1}]$, we only need to consider Floer cochains between adjacent Lagrangians among $L_0,\cdots,L_n\subset\ring{W}_f$ or the self-Floer cochains of the Lagrangians themselves, which are precisely $\eta_{p,q,r}^i$, $\delta_{p,q,r}^i$ and $\vartheta_{p,q,r}^i$ for $i=0,\cdots,n$ and $p,q,r\in\mathbb{Z}$. Using the induction argument as in the proof of Lemma \ref{lemma:other}, it is straightforward to verify the identities
\begin{equation}
\eta_{0,0,0}^i\cdot\eta_{0,0,0}^{i+1}\cdots\eta_{0,0,0}^n\cdot\eta_{0,0,0}^1\cdots\eta_{0,0,0}^{i-1}=\vartheta_{0,0,1}^i,
\end{equation}
\begin{equation}
\delta_{0,0,0}^i\cdot\delta_{0,0,0}^{i-1}\cdots\delta_{0,0,0}^0\cdot\delta_{0,0,0}^n\cdots\delta_{0,0,0}^{i+1}=\vartheta_{0,0,-1}^i,
\end{equation}
so it follows from Proposition \ref{proposition:single2} that all Floer cochains $\{\vartheta_{p,q,r}^i\}_{(p,q,r,i)\in\mathbb{Z}^4}$ are generated by the idempotents $\{\vartheta_{0,0,0}^i\}_{i\in\mathbb{Z}}$ and the concatenations of the Floer cochains in $\{\eta_{0,0,0}^i\}_{i\in\mathbb{Z}}$ and $\{\delta_{0,0,0}^i\}_{i\in\mathbb{Z}}$ over $\mathbb{K}[z_1^{\pm1},z_2^{\pm1}]$. By our discussions at the end of Section \ref{section:nearby}, we also know that the Floer cochains $\eta_{p,q,r}^i$ and $\delta_{p,q,r}^i$ for $(p,q,r,i)\in\mathbb{Z}^4$ are generated by $\eta_{0,0,0}^i$, $\delta_{0,0,0}^i$ and $\vartheta_{0,0,r}^i$ over $\mathbb{K}[z_1^{\pm1},z_2^{\pm1}]$ for $(i,r)\in\mathbb{Z}^2$, so the first claim follows. The relations in (\ref{eq:rel}) follow from Proposition \ref{proposition:gi}.
\end{proof}

\section{Homological mirror symmetry}\label{section:HMS}

In this section, we prove the two versions of homological mirror symmetry (Theorems \ref{theorem:main} and \ref{theorem:HMS1}) between the $6$-dimensional Weinstein manifold $\ring{W}_f$ and the crepant resolution $Y_f$ of the $cA_n$ singularity $R_f$ based on the computations of the wrapped Fukaya category in Section \ref{section:computation}, and then derive from them some corollaries.

\subsection{The Evans-Lekili model}\label{section:EL}

Let $R_f$ be the $cA_n$ singularity defined by (\ref{eq:Rf}). The derived category of coherent sheaves on the crepant resolution $Y_f$ of $\mathrm{Spec}(R_f)$ has a convenient in terms of the symplectic geometry of (punctured) discs with marked points, which we briefly recall here. This is discovered by Evans-Lekili in \cite{evle}. 

Consider the relative wrapped Fukaya category $\mathcal{W}(T^\ast S^1,D_n)$, where $D_n\subset T^\ast S^1$ is a finite set $\{c_0,\cdots,c_n\}$ consisting of the critical values of $\ring{\pi}:\ring{W}_f\rightarrow\mathbb{C}^\ast$. The objects of $\mathcal{W}(T^\ast S^1,D_n)$ are the arcs $\gamma_0,\cdots,\gamma_n$ that the Lagrangian submanifolds $L_0,\cdots,L_n\subset\ring{W}_f$ fiber over, and its morphisms are wrapped Floer cochain complexes $\mathit{CW}^\ast(\gamma_i,\gamma_j)$ computed in the cylinder $T^\ast S^1$. When defining the $A_\infty$-structures on $\mathcal{W}(T^\ast S^1,D_n)$, the counting of holomorphic curves are weighted by the number of times it passes through the points $c_0,\cdots,c_n$. As a consequence, $\mathcal{W}(T^\ast S^1,D_n)$ is an $A_\infty$-category linear over the polynomial ring $R=\mathbb{K}[t_0,\cdots,t_n]$ or its completion $\widehat{R}=\mathbb{K}[\![t_0,\cdots,t_n]\!]$. For a general construction of the relative Fukaya category, see \cite{psc}.

Let $\ring{Q}_n$ be the quiver with $n+1$ vertices and $2n+2$ arrows as depicted in Figure \ref{fig:quiver}. Label its vertices by $0,\cdots,n$, and the arrows by $a_0,b_0,\cdots,a_n,b_n$, where $a_i$ is the arrow starting from the vertex $i$ and ending at the vertex $i+1$, while $b_i$ is the arrow starting at the vertex $i+1$ and ending at the vertex $i$. Consider the path algebra $R\ring{Q}_n$, and the ideal
\begin{equation}\label{eq:IR}
I_R=\langle a_ib_i-t_ie_{i+1},b_ia_i-t_ie_i\vert i=0,\cdots,n\rangle,
\end{equation}
where $e_i$ is an idempotent associated to the vertex $i$. Define the algebra
\begin{equation}
\ring{\mathcal{A}}_n:=R\ring{Q}_n/I_R.
\end{equation}
We regard it as a formal $A_\infty$-algebra concentrated in degree $0$. It follows from \cite{evle}, Proposition 2.2 that we have an equivalence
\begin{equation}\label{eq:1D}
D^\mathit{perf}\mathcal{W}(T^\ast S^1,D_n)\simeq D^\mathit{perf}(\ring{\mathcal{A}}_n)
\end{equation}
between the derived relative wrapped Fukaya category of $(T^\ast S^1,D_n)$ and the derived category of perfect modules over $\ring{\mathcal{A}}_n$.

\begin{figure}
\centering
\begin{tikzpicture}
\draw (1,1.732) node[circle,fill,inner sep=2pt] {};
\draw (2,0) node[circle,fill,inner sep=2pt] {};
\draw (-1,1.732) node[circle,fill,inner sep=2pt] {};
\draw (-2,0) node {$\cdots$};
\draw (1,-1.732) node[circle,fill,inner sep=2pt] {};
\draw (-1,-1.732) node[circle,fill,inner sep=2pt] {};
\draw [->] (1,1.532) to [out=-105,in=150] (1.9,0.1);
\draw [->] (2,0.2) to [out=75,in=-30] (1.1,1.632);
\draw [->] (0.8,1.75) to [out=135,in=45] (-0.8,1.75);
\draw [->] (-0.8,1.72) to [out=-45,in=-135] (0.8,1.72);
\draw [->] (0.8,-1.72) to [out=135,in=45] (-0.8,-1.72);
\draw [->] (-0.8,-1.75) to [out=-45,in=-135] (0.8,-1.75);
\draw [->] (1.2,-1.732) to [out=30,in=-75] (2.1,-0.1);
\draw [->] (1.9,-0.1) to [out=-150,in=105] (1,-1.532);
\draw [->] (-2,-0.2) to [out=-105,in=150] (-1.1,-1.632);
\draw [->] (-0.9,-1.632) to [out=75,in=-30] (-1.8,0);
\draw [->] (-1.9,0.1) to [out=30,in=-75] (-1,1.532);
\draw [->] (-1.1,1.632) to [out=-150,in=105] (-2,0.2);

\draw (1.2,1.9) node {\small$1$};
\draw (2.3,0) node {\small$0$};
\draw (-1.2,1.9) node {\small$2$};
\draw (1.2,-2) node {\small$n$};
\draw (-1.2,-2) node {\small$n-1$};
\draw (0,-2.3) node {\small$a_{n-1}$};
\draw (0,-1.2) node {\small$b_{n-1}$};
\draw (0,2.3) node {\small$a_1$};
\draw (0,1.1) node {\small$b_1$};
\draw (2.1,1.2) node {\small$a_0$};
\draw (1,0.6) node {\small$b_0$};
\draw (2.2,-1.2) node {\small$a_n$};
\draw (1,-0.5) node {\small$b_n$};
\end{tikzpicture}
\caption{The quiver $\ring{Q}_n$}\label{fig:quiver}
\end{figure}

On the mirror side, one considers the versal deformation 
\begin{equation}
\mathrm{Spec}\left(\frac{R[u,v]}{(uv-t_0\cdots t_n)}\right)\rightarrow\mathbb{K}^{n+1}
\end{equation}
of the $A_n$ curve singularity and its toric crepant resolution $Y$. It follows from \cite{evle}, Corollary 4.6 that there is an equivalence
\begin{equation}
D^b\mathit{Coh}(Y)\simeq D^\mathit{perf}(\ring{\mathcal{A}}_n)
\end{equation}
between the bounded derived category of coherent sheaves on $Y$ and the derived category of perfect modules over $\ring{\mathcal{A}}_n$. After the base change
\begin{equation}\label{eq:bc}
R\rightarrow\mathbb{K}[x,y],\textrm{ }t_i\mapsto f_i(x,y),i=0,\cdots,n,
\end{equation}
we get equivalences
\begin{equation}\label{eq:uncom}
D^b\mathit{Coh}(Y_f)\simeq D^\mathit{perf}\left(\ring{\mathcal{A}}_n\otimes_R\mathbb{K}[x,y]\right)\simeq D^\mathit{perf}\mathcal{W}(T^\ast S^1,D_n)\otimes_R\mathbb{K}[x,y].
\end{equation}

One can also consider the completed version of the above story, replacing $R$ with $\widehat{R}$, $\mathbb{K}[x,y]$ with $\mathbb{K}[\![x,y]\!]$, $Y_f$ with $\widehat{Y}_f$, and $\mathcal{W}(T^\ast S^1,D_n)$ with its completion $\widehat{\mathcal{W}}(T^\ast S^1,D_n)$ with respect to the word-length filtration on the morphism spaces. Define
\begin{equation}
\widehat{\ring{\mathcal{A}}}_n:=\widehat{R\ring{Q}}_n/I_R
\end{equation}
to be the quotient of the completed path algebra, then the following analogue of (\ref{eq:uncom}) holds.

\begin{proposition}\label{proposition:complete}
There are equivalences
\begin{equation}\label{eq:deq1}
D^b\mathit{Coh}(\widehat{Y}_f)\simeq D^\mathit{perf}\left(\widehat{\ring{\mathcal{A}}}_n\otimes_{\widehat{R}}\mathbb{K}[\![x,y]\!]\right)\simeq D^\mathit{perf}\widehat{\mathcal{W}}(T^\ast S^1,D_n)\otimes_{\widehat{R}}\mathbb{K}[\![x,y]\!],
\end{equation}
where the left-hand side is the bounded derived category of coherent sheaves on the crepant resolution $\hat{\phi}:\widehat{Y}_f\rightarrow\mathrm{Spec}(\widehat{R}_f)$. 
\end{proposition}
\begin{proof}
The uncompleted version (\ref{eq:uncom}) of the claimed equivalences (\ref{eq:deq1}) is proved in \cite{evle}, Corollary 4.7, under the additional assumption that the crepant resolution $Y_f$ is smooth. 

To see that their argument extends to our case, first notice that the construction of Van den Bergh \cite{VdB} actually works in the complete local setting (cf. \cite{VdB}, Lemma 3.2.9), which produces a tilting bundle $\widehat{\mathcal{T}}\rightarrow\widehat{Y}$ over the crepant resolution $\widehat{Y}$ of $\mathrm{Spec}\left(\frac{R[\![u,v]\!]}{uv-t_0\cdots t_n}\right)\rightarrow\mathrm{Spec}\left(\mathbb{K}[\![t_0,\cdots,t_n]\!]\right)$. By \cite{jkd}, Lemma 2.10, which is applicable in our setting, $\widehat{\mathcal{T}}$ pulls back to a tilting bundle $j^\ast\widehat{\mathcal{T}}\rightarrow\widehat{Y}_f$ under the morphism $j:\widehat{Y}_f\rightarrow\widehat{Y}$, and satisfies
\begin{equation}
D^\mathit{perf}\left(\mathrm{End}(j^\ast\widehat{\mathcal{T}})\right)\simeq D^\mathit{perf}\left(\widehat{\ring{\mathcal{A}}}_n\otimes_{\widehat{R}}\mathbb{K}[\![x,y]\!]\right)\simeq D^\mathit{perf}\widehat{\mathcal{W}}(T^\ast S^1,D_n)\otimes_{\widehat{R}}\mathbb{K}[\![x,y]\!],
\end{equation}
where the equivalences above follow from \cite{evle}, Corollary 4.6, but their argument needs to be modified to include the completion of the path algebra $\widehat{R}\ring{Q}_n$ and the relative wrapped Fukaya category $\mathcal{W}(T^\ast S^1,D_n)$. It remains to prove that $D^b\mathit{Coh}(\widehat{Y}_f)\simeq D^\mathit{perf}\left(\mathit{End}(j^\ast\widehat{\mathcal{T}})\right)$. For this sake, the smoothness assumption on $\widehat{Y}_f$ is unnecessary, since \cite{VdB}, Proposition 3.2.10 is applicable whenever the variety $\widehat{Y}_f$ is normal, which is always true in our case.
\end{proof}

Consider the crepant resolution $\phi:Y_f\rightarrow\mathrm{Spec}(R_f)$, performing Drinfeld localization \cite{vdd} to $D^b\mathit{Coh}(Y_f)$ with respect to the idempotent associated to the structure sheaf $\mathcal{O}_{Y_f}$ gives rise to the relative singularity category $D^b\mathit{Coh}(Y_f)/\langle\mathcal{O}_{Y_f}\rangle$ defined by Kalck-Yang \cite{ky}. On the combinatorial side, this corresponds to removing the vertex $0$ and the arrows starting and ending at $0$ from the quiver $\ring{Q}_n$ in Figure \ref{fig:quiver}, and adding two loops $\alpha$ and $\beta$ at the vertices $1$ and $n$, respectively. This gives rise to the new quiver
\begin{equation}
\begin{tikzcd}
\bullet \arrow["\alpha"', loop, distance=2em, in=215, out=145] \arrow[r, "a_1", bend left] & \bullet \arrow[l, "b_1", bend left] \arrow[r, "a_2", bend left] & \cdots \arrow[l, "b_2", bend left] \arrow[r, "a_n", bend left] & \bullet \arrow[l, "b_n", bend left] \arrow["\beta"', loop, distance=2em, in=35, out=325]
\end{tikzcd}
\end{equation}
which we denote by $Q_n$. Define $\mathcal{A}_n$ to be dg algebra whose underlying associative algebra is the path algebra $RQ_n$ modulo the relations in the ideal $I_R$ given by (\ref{eq:IR}) and
\begin{equation}
\alpha^2=\beta^2=0.
\end{equation}
The differential $d:\mathcal{A}_n\rightarrow\mathcal{A}_n[1]$ is determined by
\begin{equation}
\begin{split}
da_i&=db_i=0,i=1,\cdots,n, \\
d\alpha&=t_0e_1,\textrm{ }d\beta=t_ne_n.
\end{split}
\end{equation}
In the above, the generators are graded as follows:
\begin{equation}
|a_i|=|b_i|=0,i=0,\cdots,n,\textrm{ }|\alpha|=|\beta|=-1.
\end{equation}
After the base change $R\rightarrow\mathbb{K}[x,y]$ specified by (\ref{eq:bc}), the dg algebra $\left(\mathcal{A}_n\otimes_R\mathbb{K}[x,y],d\right)$ is quasi-isomorphic to the \textit{derived contraction algebra} studied by Booth \cite{mbn}, associated to the crepant resolution $\phi:Y_f\rightarrow\mathrm{Spec}(R_f)$, and $H^0\left(\mathcal{A}_n\otimes_R\mathbb{K}[x,y]\right)$ gives the \textit{contraction algebra} introduced by Donovan-Wemyss \cite{dw}.

Switching to the symplectic perspective, the localization corresponds to excising the arc $\gamma_0\subset\mathbb{C}^\ast$, which (after deformation) gives rise to the closed disc with $n+1$ marked points $D_n=\{c_0',\cdots,c_n'\}$, see Figure \ref{fig:base1}. It follows from \cite{gps}, Proposition 11.2 that the localized $A_\infty$-category $\mathcal{W}(T^\ast S^1,D_n)/\langle\gamma_0\rangle$ is quasi-equivalent to the relative wrapped Fukaya category $\mathcal{W}(\mathbb{C},D_n)$. It is proved in \cite{evle} that we have an equivalence
\begin{equation}\label{eq:1Dl}
D^\mathit{perf}\mathcal{W}(\mathbb{C},D_n)\simeq D^\mathit{perf}(\mathcal{A}_n),
\end{equation}
from which the mirror equivalence
\begin{equation}\label{eq:HMS1d}
D^\mathit{perf}\mathcal{W}(\mathbb{C},D_n)\otimes_R\mathbb{K}[x,y]\simeq D^b\mathit{Coh}(Y_f)/\langle\mathcal{O}_{Y_f}\rangle
\end{equation}
follows.\footnote{Strictly speaking, \cite{evle} only proves the complete local version of (\ref{eq:HMS1d}), namely $D^\mathit{perf}\mathcal{W}(\mathbb{C},D_n)\otimes_{\widehat{R}}\mathbb{K}[\![x,y]\!]\simeq D^b\mathit{Coh}(\widehat{Y}_f)/\langle\mathcal{O}_{\widehat{Y}_f}\rangle$. In order to get the algebraic version, one needs to prove the split-generation of $\mathcal{W}(T^\ast S^1,D_n)$ by $\gamma_0,\cdots,\gamma_n$ over $R=\mathbb{K}[t_0,\cdots,t_n]$, without passing to the completion. This is done in \cite{dmr}. See also \cite{evle}, Remark 1.8.}

\begin{figure}
	\centering
	\begin{tikzpicture}
	\draw (0,0) circle [radius=2.5];
	\draw (2,0) node[circle,fill,inner sep=1pt] {};
	\draw (1.2,0) node[circle,fill,inner sep=1pt] {};
	\draw (0.4,0) node[circle,fill,inner sep=1pt] {};
	\draw (-0.4,0) node[circle,fill,inner sep=1pt] {};
	\draw (-1.2,0) node[circle,fill,inner sep=1pt] {};
	\draw (-2,0) node[circle,fill,inner sep=1pt] {};
	\draw (-2,-0.25) node {$c_0'$};
	\draw (-1.2,-0.25) node {$c_1'$};
	\draw (-0.4,-0.25) node {$c_2'$};
	\draw (0.4,-0.25) node {$c_3'$};
	\draw (1.2,-0.25) node {$c_4'$};
	\draw (2,-0.25) node {$c_5'$};
	\draw [teal] (-2,0)--(2,0);
	\draw [orange] (0,-2.5)--(0,2.5);
	\draw [orange] (-0.8,-2.35)--(-0.8,2.35);
	\draw [orange] (0.8,-2.35)--(0.8,2.35);
	\draw [orange] (-1.6,-1.9)--(-1.6,1.9);
	\draw [orange] (1.6,-1.9)--(1.6,1.9);
	\draw [orange] (-1.6,-2.2) node {$\gamma_1$};
	\draw [orange] (-0.8,-2.6) node {$\gamma_2$};
	\draw [orange] (0,-2.8) node {$\gamma_3$};
	\draw [orange] (0.8,-2.6) node {$\gamma_4$};
	\draw [orange] (1.6,-2.2) node {$\gamma_5$};
	\end{tikzpicture}
\caption{Base of the Morse-Bott fibration $\pi:W_f\rightarrow\mathbb{C}$ when $n=5$. The arcs $\gamma_1,\cdots,\gamma_5$ split-generate the relative wrapped Fukaya category $\mathcal{W}(\mathbb{C},D_5)$.} \label{fig:base1}
\end{figure}

\subsection{Proofs of the results}\label{section:proof}

In this subsection we prove the results related to homological mirror symmetry stated in Section \ref{section:hms}, which are Theorem \ref{theorem:main}, Corollary \ref{corollary:core-A}, Theorem \ref{theorem:HMS1}, Proposition \ref{proposition:eq}, Theorem \ref{theorem:WAn}, Proposition \ref{proposition:eq2} and Corollary \ref{corollary:core}.

\subsubsection{Theorem \ref{theorem:main}, Corollary \ref{corollary:core-A}, Theorem \ref{theorem:HMS1} and Proposition \ref{proposition:eq}}

We first prove Theorem \ref{theorem:main}, namely there is an equivalence
\begin{equation}
D^\mathit{perf}\mathcal{W}(\ring{W}_f;\mathbb{K})\simeq D^b\mathit{Coh}(\ring{Y}_f).
\end{equation}
Our strategy here is to relate both sides with the quiver algebra $\ring{\mathcal{A}}_n$ introduced in Section \ref{section:EL}.

\begin{proposition}
There is an isomorphism
\begin{equation}
\ring{\mathcal{W}}_f\cong\ring{\mathcal{A}}_n\otimes_R\mathbb{K}[x^{\pm1},y^{\pm1}]
\end{equation}
where the base change is given by (\ref{eq:bc}).
\end{proposition}
\begin{proof}
This is a direct consequence of our computation of the wrapped Fukaya $A_\infty$-algebra $\ring{\mathcal{W}}_f=\bigoplus_{0\leq i,j\leq n}\mathit{CW}^\ast(L_i,L_j)$. By Proposition \ref{proposition:conclusion}, $\ring{\mathcal{W}}_f$ is the formal $A_\infty$-algebra generated by $\vartheta_{0,0,0}^i$, $\eta_{0,0,0}^i$, $\delta_{0,0,0}^i$, $z_1^{\pm1}$ and $z_2^{\pm1}$, with $(\vartheta_{0,0,0}^i)^2=\vartheta_{0,0,0}^i$ and the relations (\ref{eq:rel}). An isomorphism $\ring{\psi}:\ring{\mathcal{W}}_f\rightarrow\ring{\mathcal{A}}_n\otimes_R\mathbb{K}[x^{\pm1},y^{\pm1}]$ is defined by
\begin{equation}\label{eq:psi}
\begin{split}
&\ring{\psi}(\vartheta_{0,0,0}^i)=e_i,\textrm{ }\ring{\psi}(\eta_{0,0,0}^i)=a_i,\textrm{ }\ring{\psi}(\delta_{0,0,0}^i)=b_i,\textrm{ }i=0,\cdots,n, \\
&\ring{\psi}(z_1)=x,\textrm{ }\ring{\psi}(z_2)=y.
\end{split}
\end{equation}
\end{proof}

To prove Theorem \ref{theorem:main}, it remains to show that

\begin{proposition}
There is an equivalence
\begin{equation}
D^b\mathit{Coh}(\ring{Y}_f)\simeq D^\mathit{perf}\left(\ring{\mathcal{A}}_n\otimes_R\mathbb{K}[x^{\pm1},y^{\pm1}]\right).
\end{equation}
\end{proposition}
\begin{proof}
By (\ref{eq:uncom}), we already have an equivalence $D^b\mathit{Coh}(Y_f)\simeq D^\mathit{perf}\left(\ring{\mathcal{A}}_n\otimes_R\mathbb{K}[x,y]\right)$. Removing the divisor $D\subset Y_f$ amounts to applying a categorical localization to $D^b\mathit{Coh}(Y_f)$. Since the divisor $D$ is the pullback of $\{xy=0\}\subset\{uv=f(x,y)\}$ under the crepant resolution, under the identification between the endomorphism algebra of the tilting bundle $\mathcal{T}$ on $Y_f$ and $\ring{\mathcal{A}}_n\otimes_R\mathbb{K}[x,y]$ in the proof of \cite{evle}, Corollary 4.7, such a localization precisely inverts the variables $x$ and $y$.
\end{proof}

We can then deduce Corollary \ref{corollary:core-A} from Theorem \ref{theorem:main}, which says that there is a fully faithful embedding
\begin{equation}\label{eq:core-A}
D^b\mathit{Coh}_C(\ring{Y}_f)\hookrightarrow D^\mathit{perf}\mathcal{F}(\ring{W}_f;\mathbb{K}).
\end{equation}

\begin{proof}[Proof of Corollary \ref{corollary:core-A}]
Let $\mathcal{Q}(\ring{W}_f;\mathbb{K})\subset\mathcal{F}(\ring{W}_f;\mathbb{K})$ be the full $A_\infty$-subcategory split-generated by the Lagrangians $Q_0,\cdots,Q_n\subset\ring{W}_f$. It follows from our discussions in Section \ref{section:CE} that $\mathcal{Q}(\ring{W}_f;\mathbb{K})$ can be recovered from the wrapped Fukaya category $\mathcal{W}(\ring{W}_f;\mathbb{K})$ via the generalized Eilenberg-Moore equivalence (\ref{eq:GEM}). On the other hand, as we have mentioned in Remark \ref{remark:MS}, Ballard proved in \cite{mbd} that $D^b\mathit{Coh}_C(\ring{Y}_f)$ is related to the ambient category $D^b\mathit{Coh}(\ring{Y}_f)$ via the Eilenberg-Moore equivalence. Corollary \ref{corollary:core-A} now follows from Theorem \ref{theorem:main}.
\end{proof}

We now prove Theorem \ref{theorem:HMS1}, which asserts that there is an equivalence
\begin{equation}\label{eq:HMS}
D^\mathit{perf}\widehat{\mathcal{W}}(\ring{W}_f;\mathbb{K})\simeq D^b\mathit{Coh}(\widehat{Y}_f).
\end{equation}
By Proposition \ref{proposition:complete}, in order to prove Theorem \ref{theorem:HMS1} it suffices to identify the completion of $\ring{\mathcal{W}}_f$ (with respect to the word length filtration) with the algebra $\widehat{\ring{\mathcal{A}}}_n\otimes_{\widehat{R}}\mathbb{K}[\![x,y]\!]$. This is done in the following proposition.

\begin{proposition}\label{proposition:qis}
We have an isomorphism
\begin{equation}\label{eq:qis}
\widehat{\ring{\mathcal{W}}}_f\cong\widehat{\ring{\mathcal{A}}}_n\otimes_{\widehat{R}}\mathbb{K}[\![x,y]\!],
\end{equation}
where $\widehat{\ring{\mathcal{W}}}_f$ is the completed wrapped Fukaya $A_\infty$-algebra.
\end{proposition}
\begin{proof}
Define a map $\hat{\ring{\psi}}:\ring{\mathcal{W}}_f\rightarrow\widehat{\ring{\mathcal{A}}}_n\otimes_{\widehat{R}}\mathbb{K}[\![x,y]\!]$ by (compare with (\ref{eq:psi}))
\begin{equation}\label{eq:map1}
\hat{\ring{\psi}}(\vartheta_{0,0,0}^i)=e_i,\textrm{ }\hat{\ring{\psi}}(\eta_{0,0,0}^i)=a_i,\textrm{ }\hat{\ring{\psi}}(\delta_{0,0,0}^i)=b_i,\textrm{ }i=0,\cdots,n,
\end{equation}
\begin{equation}\label{eq:map2}
\hat{\ring{\psi}}(z_1)=x-1,\textrm{ }\hat{\ring{\psi}}(z_2)=y-1,\textrm{ }\hat{\ring{\psi}}(z_1^{-1})=-\sum_{j=0}^\infty x^j,\textrm{ }\hat{\ring{\psi}}(z_2^{-1})=-\sum_{j=0}^\infty y^j.
\end{equation}
It is straightforward to check that $\hat{\ring{\psi}}$ is an algebra map. Passing to the completion $\widehat{\ring{\mathcal{W}}}_f$ of $\ring{\mathcal{W}}_f$ cancels the negative powers of $z_1$ and $z_2$, since $z_1^{-1}=-\sum_{j=0}^\infty(z_1+1)^j$ and $z_2^{-1}=-\sum_{j=0}^\infty(z_2+1)^j$ in $\widehat{\ring{\mathcal{W}}}_f$. It follows that the map $\hat{\ring{\psi}}$ is an isomorphism.
\end{proof}

We now prove Proposition \ref{proposition:eq}, which claims the equivalence
\begin{equation}\label{eq:eq1}
D^\mathit{perf}\mathcal{W}(T^\ast S^1,D_n)\otimes_{\widehat{R}}\mathbb{K}[\![x,y]\!]\simeq D^\mathit{perf}\mathcal{W}(\ring{W}_f;\mathbb{K})[\![x,y]\!]
\end{equation}
between the base changed relative wrapped Fukaya category of $(T^\ast S^1,D_n)$ and the derived wrapped Fukaya category of $\ring{W}_f$ completed at the value $-1$ of the Seidel elements.

Recall that the wrapped Fukaya category $\mathcal{W}(\ring{W}_f;\mathbb{K})$ is linear over the $\mathbb{K}$-algebra $\mathit{SH}^0(\ring{W}_f;\mathbb{K})$ given by degree $0$ symplectic cohomology together with its pair-of-pants product. In our case, there is a Hamiltonian $T^2$-action on $\ring{W}_f$ rotating the fibers of the Morse-Bott fibration $\ring{\pi}:\ring{W}_f\rightarrow\mathbb{C}^\ast$, see Section \ref{section:plumbing}. 

\begin{lemma}\label{lemma:vertical}
For any circle $S^1\subset T^2$, the corresponding fiberwise Hamiltonian circle action gives rise to a Seidel element $s\in\mathit{SH}^0(\ring{W}_f;\mathbb{K})^\times$, which is represented by a Hamiltonian orbit in the fiber $\ring{\pi}^{-1}(\ast)\cong T^\ast T^2$.
\end{lemma}
\begin{proof}
Consider the fiberwise Hamiltonian $H_F:\ring{W}_f\rightarrow[0,\infty)$ defined in Section \ref{section:setup}, using which we can define a vertical version of the symplectic cohomology
\begin{equation}
\mathit{SH}^\ast_\mathit{vert}(\ring{W}_f;\mathbb{K}):=\mathit{HF}^\ast(H_F;\mathbb{K}),
\end{equation}
where the right-hand side is the Hamiltonian Floer cohomology of $H_F$. Consider the Hamiltonian orbit generated by the fiberwise circle action, it defines a cochain $x\in\mathit{CF}^0(H_F;\mathbb{K})$. Since all the Floer trajectories asymptotic to orbits of $X_{H_F}$ are contained inside a single fiber $F=\ring{\pi}^{-1}(\ast)$, it follows that $x$ is a cocycle and therefore gives rise to a class $[x]\in\mathit{SH}^0_\mathit{vert}(\ring{W}_f;\mathbb{K})$. We may assume that $F$ is smooth, so it is symplectomorphic to $T^\ast T^2$. It follows that $[x]\in\mathit{SH}^0_\mathit{vert}(\ring{W}_f;\mathbb{K})^\times$. Finally, we have a continuation map
\begin{equation}
\mathit{SH}^\ast_\mathit{vert}(\ring{W}_f;\mathbb{K})\rightarrow\mathit{SH}^\ast(\ring{W}_f;\mathbb{K})
\end{equation}
preserving the pair-of-pants product, and the Seidel element $s$ is the image of $[x]$ under this continuation map.
\end{proof}

By Lemma \ref{lemma:vertical}, after fixing a basis of $H_1(T^2;\mathbb{Z})$, the fiberwise Hamiltonian $T^2$-action on $\ring{\pi}:\ring{W}_f\rightarrow\mathbb{C}^\ast$ gives rise to two Seidel elements $s_1,s_2\in\mathit{SH}^0(\ring{W}_f;\mathbb{K})^\times$. In particular, $\mathcal{W}(\ring{W}_f;\mathbb{K})$ is an $A_\infty$-category over $\mathbb{K}[s_1^{\pm1},s_2^{\pm1}]$. The $A_\infty$-category $\mathcal{W}(\ring{W}_f;\mathbb{K})[\![x,y]\!]$ on the right-hand side of (\ref{eq:eq1}) is the completion of $\mathcal{W}(\ring{W}_f;\mathbb{K})$ at $s_1=s_2=-1$. The precise statement of Proposition \ref{proposition:eq} is the following.

\begin{proposition}\label{proposition:precise}
There exists a choice of a basis of $H_1(T^2;\mathbb{Z})$, such that after taking the completion of $\mathcal{W}(\ring{W}_f;\mathbb{K})$ with respect to the corresponding Seidel elements $s_1,s_2\in\mathit{SH}^0(\ring{W}_f;\mathbb{K})^\times$ at $s_1=s_2=-1$, the equivalence (\ref{eq:eq1}) holds.
\end{proposition}

To prove Proposition \ref{proposition:precise}, we start with the following observation.

\begin{lemma}\label{lemma:Seidel}
For an appropriate choice of the basis of $H_1(T^2;\mathbb{Z})$, the Seidel elements $s_1,s_2\in\mathit{SH}^0(\ring{W}_f;\mathbb{K})^\times$ associated to the Hamiltonian $T^2$-action correspond precisely to the generators $z_1$ and $z_2$ in the wrapped Fukaya $A_\infty$-algebra $\ring{\mathcal{W}}_f$.
\end{lemma}
\begin{proof}
The Seidel representation (\ref{eq:SeiRep}), when applied to the Weinstein manifold $\ring{W}_f$, has its image in $\mathit{SH}^0(\ring{W}_f;\mathbb{K})^\times$. We have group isomorphisms
\begin{equation}\label{eq:comp}
\mathit{SH}^0(\ring{W}_f;\mathbb{K})^\times\cong\mathit{HH}^0\left(\mathcal{W}(\ring{W}_f;\mathbb{K})\right)^\times\cong Z(\ring{\mathcal{W}}_f)^\times,
\end{equation}
where $\mathit{HH}^0\left(\mathcal{W}(\ring{W}_f;\mathbb{K})\right)$ is the zeroth Hochschild cohomology of the wrapped Fukaya category and $Z(\ring{\mathcal{W}}_f)$ is the center of the (formal) $A_\infty$-algebra $\ring{\mathcal{W}}_f$. In the above, the first isomorphism follows from the fact that the closed-open string map is an isomorphism for Weinstein manifolds \cite{sgs}, while the second isomorphism is due to the fact that $\ring{\mathcal{W}}_f$ is concentrated in degree $0$. By Proposition \ref{proposition:conclusion}, $Z(\ring{\mathcal{W}}_f)$ is generated by $z_1^{\pm1}$ and $z_2^{\pm1}$. It follows from the isomorphism (\ref{eq:comp}) and the existence of two (linearly independent) Seidel elements that $\mathit{SH}^0(\ring{W}_f;\mathbb{K})^\times$ is also generated by $s_1^{\pm1}$ and $s_2^{\pm1}$. We choose a basis of $H^1(T^2;\mathbb{Z})$ so that the corresponding orbits of the Hamiltonian $S^1$-actions in the fiber $\ring{\pi}^{-1}(\zeta_i)\cong T^\ast T^2$ are homologous to the Hamiltonian chords with ends on $\ell_i\subset\ring{\pi}^{-1}(\zeta_i)$ representing $z_1$ and $z_2$ for any $0\leq i\leq n$. Since the symplectic parallel transport is $T^2$-equivariant, there is a map $u:\mathbb{D}\setminus\{0,1\}\rightarrow(\ring{W}_f,L_i)$, with asymptotics at $0$ and $1$ given by $s_j$ and $z_j$, respectively, contributing to the closed-open map $\mathit{SH}^\ast(\ring{W}_f;\mathbb{K})\rightarrow\mathit{HW}^\ast(L_i,L_i)$, whose image is contained in a fiber of $\ring{\pi}$. Since the fiber is symplectomorphic to $T^\ast T^2$, by our definition of the wrapping Hamiltonian $H_W$, there is precisely one such map $u$ for $j=1,2$, which gives the isomorphism $\mathbb{K}[s_1^{\pm1},s_2^{\pm1}]\xrightarrow{\cong}\mathbb{K}[z_1^{\pm1},z_2^{\pm1}]$ from the subalgebra of $\mathit{SH}^0(\ring{W}_f;\mathbb{K})$ generated by $s_1^{\pm1}$ and $s_2^{\pm1}$ to the fiberwise wrapped Floer cohomology $\mathit{HW}^0(\ell_i,\ell_i)$, mapping $s_j$ to $z_j$.
\end{proof}

\begin{proof}[Proof of Proposition \ref{proposition:precise}]
It follows from the proof of Proposition \ref{proposition:qis} that we have an algebra map $\ring{\psi}':\ring{\mathcal{W}}_f\rightarrow\widehat{R}\ring{Q}_n/I_R$, defined in the same way as $\hat{\ring{\psi}}$, see (\ref{eq:map1}) and (\ref{eq:map2}). Notice the difference here: we do not take the completion of the path algebra of $\ring{Q}_n$ in the target of the map $\ring{\psi}'$, only the coefficient ring is completed from $R$ to $\widehat{R}$. By Lemma \ref{lemma:Seidel} above, it suffices to show that $\ring{\psi}'$ induces an isomorphism after passing to the completion of $\ring{\mathcal{W}}_f$ at $z_1=z_2=-1$, but this follows from the simple observation that such a completion is enough to invert $z_1$ and $z_2$.
\end{proof}

\subsubsection{Theorem \ref{theorem:WAn} and Proposition \ref{proposition:eq2}}\label{section:68}

We now proceed to prove Theorem \ref{theorem:WAn}, which says that under the assumption $(k_0,\pm l_0)\neq(k_n,\pm l_n)$ (i.e. $Q_0\ncong S^1\times S^2$), there is an equivalence
\begin{equation}
D^\mathit{perf}\mathcal{W}(W_f;\mathbb{K})\simeq D^b\mathit{Coh}(Y_f)/\langle\mathcal{O}_{Y_f}\rangle
\end{equation}
between the derived wrapped Fukaya category of the multi-bubble plumbing associated to the data provided by the factorization of $f(z_1,z_2)$ and the relative singularity category of the corresponding crepant resolution of the $cA_n$ singularity $R_f$.

By \cite{lllm}, Lemma 4.5, the non-compact Lagrangian submanifold $L_0\subset\ring{W}_f$ admits a Weinstein neighborhood $U(L_0)$, which is symplectomorphic to an open neighborhood of $L_0$ in $T^\ast L_0$. Since $L_0$ is the cocore of a Weinstein $3$-handle, removing $U(L_0)$ from $\ring{W}_f$ gives rise to a Weinstein domain $\overline{\ring{W}}_f\setminus U(L_0)$.

\begin{lemma}\label{lemma:ex}
The Weinstein domain $\overline{\ring{W}}_f\setminus U(L_0)$ is symplectically equivalent to the multi-bubble plumbing $\overline{W}_f$, with an additional $1$-handle attached.
\end{lemma}
\begin{proof}
By construction, up to Weinstein deformation $\overline{\ring{W}}_f\setminus U(L_0)$ is the total space of the Morse-Bott fibration $\bar{\pi}:\overline{W}_f\rightarrow D$ with base a closed disc $D$ obtained by excising $\gamma_0$ from $D^\ast S^1$, $n+1$ critical values $c_0,\cdots,c_n\in D$ and vanishing cycles specified as before by (\ref{eq:kl}), together with a $3$-dimensional $1$-handle attached to $\partial\overline{W}_f$, which is also fibered over a $1$-dimensional $1$-handle attached to $\partial D$ at two distinct points (to make it the annulus $D^\ast S^1$), with fiber being the product $D\times D$ of two closed discs. After smoothing the corners and passing to the completion, it follows from Lemma \ref{lemma:plumbing1} that the total space $W_f$ of the Morse-Bott fibration $\pi:W_f\rightarrow\mathbb{C}$ is symplectically equivalent to the multi-bubble plumbing of cotangent bundles $T^\ast Q_1,\cdots,T^\ast Q_n$ according to the data provided by (\ref{eq:f}).
\end{proof}

Since subcritical handle attachment does not affect the wrapped Fukaya category of Weinstein manifolds up to quasi-equivalence \cite{fbw,gps,kih}, it follows from Lemma \ref{lemma:ex} above and \cite{gps}, Proposition 11.2 that we have a quasi-equivalence
\begin{equation}
\mathcal{W}(W_f;\mathbb{K})\simeq\mathcal{W}(\ring{W}_f;\mathbb{K})/\langle L_0\rangle,
\end{equation}
where the right-hand side is the Drinfeld localization of $\mathcal{W}(\ring{W}_f;\mathbb{K})$ with respect to the idempotent $\vartheta_{0,0,0}^0\in\mathit{CW}^0(L_0,L_0)$. By the equivalence (\ref{eq:HMS1d}), in order to prove Theorem \ref{theorem:WAn}, it suffices to establish the quasi-equivalence
\begin{equation}\label{eq:loceq}
\mathcal{W}(\ring{W}_f;\mathbb{K})/\langle L_0\rangle\simeq\mathcal{W}(T^\ast S^1,D_n)\otimes_R\mathbb{K}[x,y]/\langle\gamma_0\rangle
\end{equation}
between the localized wrapped Fukaya categories when $Q_0$ is not diffeomorphic to $S^1\times S^2$. Note that the completed version of (\ref{eq:loceq}), namely
\begin{equation}
\mathcal{W}(\ring{W}_f;\mathbb{K})[\![x,y]\!]/\langle L_0\rangle\simeq\mathcal{W}(T^\ast S^1,D_n)\otimes_R\mathbb{K}[\![x,y]\!]/\langle\gamma_0\rangle,
\end{equation}
is obviously true, since by Proposition \ref{proposition:precise} above, both sides are the same localizations of equivalent categories. In view of the equivalence (\ref{eq:1Dl}), in order to prove (\ref{eq:loceq}), it suffices to prove the following.

\begin{proposition}\label{proposition:loc}
Suppose that $(k_0,\pm l_0)\neq(k_n,\pm l_n)$, there is a quasi-equivalence
\begin{equation}
\mathcal{W}(\ring{W}_f;\mathbb{K})/\langle L_0\rangle\simeq\left(\mathcal{A}_n\otimes_R\mathbb{K}[x,y]\right)^\mathit{perf},
\end{equation}
where the right-hand side is the dg category of perfect modules over the dg algebra $\mathcal{A}_n\otimes_R\mathbb{K}[x,y]$, with the base change given by (\ref{eq:bc}).
\end{proposition}
\begin{proof}
From the definition of Drinfeld localization, we see that the localization of the endomorphism algebra $\ring{\mathcal{W}}_f$ of $\mathcal{W}(\ring{W}_f;\mathbb{K})$ with respect to the idempotent $\vartheta_{0,0,0}^0$ associated to $L_0$ is quasi-isomorphic to the dg algebra generated by $\vartheta_{0,0,0}^i$, $\eta_{0,0,0}^i$, $\delta_{0,0,0}^i$ for $i=1,\cdots,n$, $z_1^{\pm1}$, $z_2^{\pm1}$, which have degree $0$ and satisfying $(\vartheta_{0,0,0}^i)^2=\vartheta_{0,0,0}^i$ and (\ref{eq:rel}), together with two additional generators $\varepsilon_0$ and $\varepsilon_n$ coming from the localization, with grading $-1$, satisfying
\begin{equation}\label{eq:f1}
\varepsilon_0^2=\varepsilon_n^2=0,
\end{equation}
and having the differentials
\begin{equation}\label{eq:f2}
d\varepsilon_0=f_0(z_1,z_2)\vartheta_{0,0,0}^1,\textrm{ }d\varepsilon_n=f_n(z_1,z_2)\vartheta_{0,0,0}^n.
\end{equation}
We can define a map $\psi$ from $\mathcal{A}_n\otimes_R\mathbb{K}[x,y]$ to the localization of $\ring{\mathcal{W}}_f$ by
\begin{equation}
\psi(e_i)=\vartheta_{0,0,0}^i,\textrm{ }\psi(a_i)=\eta_{0,0,0}^i,\textrm{ }\psi(b_i)=\delta_{0,0,0}^i,\textrm{ }i=1,\cdots,n,
\end{equation}
\begin{equation}
\psi(x)=z_1+1,\textrm{ }\psi(y)=z_2+1.
\end{equation}
By Proposition \ref{proposition:conclusion} and the formulae (\ref{eq:f1}) and (\ref{eq:f2}) above, it is straightforward to verify that $\psi$ is a dg algebra map. To show that it is a quasi-isomorphism, it suffices to show that the generators $z_1^{-1}$ and $z_2^{-1}$ in the Drinfeld localization of $\ring{\mathcal{W}}_f$ are redundant, namely they can be generated by the non-negative powers of $z_1$ and $z_2$. To see this, we make use of (\ref{eq:f2}), which gives the relations
\begin{equation}\label{eq:rel1}
z_1^{l_0}\pm z_2^{k_0}=z_1^{l_n}\pm z_2^{k_n}=0.
\end{equation}
We now need the assumption that $(k_0,\pm l_0)\neq(k_n,\pm l_n)$. To prove that $z_1^{-1}$ and $z_2^{-1}$ are redundant as generators, we first consider the situation when $l_0=1$, in which case it follows from (\ref{eq:rel1}) that $(\mp z_2)^{k_0l_n}\pm z_2^{k_n}=0$.
\begin{itemize}
\item If $l_n=1$, then $(\mp z_2)^{k_0}\pm z_2^{k_n}=0$, where $k_0\neq k_n$. Without loss of generality, we may assume that $k_0<k_n$, then $z_2^{k_n-k_0}=\pm1$, which implies that $z_2^{-1}=\pm z_2^{k_n-k_0-1}$. Since $z_1=\mp z_2^{k_0}$, we also have $z_1^{-1}=\pm z_2^{k_0(k_n-k_0-1)}$. Thus both of $z_1^{-1}$ and $z_2^{-1}$ are either the identity or generated by the positive powers of $z_2$.
\item If $k_n=1$, then $(\mp z_2)^{k_0l_n}\pm z_2=0$. If $k_0$ or $l_n=0$, then $z_2^{-1}=z_2=\pm1$ and also $z_1=\pm1$ by (\ref{eq:rel1}), so we are done. Otherwise, neither $k_0$ nor $l_n$ is zero, so $\pm z_2^{k_0l_n-1}=1$. It follows that $z_2^{-1}=\pm z_2^{k_0l_n-2}$. Note that $k_0l_n\geq2$, since if $k_0=l_n=1$, then $(k_0,\pm l_0)=(k_n,\pm l_n)=(1,\pm1)$, contradicting our assumption that $(k_0,\pm l_0)\neq(k_n,\pm l_n)$. Since $z_1=\mp z_2^{k_0}$, we obtain $z_1^{-1}=\pm z_2^{k_0(k_0l_n-2)}$, and we again conclude that $z_1^{-1}$ and $z_2^{-1}$ are either the identity or generated by the positive powers of $z_2$.
\end{itemize}
The case when $k_0=1$ is parallel, and we omit the details.
\end{proof}

This completes the proof of Theorem \ref{theorem:WAn}. Proposition \ref{proposition:eq2}, which concerns with the equivalence
\begin{equation}
D^\mathit{perf}\mathcal{W}(\mathbb{C},D_n)\otimes_R\mathbb{K}[x,y]\simeq D^\mathit{perf}\mathcal{W}(W_f;\mathbb{K})
\end{equation}
under the same assumption that $(k_0,\pm l_0)\neq(k_n,\pm l_n)$, follows from Proposition \ref{proposition:loc} and the quasi-equivalence $\mathcal{W}(\mathbb{C},D_n)\simeq\mathcal{W}(T^\ast S^1,D_n)/\langle\gamma_0\rangle$ mentioned in Section \ref{section:EL}, which is again a consequence of \cite{gps}, Proposition 11.2.

\begin{remark}\label{remark:Seidel}
Consider the completion $\mathcal{W}(W_f;\mathbb{K})[\![x,y]\!]$ of the wrapped Fukaya category of the multi-bubble plumbing at $-1$ of the Seidel elements $\mathfrak{s}_1,\mathfrak{s}_2\in\mathit{SH}^0(W_f;\mathbb{K})^\times$, now associated to the fiberwise $T^2$-rotation of the Morse-Bott fibration $\pi:W_f\rightarrow\mathbb{C}$. Applying the same argument as in Lemma \ref{lemma:Seidel}, one finds that $\mathfrak{s}_1$ and $\mathfrak{s}_2$ correspond to $z_1$ and $z_2$ under the maps\footnote{The second map is well-defined since the cohomology $H^\ast(\mathcal{W}_f)$ is concentrated in (non-positive) even degrees.}
\begin{equation}
\mathit{SH}^0(W_f;\mathbb{K})^\times\xrightarrow{\mathit{CO}}\mathit{HH}^0\left(\mathcal{W}(W_f;\mathbb{K})\right)^\times\rightarrow Z(\mathcal{W}_f)^\times,
\end{equation}
provided that we choose an appropriate basis of $H_1(T^2;\mathbb{Z})$. Note that here $z_1$ and $z_2$ are (by abuse of notation) images of the same generators in the wrapped Fukaya $A_\infty$-algebra $\ring{\mathcal{W}}_f$ of $\ring{W}_f$ under the Drinfeld localization which removes $L_0$. As we have seen in the proof of Proposition \ref{proposition:loc}, under the assumption that $(k_0,\pm l_0)\neq(k_n,\pm l_n)$, $z_1$ and $z_2$ become torsion units in $H^0(\mathcal{W}_f)$, which implies that $\mathfrak{s}_1$ and $\mathfrak{s}_2$ are also torsion units in $\mathit{SH}^0(W_f;\mathbb{K})$. See also \cite{alp}, Corollary 2.24 for the explicit calculation of $H^0(\mathcal{W}_f)$ in the case when $W_f=W_k$ is a double bubble plumbing with $k\neq0$. This gives an alternative, yet more conceptual explanation of the fact that the completion of the wrapped Fukaya category $\mathcal{W}(W_f;\mathbb{K})$ at $\mathfrak{s}_1=\mathfrak{s}_2=-1$ is in fact redundant when the Morse-Bott surgery $Q_0$ of $Q_1,\cdots,Q_n$ is not diffeomorphic to $S^1\times S^2$.
\end{remark}

The observations made in the remark above leads to the following obstruction to the exact Lagrangian embedding in multi-bubble plumbings. Recall that a \textit{cyclic quasi-dilation} is a pair $(\tilde{b},h)\in\mathit{SH}_{S^1}^1(W_f;\mathbb{K})\times\mathit{SH}^0(W_f;\mathbb{K})^\times$, such that $B(\tilde{b})=h$. Here, $\mathit{SH}_{S^1}^\ast(W_f;\mathbb{K})$ denotes the $S^1$-equivariant symplectic cohomology, and $B:\mathit{SH}_{S^1}^\ast(W_f;\mathbb{K})\rightarrow\mathit{SH}^{\ast-1}(W_f;\mathbb{K})$ is the marking map. See \cite{yle,yla} for details.

\begin{proposition}
Suppose $(k_0,\pm l_0)\neq(k_n,\pm l_n)$. Let $L\subset W_f$ be a closed oriented exact Lagrangian submanifold, then $L$ cannot be a $K(\pi,1)$ space.
\end{proposition}
\begin{proof}
Let $\mathbb{K}$ be any field of characteristic $0$. Our Theorem \ref{theorem:WAn} combined with \cite{hkc}, Theorem A implies the existence of a cyclic quasi-dilation over $\mathbb{K}$ for the multi-bubble plumbing $W_f$ as long as $(k_0,\pm l_0)\neq(k_n,\pm l_n)$. Suppose $L\subset W_f$ is an exact Lagrangian $K(\pi,1)$ space. Arguing as in \cite{alp}, Lemma 4.7, we see that the existence of a cyclic quasi-dilation and the Viterbo functoriality applied to the Weinstein embedding $D^\ast L\hookrightarrow W_f$ implies that there must be some non-trivial unit in $\mathit{SH}^0(W_f;\mathbb{K})$ that is not torsion. This contradicts with what we have observed in Remark \ref{remark:Seidel}, which suggests that all units in $\mathit{SH}^0(W_f;\mathbb{K})$ are torsion.
\end{proof}

For $n=1$ the proposition above is more or less well-known, since the cotangent bundle of a lens space admits a dilation over characteristic $0$, see \cite{psd}, Example 3.14. It is unclear to us, however, whether the first Gutt-Hutchings capacity of the multi-bubble plumbing $\overline{W}_f$ is finite when $n\geq2$.

\subsubsection{Corollary \ref{corollary:core}}\label{section:9}

Finally, we prove Corollary \ref{corollary:core}, which says that there is a fully faithful embedding
\begin{equation}
D^b\mathit{Coh}_C(Y_f)\hookrightarrow D^\mathit{perf}\mathcal{F}(W_f;\mathbb{K}),
\end{equation} 
when $Q_0\subset\ring{W}_f$ is not diffeomorphic to $S^1\times S^2$.

Regarding $L_1,\cdots,L_n$ as Lagrangian submanifolds in the partial compactification $W_f$ of $\ring{W}_f$, consider the wrapped Fukaya $A_\infty$-algebra
\begin{equation}
\mathcal{W}_f:=\bigoplus_{1\leq i,j\leq n}\mathit{CW}^\ast(L_i,L_j).
\end{equation}
It follows from the surgery quasi-isomorphism (\ref{eq:surgery}) that there is an equivalence $D^\mathit{perf}\mathcal{W}(W_f;\mathbb{K})\simeq D^\mathit{perf}(\mathcal{W}_f)$ between the derived wrapped Fukaya category of $W_f$ and the derived category of perfect modules over $\mathcal{W}_f$. On the other hand, one can use the compact cores $Q_1,\cdots,Q_n$ in the plumbing $W_f$ to form another Fukaya $A_\infty$-algebra
\begin{equation}
\mathcal{Q}_f:=\bigoplus_{1\leq i,j\leq n}\mathit{CF}^\ast(Q_i,Q_j),
\end{equation}
which is the endomorphism $A_\infty$-algebra of the corresponding objects in the compact Fukaya category $\mathcal{F}(W_f;\mathbb{K})$. Note that both of $\mathcal{W}_f$ and $\mathcal{Q}_f$ can be regarded as $A_\infty$-algebras over the semisimple ring $\Bbbk:=\bigoplus_{1\leq i\leq n}\mathbb{K}e_i$. The following is a consequence of (\ref{eq:GEM}).

\begin{proposition}\label{proposition:EMA}
There is a quasi-isomorphism
\begin{equation}\label{eq:EM}
\mathcal{Q}_f\simeq R\hom_{\mathcal{W}_f}(\Bbbk,\Bbbk)
\end{equation}
between $A_\infty$-algebras over $\Bbbk$.
\end{proposition}

Denote by $\mathcal{B}_f$ the derived contraction algebra $\mathcal{A}_n\otimes_R\mathbb{K}[x,y]$, then we have
\begin{equation}
D^b\mathit{Coh}(Y_f)/\langle\mathcal{O}_{Y_f}\rangle\simeq D^\mathit{perf}(\mathcal{B}_f).
\end{equation}
by (\ref{eq:HMS1d}). On the other hand, let $\mathcal{C}_f$ be the endomorphism dg algebra of the object $\bigoplus_{i=1}^n\mathcal{O}_{C_i}(-1)$ in the dg enhancement of $D^b\mathit{Coh}(Y_f)$. Since they generate the subcategory $D^b\mathit{Coh}_C(Y_f)$ supported on the exceptional loci $C=\bigcup_{i=1}^nC_i$, we have an equivalence
\begin{equation}
D^b\mathit{Coh}_C(Y_f)\simeq D^\mathit{perf}(\mathcal{C}_f).
\end{equation}
The following proposition is the B-side analogue of Proposition \ref{proposition:EMA}.

\begin{proposition}\label{proposition:EMB}
There is a quasi-isomorphism
\begin{equation}
\mathcal{C}_f\simeq R\hom_{\mathcal{B}_f}(\Bbbk,\Bbbk)
\end{equation}
between dg algebras over $\Bbbk$.
\end{proposition}
\begin{proof}
This follows directly from the definition, since $\mathcal{O}_{C_i}(-1)[1]$ for $i=1,\cdots,n$ give simple modules over the noncommutative crepant resolution of $R_f$ (cf. \cite{VdB}, Proposition 3.5.7).
\end{proof}

By our assumption that $Q_0\cong S^1\times S^2$, Theorem \ref{theorem:WAn} holds. Combining Propositions \ref{proposition:EMA} and \ref{proposition:EMB}, we get from Theorem \ref{theorem:WAn} the equivalence
\begin{equation}\label{eq:HMS-core}
D^b\mathit{Coh}_C(Y_f)\simeq D^\mathit{perf}(\mathcal{Q}_f),
\end{equation}
which completes the proof of Corollary \ref{corollary:core}.

\section{The symplectic mapping class groups of $\ring{W}_f$}\label{section:Symp}

In this section we prove Theorems \ref{theorem:MCG} and \ref{theorem:MCG1}. We shall follow the general strategy of \cite{kss}.

\subsection{The braid group action}\label{section:braid}

We start by recalling a geometric presentation of the pure braid group $\mathit{PBr}_{n+2}$ due to Margalit-McCammond \cite{mmg}. Consider the base of the Morse-Bott fibration $\ring{\pi}:\overline{\ring{W}}_f\rightarrow D\setminus\{0\}$ on the Liouville domain $\overline{\ring{W}}_f$ depicted in Figure \ref{fig:base}, which is a punctured closed disc (of radius $>1$) with $n+1$ additional marked points $c_0,\cdots,c_n$. For convenience, we label the puncture at the origin by $c_{n+1}$. The punctured disc $(D,c_0,\cdots,c_{n+1})$ can be deformed to a \textit{convexly punctured disc} in the sense of \cite{mmg}, Definition 1.2, which means $c_0,\cdots,c_{n+1}$ form the vertices of an $(n+2)$-gon. This can be achieved, for example, by putting all the critical values $c_0,\cdots,c_n$ of $\ring{\pi}$ on the upper half unit circle. By \cite{mmg}, Theorem 2.3, $\mathit{PBr}_{n+2}$ has a presentation with generators 
\begin{equation}
R_{ij},\textrm{ }0\leq i<j\leq n+1
\end{equation}
given by the full twists described in Figure \ref{fig:PB}, which satisfy the following three types of relations (cf. Figure \ref{fig:relation}):
\begin{itemize}
	\item[(i)]$[R_{ij},R_{rs}]=1$ when $\{c_i,c_j\}$ and $\{c_r,c_s\}$ are non-crossing,
	\item[(ii)]$[R_{ij},R_{js}R_{rs}R_{js}^{-1}]=1$ when $\{c_r,c_s\}$ and $\{c_i,c_j\}$ cross in cyclic order,
	\item[(iii)]$R_{sj}R_{rs}R_{rj}=R_{rs}R_{rj}R_{sj}=R_{rj}R_{sj}R_{rs}$ when $(c_r,c_s,c_j)$ are vertices of a triangle in clockwise order.
\end{itemize}

\begin{figure}
	\centering
	\begin{tikzpicture}
	\draw (0,0) circle [radius=2];
	\draw (0,0) node {$\times$};
	\draw (4.5,0) circle [radius=2];
	\draw (4.5,0) node {$\times$};
	\draw (9,0) circle [radius=2];
	\draw (9,0) node {$\times$};
	\draw (-1,0) node[circle,fill,inner sep=1pt] {};
	\draw (1,0) node[circle,fill,inner sep=1pt] {};
	\node at (0,-0.3) {$c_{n+1}$};
	\node at (-1.25,0) {$c_j$};
	\node at (1.25,0) {$c_i$};
	\draw (3.5,0) node[circle,fill,inner sep=1pt] {};
	\draw (5.5,0) node[circle,fill,inner sep=1pt] {};
	\draw (8,0) node[circle,fill,inner sep=1pt] {};
	\draw (10,0) node[circle,fill,inner sep=1pt] {};
	\node at (4.5,-0.3) {$c_{n+1}$};
	\node at (3.25,0) {$c_j$};
	\node at (5.75,0 ) {$c_i$};
	\node at (9,-0.3) {$c_{n+1}$};
	\node at (7.75,0) {$c_j$};
	\node at (10.25,0) {$c_i$};
	\node at (0,-2.4) {$R_{j(n+1)}$};
	\node at (4.5,-2.4) {$R_{i(n+1)}$};
	\node at (9,-2.4) {$R_{ij}$};
	\draw [orange,->] plot [smooth] coordinates {(-1,-0.2) (-0.75,-0.5) (-0.25,-0.7) (0.25,-0.5) (0.5,0) (0.25,0.5) (-0.25,0.7) (-0.75,0.5) (-1,0.2)};
	\draw [orange,->] plot [smooth] coordinates {(5.5,0.2) (5.25,0.5) (4.75,0.7) (4.25,0.5) (4,0) (4.25,-0.5) (4.75,-0.7) (5.25,-0.5) (5.5,-0.2)};
	\draw [orange,->] plot [smooth] coordinates {(9.8,0) (9,-0.5) (8.3,0) (8,0.4) (7.5,0) (8,-0.5) (9,-1) (10,-0.2)};
	\end{tikzpicture}
	\caption{The generators of $\mathit{PBr}_{n+2}$, where $0\leq i<j<n+1$}
	\label{fig:PB}
\end{figure}

\begin{figure}
	\centering
	\begin{tikzpicture}
	\draw (0,0) circle [radius=2];
	\draw (4.5,0) circle [radius=2];
	\draw (9,0) circle [radius=2];
	\node at (0,-2.4) {(i)};
	\node at (4.5,-2.4) {(ii)};
	\node at (9,-2.4) {(iii)};
	\draw (0,1) node[circle,fill,inner sep=1pt] {};
	\draw (1,0) node[circle,fill,inner sep=1pt] {};
	\draw (-1,0) node[circle,fill,inner sep=1pt] {};
	\draw (0,-1) node[circle,fill,inner sep=1pt] {};
	\node at (0,1.2) {$c_i$};
	\node at (1.3,0) {$c_j$};
	\node at (-1.3,0) {$c_s$};
	\node at (0,-1.2) {$c_r$};
	\draw (0,1)--(1,0);
	\draw (-1,0)--(0,-1);
	\draw (4.5,1) node[circle,fill,inner sep=1pt] {};
	\draw (5.5,0) node[circle,fill,inner sep=1pt] {};
	\draw (3.5,0) node[circle,fill,inner sep=1pt] {};
	\draw (4.5,-1) node[circle,fill,inner sep=1pt] {};
	\node at (4.5,1.2) {$c_r$};
	\node at (5.8,0) {$c_i$};
	\node at (3.2,0) {$c_j$};
	\node at (4.5,-1.2) {$c_s$};
	\draw (5.5,0)--(3.5,0);
	\draw (4.5,1)--(4.5,-1);
	\draw (9,1) node[circle,fill,inner sep=1pt] {};
	\draw (8,0) node[circle,fill,inner sep=1pt] {};
	\draw (9,-1) node[circle,fill,inner sep=1pt] {};
	\draw (9,1)--(9,-1);
	\draw (9,1)--(8,0);
	\draw (9,-1)--(8,0);
	\node at (9,1.2) {$c_r$};
	\node at (7.7,0) {$c_j$};
	\node at (9,-1.2) {$c_s$};
	\end{tikzpicture}
	\caption{The configuration of punctures for the relations (i)---(iii) in the braid group presentation}
	\label{fig:relation}
\end{figure}

The center $Z(\mathit{PBr}_{n+2})\cong\mathbb{Z}$ is generated by $R_{0(n+1)}\cdot\prod_{i=n}^{0}R_{i(i+1)}$, which corresponds to an inverse Dehn twist along $\partial D\cong S^1$.

Recall that for any fixed pair of critical values $c_i$ and $c_j$ of $\ring{\pi}$, where $0\leq i<j\leq n$, there are a sequence of mutually non-isotopic Lagrangian matching lens spaces $Q_{ij}^{(k)}\subset\ring{W}_f$, $k\in\mathbb{Z}$, whose matching paths $\gamma_{ij}^{(k)}\subset\mathbb{C}^\ast$ have $c_i$ and $c_j$ as their common endpoints, where $k\in\mathbb{Z}$ corresponds to the winding number around the origin, see Figure \ref{fig:matching}. In particular, $Q_{ij}^{(0)}$ is the iterated Morse-Bott surgery of the Lagrangian lens spaces $Q_i,Q_{i+1},\cdots,Q_{j-1}$. We need to fix the gradings for the Lagrangian lens spaces $Q_{ij}^{(k)}\subset\ring{W}_f$. Choose an arbitrary grading for $Q_{ij}^{(0)}$. For $k\neq0$, we grade $Q_{ij}^{(k)}$ as follows. Make fiberwise Morse perturbations such that $Q_{ij}^{(0)}$ intersects $Q_{ij}^{(k)}$ transversely in $4|k|$ points. We choose a grading for $Q_{ij}^{(k)}$ so that for $k>0$, $\mathit{CF}^0\left(Q_{ij}^{(0)},Q_{ij}^{(-k)}\right)$ has rank $k+1$, $\mathit{CF}^1\left(Q_{ij}^{(0)},Q_{ij}^{(-k)}\right)$ has rank $2k$, and $\mathit{CF}^2\left(Q_{ij}^{(0)},Q_{ij}^{(-k)}\right)$ has rank $k-1$; the same holds if $k<0$.

On the other hand, given an embedded circle $\gamma\subset\mathbb{C}^\ast\setminus\{c_0,\cdots,c_n\}$ in the base of the Morse-Bott fibration $\ring{\pi}$, parallel transporting the $T^2$ vanishing cycle gives a Lagrangian $3$-torus. The Lagrangian torus is exact if and only if $\gamma$ is Hamiltonian isotopic to the unit circle in $\mathbb{C}^\ast$. Fix any choice of such $\gamma$ and denote the resulting exact Lagrangian torus by $T\subset\ring{W}_f$. Since $H_1(\ring{W}_f;\mathbb{Z})\cong\mathbb{Z}$, we need to fix the choice of a grading, and the convention is that the exact Lagrangian torus $T$ should have Maslov class $0$.

Given any matching path $\gamma_{ij}\subset\mathbb{C}^\ast$ between two fixed critical values $c_i$ and $c_j$ for any $0\leq i<j\leq n$, one can consider the full twist $t_{\gamma_{ij}}$, which defines an element in $\mathit{PBr}_{n+2}$. The following proposition generalizes \cite{kss}, Proposition 3.9.

\begin{proposition}
Let $b_{c_i}$ be the strand based at $c_i$ for $0\leq i\leq n+1$. The subgroup of $\mathit{PBr}_{n+2}$ generated by the twists $t_{\gamma_{ij}}$, where $0\leq i<j\leq n$ is isomorphic to
\begin{equation}\label{eq:PBc}
\mathit{PBr}_{n+2}^c=\ker\left(\mathit{PBr}_{n+2}\xrightarrow{\textrm{forget }b_{c_0},\cdots,\hat{b}_{c_i},\cdots, b_{c_n}}\mathit{PBr}_2\times\cdots\times\mathit{PBr}_2\cong\mathbb{Z}^{n+1}\right),
\end{equation}
where by forgetting $b_{c_0},\cdots,\hat{b}_{c_i}\cdots,b_{c_n}$ we mean forgetting all the strands $b_{c_0},\cdots,b_{c_n}$ except for $b_{c_i}$, where $0\leq i\leq n$.
\end{proposition}
\begin{proof}
Given any matching path $\gamma_{ij}$, it is clear that $t_{\gamma_{ij}}\in\mathit{PBr}_{n+2}^c$. Conversely, let $t\in\mathit{PBr}_{n+2}^c$. Forgetting the strand $b_{c_{n+1}}$ gives rise to the standard Faddeev-Neuwirth short exact sequence
\begin{equation}
0\rightarrow F_{n+1}\rightarrow\mathit{PBr}_{n+2}\xrightarrow{\textrm{forget }b_{c_{n+1}}}\mathit{PBr}_{n+1}\rightarrow0,
\end{equation}
where $F_{n+1}$ is the free group generated by $R_{0(n+1)},\cdots,R_{n(n+1)}$. Denote by $\varphi$ the forgetful map $\mathit{PBr}_{n+2}\rightarrow\mathbb{Z}^{n+1}$ in (\ref{eq:PBc}). Since $F_{n+1}$ is free, the restriction $\varphi|_{F_{n+1}}:F_{n+1}\rightarrow\mathbb{Z}^{n+1}$ to the subgroup $F_{n+1}$ is the abelianization map, and the kernel is given by the commutator $[F_{n+1},F_{n+1}]$, which is in turn generated by all conjugates of the basic commutators, i.e. the full twists
\begin{equation}
t_{\gamma_{ij}^{(k)}}=R_{j(n+1)}^{-k}R_{ij}R_{j(n+1)}^k,\textrm{ }0\leq i<j\leq n,k\in\mathbb{Z}
\end{equation}
along the paths $\gamma_{ij}^{(k)}$.
\end{proof}

\begin{figure}
	\centering
	\begin{tikzpicture}
	\draw (-1,0) node[circle,fill,inner sep=1pt] {};
	\draw (1,0) node[circle,fill,inner sep=1pt] {};
	\draw (0,0) node {$\times$};
	\node at (-1,0.2) {$c_i$};
	\node at (1,0.2) {$c_j$};
	\draw [orange] plot [smooth] coordinates {(1,0) (0.5,0.4) (0,0.3) (-0.3,0) (0,-0.3) (0.5,-0.5) (1.3,0) (0.9,0.5) (0.5,0.7) (0,0.8) (-0.5,0.5) (-1,0)};
	\draw (2.5,0) node[circle,fill,inner sep=1pt] {};
	\draw (4.5,0) node[circle,fill,inner sep=1pt] {};
	\draw (3.5,0) node {$\times$};
	\node at (2.5,0.2) {$c_i$};
	\node at (4.5,0.2) {$c_j$};
	\draw [orange] plot [smooth] coordinates {(2.5,0) (3,0.5) (3.5,0.7) (4,0.5) (4.5,0)};
	\draw [orange] plot [smooth] coordinates {(2.5,0) (3,-0.5) (3.5,-0.7) (4,-0.5) (4.5,0)};
	\draw [orange] plot [smooth] coordinates {(6,0) (6.5,-0.5) (7,-0.6) (7.6,-0.5) (8.3,0) (7.9,0.5) (7.5,0.6) (7,0.5) (6.5,0) (7,-0.3) (8,0)};
	\draw [orange] plot [smooth] coordinates {(11.5,0) (11,-0.2) (10.5,-0.3) (10.2,0) (10.5,0.3) (11,0.6) (11.4,0.5) (11.8,0) (11,-0.5) (10.5,-0.6) (9.9,0) (10.5,0.7) (11,0.9) (11.5,0.7) (12.1,0) (11.3,-0.7) (10.5,-0.9) (10,-0.7) (9.5,0)};
	\draw (6,0) node[circle,fill,inner sep=1pt] {};
	\draw (8,0) node[circle,fill,inner sep=1pt] {};
	\draw (7,0) node {$\times$};
	\node at (6,0.2) {$c_i$};
	\node at (8,0.2) {$c_j$};
	\draw (9.5,0) node[circle,fill,inner sep=1pt] {};
	\draw (11.5,0) node[circle,fill,inner sep=1pt] {};
	\draw (10.5,0) node {$\times$};
	\node at (9.5,0.2) {$c_i$};
	\node at (11.5,0.2) {$c_j$};
	\draw [orange] (0,1.1) node {$Q_{ij}^{(2)}$}; 
	\draw [orange] (3.5,1) node {$Q_{ij}^{(1)}$}; 
	\draw [orange] (3.5,-1) node {$Q_{ij}^{(0)}$}; 
	\draw [orange] (7,0.8) node {$Q_{ij}^{(-1)}$}; 
	\draw [orange] (10.5,1.2) node {$Q_{ij}^{(-2)}$}; 
	\end{tikzpicture}
	\caption{The Lagrangian matching cycles $Q_{ij}^{(k)}$, $0\leq i<j\leq n$, $k\in\mathbb{Z}$}
	\label{fig:matching}
\end{figure}

Define
\begin{equation}
\mathit{PBr}_{n+2}^\circ:=\left\langle R_{ij},R_{k(n+1)}|0\leq i<j\leq n,1\leq k\leq n\right\rangle
\end{equation}
to be the subgroup of $\mathit{PBr}_{n+2}$ generated by $R_{ij}$ and $R_{k(n+1)}$. It is clear that there is a short exact sequence
\begin{equation}
0\rightarrow\mathit{PBr}_{n+2}^\circ\rightarrow\mathit{PBr}_{n+2}\xrightarrow{\textrm{forget }b_{c_1},\cdots,b_{c_n}}\mathbb{Z}\rightarrow0.
\end{equation}
Moreover, there is an obvious identification
\begin{equation}
\mathit{PBr}_{n+2}^c=\ker\left(\mathit{PBr}_{n+2}^\circ\xrightarrow{\textrm{forget }n\textrm{ of the first }n+1\textrm{ strands except for }b_{c_1},\cdots,b_{c_n}}\mathbb{Z}^n\right),
\end{equation}
so we have the short exact sequence
\begin{equation}
0\rightarrow\mathit{PBr}_{n+2}^c\rightarrow\mathit{PBr}_{n+2}^\circ\rightarrow\mathbb{Z}^n\rightarrow0.
\end{equation}

Sending $R_{i(n+1)}^{-k}R_{ij}R_{i(n+1)}^k$ to the (generalized) Dehn twist \cite{mwd} along the Lagrangian lens space $Q_{ij}^{(k)}\subset\ring{W}_f$, which we denote by $\tau_{Q_{ij}^{(k)}}$, defines a map $\mathit{PBr}_{n+2}^c\rightarrow\pi_0\mathit{Symp}_c(\ring{W}_f)$. One can extend this map to the whole pure braid group $\mathit{PBr}_{n+2}$, generalizing \cite{kss}, Lemma 3.10.

\begin{lemma}
Let $\mathit{Symp}_\mathit{ex}(\ring{W}_f)$ be the group of exact symplectomorphisms of $\ring{W}_f$. There is a map
\begin{equation}\label{eq:PB}
\mathit{PBr}_{n+2}\rightarrow\pi_0\mathit{Symp}_\mathit{ex}(\ring{W}_f),
\end{equation}
extending the composition $\mathit{PBr}_{n+2}^c\rightarrow\pi_0\mathit{Symp}_c(\ring{W}_f)\rightarrow\pi_0\mathit{Symp}_\mathit{ex}(\ring{W}_f)$. This is compatible with the action of $\mathit{PBr}_{n+2}$ on the matching paths of $\ring{\pi}:\ring{W}_f\rightarrow\mathbb{C}^\ast$ in the following sense: if $\gamma$ is a matching path and $R\in\mathit{PBr}_{n+2}$ and $\gamma'=R\cdot\gamma$, then the image of $R$ in $\pi_0\mathit{Symp}_\mathit{ex}(\ring{W}_f)$ under (\ref{eq:PB}) takes the Lagrangian matching cycle $Q_\gamma$ to $Q_{\gamma'}$.
\end{lemma}
\begin{proof}
There is a symplectic fiber bundle $\mathcal{X}\rightarrow(\mathbb{C}^\ast)^{n+1}$ parametrized by $(c_0,\cdots,c_n)\in(\mathbb{C}^\ast)^{n+1}$, whose generic fiber is symplectomorphic to $\ring{W}_f$, and there are singular fibers over the union of complex hyperplanes $H_{ij}=\{z_i=z_j\}$, where $1\leq i<j\leq n+1$. The fundamental group $\pi_1\left((\mathbb{C}^\ast)^{n+1}\setminus\bigcup_{1\leq i<j\leq n+1}H_{ij}\right)$ coincides with the pure braid group $\mathit{PBr}_{n+2}$. Given any path $\gamma:[0,1]\rightarrow(\mathbb{C}^\ast)^{n+1}\setminus\bigcup_{1\leq i<j\leq n+1}H_{ij}$, the symplectic parallel transport along $\gamma$ is well defined, therefore the symplectic monodromy defines the representation (\ref{eq:PB}). Moreover, let $H_{ij}^\circ$ be the complement in $H_{ij}$ of all the other $H_{kl}$ with $(k,l)\neq(i,j)$ and $1\leq k<l\leq n+1$, the fibers over $H_{ij}^\circ\subset(\mathbb{C}^\ast)^{n+1}$ have ordinary $3$-fold double points, which implies that the monodromy around the meridian (oriented anticlockwisely) to $H_{ij}^\circ$ is a Dehn twist. This implies that (\ref{eq:PB}) extends the previously defined map $\mathit{PBr}_{n+2}^c\rightarrow\pi_0\mathit{Symp}_c(\ring{W}_f)$.
\end{proof}

The map $\mathit{PBr}_{n+2}\rightarrow\pi_0\mathit{Symp}_\mathit{ex}(\ring{W}_f)$ can be lifted to graded symplectomorphisms, giving rise to a map
\begin{equation}
\mathit{PBr}_{n+2}\rightarrow\pi_0\mathit{Symp}_\mathit{ex}^\mathit{gr}(\ring{W}_f),
\end{equation}
where $\mathit{Symp}_\mathit{ex}^\mathit{gr}(\ring{W}_f)=\mathit{Symp}^\mathit{gr}(\ring{W}_f)\cap\mathit{Symp}_\mathit{ex}(\ring{W}_f)$ is the group of graded exact symplectomorphisms. We shall fix a grading on $\mathit{Symp}_\mathit{ex}(\ring{W}_f)$ by specifying the lifts of each of the generators $R_{ij}$, $0\leq i<j\leq n+1$ as follows. First, the images of $R_{ij}$ for $0\leq i<j\leq n$ are generalized Dehn twists $\tau_{Q_{ij}}$, we give it the preferred grading for a compactly supported symplectomorphism. Second, we grade the image of $R_{j(n+1)}$, which we denote by $\lambda_j$ for $j=1,\cdots,n$, by requiring $\lambda_j^k\left(Q_{ij}^{(0)}\right)=Q_{ij}^{(-k)}$ as graded Lagrangian branes. Finally, let $\rho$ be the image of $R_{0(n+1)}$, then the product $\rho\circ\lambda_n\circ\tau_{Q_{n}}\circ\cdots\circ\tau_{Q_1}$ is the Dehn twist along $\partial D$, and in particular takes $Q_0$ to itself as an ungraded Lagrangian. We fix the grading of $\rho$ so that as a graded Lagrangian brane, the action of $\rho\circ\lambda_n\circ\tau_{Q_{n}}\circ\cdots\circ\tau_{Q_1}$ shifts the grading of $Q_0$ down by $1$. The choice of the particular shift $[1]$ is natural from the homological mirror symmetry perspective, as we shall see below.

Consider the Grothendieck group $K\left(\mathcal{W}(\ring{W}_f;\mathbb{K})\right)\cong\mathbb{Z}^{n+1}$, which is generated by the cotangent fibers $L_0,\cdots,L_n$. This can be naturally identified with the relative homology group $H_3(\ring{W}_f,\partial_\infty\ring{W}_f;\mathbb{Z})$, where $\partial_\infty\ring{W}_f\subset\ring{W}_f$ is the ideal boundary. There is a pairing
\begin{equation}\label{eq:pair}
\begin{split}
K\left(\mathcal{F}(\ring{W}_f;\mathbb{K})\right)\times K\left(\mathcal{W}(\ring{W}_f;\mathbb{K})\right)&\rightarrow\mathbb{Z}, \\
(Q,L)&\mapsto\chi\left(\mathit{HF}^\ast(L,Q)\right).
\end{split}
\end{equation}
Evaluating this on the zero sections and fibers in the plumbing $\ring{W}_f$ shows that the map $\mathbb{Z}^{n+1}=\langle Q_0,\cdots,Q_n\rangle\rightarrow K\left(\mathcal{F}(\ring{W}_f;\mathbb{K})\right)$ is injective.

The \textit{numerical Grothendieck group} of $\mathcal{F}(\ring{W}_f;\mathbb{K})$, denoted $K_\mathit{num}\left(\mathcal{F}(\ring{W}_f;\mathbb{K})\right)$, is defined to be the quotient $K\left(\mathcal{F}(\ring{W}_f;\mathbb{K})\right)/\ker$, where $\ker$ is the subgroup generated by compact exact Lagrangians $Q\subset\ring{W}_f$ such that $\chi\left(\mathit{HW}^\ast(\cdot,Q)\right)=0$ for every object in $\mathcal{W}(\ring{W}_f;\mathbb{K})$.

After passing to the quotient, the pairing (\ref{eq:pair}) becomes a non-degenerate one
\begin{equation}\label{eq:pair1}
K_\mathit{num}\left(\mathcal{F}(\ring{W}_f;\mathbb{K})\right)\times K\left(\mathcal{W}(\ring{W}_f;\mathbb{K})\right)\rightarrow\mathbb{Z},
\end{equation}
which is further identified with the intersection pairing
\begin{equation}
H_3(\ring{W}_f;\mathbb{Z})\times H_3(\ring{W}_f,\partial_\infty\ring{W}_f;\mathbb{Z})\rightarrow\mathbb{Z}.
\end{equation}

Since both of the actions of $\mathit{Symp}^\mathit{gr}(\ring{W}_f)$ and $\mathit{Symp}_c(\ring{W}_f)$ on $K\left(\mathcal{F}(\ring{W}_f;\mathbb{K})\right)$ and $K\left(\mathcal{W}(\ring{W}_f;\mathbb{K})\right)$ are compatible with the pairing (\ref{eq:pair}), there is an induced action of $\mathit{Symp}_c(\ring{W}_f)$ on $K_\mathit{num}\left(\mathcal{F}(\ring{W}_f;\mathbb{K})\right)$. By direct calculation, one can see that the intersection form on the middle homology of $\ring{W}_f$ vanishes, therefore the map $H_3(\ring{W}_f;\mathbb{Z})\rightarrow H_3(\ring{W}_f,\partial_\infty\ring{W}_f;\mathbb{Z})$ is zero, and we obtain the following.

\begin{corollary}\label{corollary:trivial}
Any $f\in\pi_0\mathit{Symp}_c(\ring{W}_f)$ acts trivially on $K_\mathit{num}\left(\mathcal{F}(\ring{W}_f;\mathbb{K})\right)$ and $K\left(\mathcal{W}(\ring{W}_f;\mathbb{K})\right)$.
\end{corollary}
\begin{proof}
We have the identification $K_\mathit{num}\left(\mathcal{F}(\ring{W}_f;\mathbb{K})\right)\cong H_3(\ring{W}_f;\mathbb{Z})$. Since compactly supported homeomorphisms act trivially on $H_3(\partial_\infty\ring{W}_f;\mathbb{Z})$, and the map $H_3(\partial_\infty\ring{W}_f;\mathbb{Z})\rightarrow H_3(\ring{W}_f;\mathbb{Z})$ is injective, we see that $f_\ast$ also acts trivially on $H_3(\ring{W}_f;\mathbb{Z})$. The rest now follows from the non-degeneracy of the pairing (\ref{eq:pair1}).
\end{proof}

\subsection{Autoequivalences of $D^b\mathit{Coh}_C(\ring{Y}_f)$}

Consider the mirror $\ring{Y}_f$ of $\ring{W}_f$, which is the crepant resolution of the $cA_n$ singularity $\{uv=f(x,y)\}$ with the divisor $\{xy=0\}$ removed, where the polynomial $f(x,y)$ has the form specified by (\ref{eq:f}) and (\ref{eq:f'}). Let $C_1,\cdots,C_n\subset\ring{Y}_f$ be the exceptional curves.

The Picard group $\mathrm{Pic}(\ring{Y}_f)\cong\mathbb{Z}^n$ is generated by the line bundles $\mathcal{L}_1,\cdots,\mathcal{L}_n$ such that
\begin{equation}
\mathcal{O}_{C_i}(k)\otimes\mathcal{L}_j\cong\left\{\begin{array}{ll}
\mathcal{O}_{C_i}(k+1) & \textrm{ if }i=j, \\
\mathcal{O}_{C_i}(k) & \textrm{ if }i\neq j,
\end{array}\right.
\end{equation}
where $1\leq i,j\leq n$ and $k\in\mathbb{Z}$. More precisely, one can find divisors $D_1,\cdots,D_n\subset\ring{Y}_f$ such that $C_i\cdot D_j=\delta_{ij}$, and $\mathcal{L}_j:=\mathcal{O}(D_j)$.

Recall that we have the full subcategory $D^b\mathit{Coh}_C(\ring{Y}_f)\subset D^b\mathit{Coh}(\ring{Y}_f)$ of complexes whose cohomology sheaves are (set-theoretically) supported on the exceptional locus $C=C_1\cup\cdots\cup C_n$. According to \cite{VdB}, Theorem B, $D^b\mathit{Coh}(\ring{Y}_f)$ is generated by $\mathcal{O}_{\ring{Y}_f},\mathcal{L}_1,\cdots,\mathcal{L}_n$, while its full subcategory $D^b\mathit{Coh}_C(\ring{Y}_f)$ is generated by $\mathcal{O}_C,\mathcal{O}_{C_1}(-1),\cdots,\mathcal{O}_{C_n}(-1)$. In particular, $K\left(D^b\mathit{Coh}_C(\ring{Y}_f)\right)\cong\mathbb{Z}^{n+1}$.

The following result follows from Lunts-Orlov \cite{lou}.

\begin{proposition}
Any autoequivalence $\Phi$ of $D^b\mathit{Coh}(\ring{Y}_f)$ is Fourier-Mukai, and the Fourier-Mukai kernel $P_\Phi\in D^b\mathit{Coh}(\ring{Y}_f\times\ring{Y}_f)$ is unique up to isomorphism.
\end{proposition}

Let $\ring{R}_f=\Gamma(\mathcal{O}_{\ring{Y}_f})$ be the ring of regular functions on $\ring{Y}_f$, then the contraction map $\ring{Y}_f\rightarrow\mathrm{Spec}(\ring{R}_f)$ equips $D^b\mathit{Coh}(\ring{Y}_f\times\ring{Y}_f)$ with the structure of an $\ring{R}_f$-linear category via pullback. The proof of the following lemma is the same as \cite{kss}, Lemma 4.5.

\begin{lemma}\label{lemma:shift}
Let $\Phi$ be an $\ring{R}_f$-linear autoequivalence of $D^b\mathit{Coh}(\ring{Y}_f)$. Then for any point $x\in\ring{Y}_f\setminus C$, $\Phi$ preserves the skyscraper sheaf $\mathcal{O}_x$ up to degree shift.
\end{lemma}

By the same argument as \cite{kss}, Lemma 4.6, we have the following.

\begin{lemma}
The $\ring{R}_f$-linear autoequivalences of $D^b\mathit{Coh}(\ring{Y}_f)$ preserve the subcategory $D^b\mathit{Coh}_C(\ring{Y}_f)$ and the corresponding Fourier-Mukai kernels have support on the fiber product $\ring{Y}_f\times_{\mathrm{Spec}(\ring{R}_f)}\ring{Y}_f$.
\end{lemma}

Let $\mathrm{Stab}D^b\mathit{Coh}_C(\ring{Y}_f)$ be the space of Bridgeland stability conditions on the triangulated category $D^b\mathit{Coh}_C(\ring{Y}_f)$. This consists of couples $(Z,\mathcal{P})$, where $Z:K\left(D^b\mathit{Coh}_C(\ring{Y}_f)\right)\rightarrow\mathbb{C}$ is a linear map called central charge and there are full additive subcategories $\mathcal{P}(\theta)\subset D^b\mathit{Coh}_C(\ring{Y}_f)$ for each $\theta\in\mathbb{R}$. Let $\mathrm{Stab}_nD^b\mathit{Coh}_C(\ring{Y}_f)$ be the space of \textit{normalized} stability conditions, i.e. stability conditions with $Z\left([\mathcal{O}_x]\right)=-1$ for any $x\in C$. $\mathrm{Stab}_nD^b\mathit{Coh}_C(\ring{Y}_f)$ has an open subset consisting of stability conditions with the standard heart $D^b\mathit{Coh}_C(\ring{Y}_f)\cap\mathit{Coh}(\ring{Y}_f)$, where $\mathit{Coh}(\ring{Y}_f)$ is the abelian category of coherent sheaves on $\ring{Y}_f$ and central charges of the form
\begin{equation}
Z(E):=-\int\exp(-\xi+i\zeta)\mathrm{ch}(E)
\end{equation}
for $\xi+i\zeta\in A\left(D^b\mathit{Coh}_C(\ring{Y}_f)\right)_\mathbb{C}$ in the complexified ample cone. Denote by
\begin{equation}
\mathrm{Stab}_n^\circ D^b\mathit{Coh}_C(\ring{Y}_f)\subset\mathrm{Stab}_nD^b\mathit{Coh}_C(\ring{Y}_f)
\end{equation}
the connected component containing this open subset.

Denote by $\mathrm{Auteq}D^b\mathit{Coh}_C(\ring{Y}_f)$ the group of $\ring{R}_f$-linear autoequivalences of $D^b\mathit{Coh}(\ring{Y}_f)$ preserving the subcategory $D^b\mathit{Coh}_C(\ring{Y}_f)$. Let $\mathrm{Auteq}^\circ D^b\mathit{Coh}_C(\ring{Y}_f)\subset\mathrm{Auteq}D^b\mathit{Coh}_C(\ring{Y}_f)$ be the subgroup of $\ring{R}_f$-linear autoequivalences preserving $\mathcal{O}_x\in D^b\mathit{Coh}_C(\ring{Y}_f)$ for some $x\in\ring{Y}_f\setminus C$.

The endomorphism $A_\infty$-algebra of an object $E$ in a $3$-Calabi-Yau category gives rise to a noncommutative deformation functor, and hence a universal object $\mathcal{E}$. We say that $E$ is \textit{fat spherical} if $\hom(\mathcal{E},E)$ is the cohomology of a sphere, see \cite{yto}. Denote by $T_\mathcal{E}$ the algebraic twist along the fat spherical object $E$. For example, both the Lagrangian lens spaces and the Lagrangian $S^1\times S^2$'s (if exist) are fat spherical objects in the Fukaya category $\mathcal{F}(\ring{W}_f;\mathbb{K})$, with the former being $3$-dimensional fat spherical objects and the later being $2$-dimensional. Under the assumption that the $cA_n$ singularity $R_f$ is isolated, $2$-dimensional fat spherical objects do not appear, and the algebraic twists along the lens spaces $Q_{ij}^{(k)}\subset\ring{W}_f$ coincide with the generalized Dehn twists. Similarly, on the B-side, we have a collection of fat spherical objects $E_{ij}^{(k)}$ of $D^b\mathit{Coh}_C(\ring{Y}_f)$, where $0\leq i< j\leq n$ and $k\in\mathbb{Z}$, defined as follows. For $1\leq i\leq n$, define $E_{i(i+1\textrm{ mod }n+1)}:=\mathcal{O}_{C_i}(-1)$. For $i<m\leq j$, define $E_{im}$ inductively by the unique non-trivial extension
\begin{equation}\label{eq:ext}
0\rightarrow E_{i(m-1)}\rightarrow E_{im}\rightarrow\mathcal{O}_{C_m}(-1)\rightarrow0.
\end{equation}
Note that as part of our convention, $E_{0j}$ is defined to be $E_{j(n+1)}$.\footnote{The indexing of the fat spherical objects $E_{ij}$ here is different from \cite{yt2}, where our $E_{i(i+1)}$ here is denoted by $E_{ii}$ there. We are choosing this convention on purpose here to match the corresponding generators in the pure braid group.} Replacing $\mathcal{O}_{C_m}(-1)$ with $\mathcal{O}_{C_m}(-k)$ in (\ref{eq:ext}) defines the fat spherical objects $E_{ij}^{(k)}$ for an arbitrary $k\in\mathbb{Z}$.

The following theorem generalizes \cite{kss}, Theorem 4.9.

\begin{theorem}\label{theorem:Toda+}
	\begin{itemize}
	\item[(i)] There is an isomorphism
	\begin{equation}
	\mathrm{Auteq}D^b\mathit{Coh}_C(\ring{Y}_f)\cong\left.\left\langle T_{\mathcal{E}_{ij}},\mathrm{Pic}(\ring{Y}_f)\right\vert0\leq i<j\leq n\right\rangle\times\mathbb{Z},
	\end{equation}
	where the first factor is a subgroup generated by the fat spherical twists in $E_{ij}$ and $\cdot\otimes\mathcal{L}$ for $\mathcal{L}\in\mathrm{Pic}(\ring{Y}_f)$. The second factor $\mathbb{Z}$ comes from the grading shift.
	\item[(ii)] We have
	\begin{equation}\label{eq:iso}
	\mathrm{Auteq}^\circ D^b\mathit{Coh}_C(\ring{Y}_f)\cong\left.\left\langle T_{\mathcal{E}_{ij}},\mathrm{Pic}(\ring{Y}_f)\right\vert0\leq i<j\leq n\right\rangle. 
	\end{equation}
	Let $\mathrm{Auteq}_\mathrm{Tor}^\circ D^b\mathit{Coh}_C(\ring{Y}_f)$ be the subgroup of $\mathrm{Auteq}^\circ D^b\mathit{Coh}_C(\ring{Y}_f)$ preserving the connected component $\mathrm{Stab}_n^\circ D^b\mathit{Coh}_C(\ring{Y}_f)$ in the space of stability conditions. Equivalently, it is also the kernel of the group homomorphism $\mathrm{Auteq}^\circ D^b\mathit{Coh}_C(\ring{Y}_f)\rightarrow\mathbb{Z}^n\rtimes\mathit{GL}_n(\mathbb{Z})$, where the right-hand side is the stabilizer of $[\mathcal{O}_x]$ in $\mathit{GL}_{n+1}(\mathbb{Z})=\mathrm{Aut}K\left(D^b\mathit{Coh}_C(\ring{Y}_f)\right)$. The image of this homomorphism is $\mathbb{Z}^n$, and we have a commutative diagram
	\begin{equation}\label{eq:dia}
	\begin{tikzcd}
	0 \arrow[r] &\mathrm{Auteq}_\mathrm{Tor}^\circ D^b\mathit{Coh}_C(\ring{Y}_f) \arrow[r,hook] &\mathrm{Auteq}^\circ D^b\mathit{Coh}_C(\ring{Y}_f) \arrow[r] &\mathbb{Z}^n \arrow[r] &0 \\
	0 \arrow[r] &\mathit{PBr}_{n+2}^c \arrow[u,"\cong"] \arrow[r] &\mathit{PBr}_{n+2}^\circ \arrow[u,"\cong"] \arrow[r] &\mathbb{Z}^n \arrow[u,"="] \arrow[r] &0 \\
	0 \arrow[r] &\mathbb{Z}^{\ast\infty} \arrow[r] \arrow[u,hook] &\mathbb{Z}\ast\mathbb{Z} \arrow[r] \arrow[u,hook] &\mathbb{Z} \arrow[r] \arrow[u,hook] &0
	\end{tikzcd}
	\end{equation}
	The first vertical isomorphism maps the standard generators $R_{j(n+1)}^{-k}R_{ij}R_{j(n+1)}^k$ of $\mathit{PBr}_{n+2}^c$ to $\otimes\mathcal{L}_{ij}^k\circ T_{\mathcal{E}_{ij}}\circ\otimes\mathcal{L}_{ij}^{-k}$ for any $k\in\mathbb{Z}$, where $\mathcal{L}_{ij}\in\mathrm{Pic}(\ring{Y}_f)$ is the class of the line bundle $\mathcal{L}_i\otimes\mathcal{L}_j$. The second vertical isomorphism maps the standard generators $R_{ij}$ and $R_{k(n+1)}$ of $\mathit{PBr}_{n+2}^\circ$, where $1\leq k\leq n$, to $T_{\mathcal{E}_{(i+1)j}}$ and $\otimes\mathcal{L}_k$, respectively.
	\end{itemize}
\end{theorem}
\begin{proof}
(i) is proved by Toda in \cite{yt2}, Section 7 (Appendix). 

For (ii), the isomorphism (\ref{eq:iso}) follows from the definition and Lemma \ref{lemma:shift}. To see the image of the homomorphism $\mathrm{Auteq}^\circ D^b\mathit{Coh}_C(\ring{Y}_f)\rightarrow\mathbb{Z}^n\rtimes\mathit{GL}_n(\mathbb{Z})$ is $\mathbb{Z}^n$, note that in the Grothendieck group $K\left(D^b\mathit{Coh}_C(\ring{Y}_f)\right)$, the class of $[\mathcal{O}_x]$ can be expressed as $[\mathcal{O}_C]-\sum_{i=1}^n\left[\mathcal{O}_{C_i}(-1)\right]$. Pick any line bundle $\mathcal{L}\in\mathrm{Pic}(\ring{Y}_f)=\langle\mathcal{L}_1,\cdots,\mathcal{L}_n\rangle\cong\mathbb{Z}^n$. Up to possibly replacing $\mathcal{L}$ with its dual, its action (by tensor) on $K\left(D^b\mathit{Coh}_C(\ring{Y}_f)\right)$ with respect to the basis $\left\{[\mathcal{O}_x],\left[\mathcal{O}_{C_1}\right],\cdots,\left[\mathcal{O}_{C_n}\right]\right\}$ is given by the block matrix
\begin{equation}
\left(\begin{array}{ll}
1 & \vec{v} \\
\vec{0} & E
\end{array}\right),
\end{equation}
where $\vec{v}\in\mathbb{Z}^n$ is the vector of $\mathcal{L}$ in the Picard group and $E$ is the $n\times n$ identity matrix. The two vertical isomorphisms in (\ref{eq:dia}) between subgroups of $\mathit{PBr}_{n+2}$ and $\mathrm{Auteq}D^b\mathit{Coh}_C(\ring{Y}_f)$ follow from the definitions of these groups, note in particular that $\otimes\mathcal{L}_{ij}^k\circ T_{\mathcal{E}_{ij}}\circ\otimes\mathcal{L}_{ij}^{-k}=T_{\mathcal{E}_{ij}^{(k)}}$. Fixing any pair of $(i,j)$ with $0\leq i<j\leq n$, we get embeddings from the third row of the diagram to the second row by sending the generators of $\mathbb{Z}^{\ast\infty}$ to $R_{j(n+1)}^{-k}R_{ij}R_{j(n+1)}^k$, $k\in\mathbb{Z}$; the two generators of $\mathbb{Z}\ast\mathbb{Z}$ to $R_{ij}$ and $R_{j(n+1)}$, respectively; and the unique generator of $\mathbb{Z}$ to the image of $R_{j(n+1)}$ under the quotient map $\mathit{PBr}_{n+2}^\circ\rightarrow\mathbb{Z}^n$. The commutativity of the two bottom squares in (\ref{eq:dia}) are then clear from the definitions.
\end{proof}

\subsection{Proof of Theorem \ref{theorem:MCG}}

In order to make use of the homological mirror symmetry between $\ring{W}_f$ and $\ring{Y}_f$, we need a mirror functor realizing the equivalence of Theorem \ref{theorem:main}, whose existence is established in the following lemma.

\begin{lemma}\label{lemma:functor}
There is a functor
\begin{equation}\label{eq:fun}
\Upsilon:D^\mathit{perf}\mathcal{W}(\ring{W}_f;\mathbb{K})\xrightarrow{\simeq}D^b\mathit{Coh}(\ring{Y}_f)
\end{equation}
realizing the homological mirror symmetry equivalence in Theorem \ref{theorem:main}, such that
\begin{equation}\label{eq:fun-ob}
\Upsilon(L_0)=\mathcal{O}_{\ring{Y}_f},\Upsilon(L_i)=\mathcal{L}_i\textrm{ for }i=1,\cdots,n
\end{equation}
and
\begin{equation}
\Upsilon\left(Q_{ij}^{(k)}\right)=E_{ij}^{(k)}\textrm{ for }0\leq i<j\leq n\textrm{ and }k\in\mathbb{Z}.
\end{equation}
\end{lemma}
\begin{proof}
In the proof of the equivalence (\ref{eq:HMS-un}), we have determined the generators $L_0,\cdots,L_n$ of the wrapped Fukaya category $\mathcal{W}(\ring{W}_f;\mathbb{K})$ and the generators $\mathcal{O}_{\ring{Y}_f},\mathcal{L}_1,\cdots,\mathcal{L}_n$ of the derived category $D^b\mathit{Coh}(\ring{Y}_f)$. We have then identified their endomorphism algebras with $\ring{\mathcal{A}}_n\otimes\mathbb{K}[x^{\pm1},y^{\pm1}]$. Recall that the endomorphism algebra of the tilting bundle $\ring{\mathcal{T}}=\mathcal{O}_{\ring{Y}_f}\oplus\bigoplus_{i=1}^n\mathcal{L}_i$ is the localization of the corresponding tilting bundle $\mathcal{T}\rightarrow Y_f$, whose endomorphism algebra we have identified to be $\ring{\mathcal{A}}_n\otimes\mathbb{K}[x,y]$, see (\ref{eq:uncom}). The identifications can be done over the semisimple ring $\Bbbk=\bigoplus_{i=0}^n\mathbb{K}e_i$, therefore defining the functor as in (\ref{eq:fun-ob}) (since the underlying $A_\infty$ and dg categories are formal, there is no obstruction of defining $\Upsilon$ on the morphism level) and extending it to the whole category $D^\mathit{perf}\mathcal{W}(\ring{W}_f;\mathbb{K})$ we get an equivalence (\ref{eq:fun}). It follows from the generalized Eilenberg-Moore equivalence (\ref{eq:EM}) that the behavior of the functor $\Upsilon$ on compact Lagrangians is determined by its behavior on the generators $L_0,\cdots,L_n$ of the wraped Fukaya category. See also, \cite{cpul}, Section 7, note that Conjecture 7.5 loc. cit. is true by \cite{sgc}.
\end{proof}

\begin{corollary}\label{corollary:linear}
The equivalence $\Upsilon$ defined above intertwines the $\mathit{SH}^0(\ring{W}_f;\mathbb{K})$-linear structure of $D^\mathit{perf}\mathcal{W}(\ring{W}_f;\mathbb{K})$ and the $\ring{R}_f$-linear structure of $D^b\mathit{Coh}(\ring{Y}_f)$.
\end{corollary}
\begin{proof}
This is the same as \cite{kss}, Corollary 5.2.
\end{proof}

\begin{lemma}\label{lemma:mirror-T}
Let $(T,\zeta)$ be a Lagrangian brane associated to the exact Lagrangian torus $T\subset\ring{W}_f$ defined in Section \ref{section:braid}, where $\zeta$ is a unitary rank $1$ local system. Under the mirror equivalence functor $\Upsilon$, we have
\begin{equation}
\Upsilon(T,\zeta)=\mathcal{O}_x[k]
\end{equation}
for some $x\in\ring{Y}_f\setminus C$ and $k\in\mathbb{Z}$.
\end{lemma}
\begin{proof}
Let $\mathcal{E}$ be the mirror to $(T,\zeta)$. For each $0\leq i\leq n$, one can find a Hamiltonian isotopy of $T$ which displaces it from $Q_i$, and the wrapped Floer cochain complex $\mathit{CW}^\ast(T,L_i)$ has a single generator. By Lemma \ref{lemma:functor}, $\mathcal{E}$ is orthogonal to $\mathcal{O}_C$ and $\mathcal{O}_{C_i}(-1)$ for $1\leq i\leq n$ in $D^b\mathit{Coh}(\ring{Y}_f)$. Since $D^b\mathit{Coh}_C(\ring{Y}_f)$ is generated by $\mathcal{O}_C,\mathcal{O}_{C_1}(-1),\cdots,\mathcal{O}_{C_n}(-1)$, we see that $\mathcal{E}$ is orthogonal to the whole subcategory $D^b\mathit{Coh}_C(\ring{Y}_f)$ of $D^b\mathit{Coh}(\ring{Y}_f)$. It follows that $\mathcal{E}$ is orthogonal to $\mathcal{O}_x$ for any $x\in C$, therefore $\mathrm{supp}(\mathcal{E})\subset\ring{Y}_f\setminus C$.

Since the wrapped Floer cohomologies between $(T,\zeta)$ and each of the cotangent fibers $L_0,\cdots,L_n$ are finite-dimensional, by Lemma \ref{lemma:functor}, its $\mathcal{E}$ also has finite rank total Ext groups with the tilting bundle $\ring{\mathcal{T}}=\bigoplus_{i=1}^n\mathcal{L}_i\oplus\mathcal{O}_{\ring{Y}_f}$. This implies that $\mathrm{supp}(\mathcal{E})$ is proper.

Since the only proper subvarieties of $\ring{Y}_f\setminus C$ are finite unions of points, and $(T,\zeta)$ has self-Floer cohomology the cohomology of $T^3$, its mirror $\mathcal{E}$ is simple and has no negative degree self-Ext groups, so $\mathcal{E}=\mathcal{O}_x[k]$ for some $x\in\ring{Y}_f\setminus C$ and $k\in\mathbb{Z}$.
\end{proof}

The same argument as in the proof of \cite{kss}, Lemma 5.5 yields the following.

\begin{lemma}\label{lemma:T}
For any $f\in\pi_0\mathit{Symp}_c(\ring{W}_f)$, equipped with the preferred grading, $\phi(T,\zeta)=(T,\zeta)$ in $D^\mathit{perf}\mathcal{W}(\ring{W}_f;\mathbb{K})$.
\end{lemma}

Define $\mathit{Symp}^\circ(\ring{W}_f)\subset\mathit{Symp}^\mathit{gr}(\ring{W}_f)$ to be the subgroup of graded symplectomorphisms which take the Lagrangian brane $(T,\zeta)$ to itself. Now we can generalize \cite{kss}, Lemma 5.7 as follows.

\begin{lemma}\label{lemma:circ}
Consider the map $\mathit{PBr}_{n+2}\rightarrow\pi_0\mathit{Symp}^\mathit{gr}(\ring{W}_f)$ defined in Section \ref{section:braid}. The intersection of its image with $\pi_0\mathit{Symp}^\circ(\ring{W}_f)$ is precisely the subgroup generated by $\tau_{Q_{ij}}$ for $0\leq i<j\leq n$ and $\lambda_{1},\cdots,\lambda_{n}$, i.e. the image of the subgroup $\mathit{PBr}_{n+2}^\circ\subset\mathit{PBr}_{n+2}$.
\end{lemma}
\begin{proof}
Recall that the image of $R_{ij}$, where $0\leq i<j\leq n$ is the generalized Dehn twist $\tau_{Q_{ij}}$ along the Lagrangian lens space $Q_{ij}\subset\ring{W}_f$ with its preferred grading, and the image of $R_{j(n+1)}$ for $j=1,\cdots,n$ are the non-compactly supported symplectomorphisms $\lambda_j$ which map $Q_{ij}^{(k)}$ to $Q_{ij}^{(k+1)}$ for all $k\in\mathbb{Z}$. They induce $\ring{R}_f$-linear autoequivalences of $\mathcal{W}(\ring{W}_f;\mathbb{K})$. It follows from Lemma \ref{lemma:T} that $\tau_{Q_{ij}}\in\pi_0\mathit{Symp}^\circ(\ring{W}_f)$. By Lemma \ref{lemma:functor} and Corollary \ref{corollary:linear}, the mirror autoequivalences in $\mathrm{Auteq}D^b\mathit{Coh}_C(\ring{Y}_f)$ must take $\mathcal{O}_{C_j}(-k)$ to $\mathcal{O}_{C_j}(-k-1)$. By Theorem \ref{theorem:Toda+}, the only possibility for these autoequivalences are $\otimes\mathcal{L}_j$, which in particular fix $\mathcal{O}_x$ for $x\in\ring{Y}_f\setminus C$. By Lemma \ref{lemma:mirror-T}, we have $\lambda_j\in\pi_0\mathit{Symp}^\circ(\ring{W}_f)$ for all $j=1,\cdots,n$. Finally, notice that the generator of $Z(\mathit{PBr}_{n+2})$, which is the product $R_{0(n+1)}\cdot R_{n(n+1)} R_{(n-1)n}\cdots R_{01}$, shifts the grading down by $1$.
\end{proof}

Since the group $\pi_0\mathit{Symp}^\mathit{gr}(\ring{W}_f)$ acts on the wrapped Fukaya category by $\mathit{SH}^0(\ring{W}_f;\mathbb{K})$-linear autoequivalences, we have a map
\begin{equation}
\pi_0\mathit{Symp}^\mathit{gr}(\ring{W}_f)\rightarrow\mathrm{Auteq}_{\mathit{SH}^0}D^\mathit{perf}\mathcal{W}(\ring{W}_f;\mathbb{K}),
\end{equation}
whose image preserves the subcategory $D^\mathit{perf}\mathcal{Q}(\ring{W}_f;\mathbb{K})$ of $D^\mathit{perf}\mathcal{F}(\ring{W}_f;\mathbb{K})$ split-generated by $Q_0,\cdots,Q_n$, which is equivalent to $D^b\mathit{Coh}_C(\ring{Y}_f)$ under mirror symmetry. Moreover, by Corollary \ref{corollary:trivial}, any element of $\pi_0\mathit{Symp}^\circ(\ring{W}_f)$ lands in $\mathrm{Auteq}^\circ D^b\mathit{Coh}_C(\ring{Y}_f)$. Altogether, we have a commutative diagram
\begin{equation}
	\begin{tikzcd}
	\pi_0\mathit{Symp}_c(\ring{W}_f) \arrow[d] \arrow[r] &\pi_0\mathit{Symp}^\circ(\ring{W}_f) \arrow[d] \\
	\mathrm{Auteq}_{\mathrm{Tor}}D^\mathit{perf}\mathcal{Q}(\ring{W}_f;\mathbb{K}) \arrow[r] \arrow[d,"\cong"] &\mathrm{Auteq}D^\mathit{perf}\mathcal{Q}(\ring{W}_f;\mathbb{K}) \arrow[d,"\cong"] \\
	\mathrm{Auteq}_{\mathrm{Tor}}^\circ D^b\mathit{Coh}_C(\ring{Y}_f) \arrow[r]  &\mathrm{Auteq}^\circ D^b\mathit{Coh}_C(\ring{Y}_f)
	\end{tikzcd}
\end{equation}

Recall that $\pi_0\mathit{Symp}^\mathit{gr}(\ring{W}_f)$, and in particular $\pi_0\mathit{Symp}^\circ(\ring{W}_f)$ acts on $K_\mathit{num}\left(\mathcal{F}(\ring{W}_f;\mathbb{K})\right)$. Since the image of $\mathrm{Auteq}^\circ D^b\mathit{Coh}_C(\ring{Y}_f)$ in $\mathit{GL}\left(K\left(D^b\mathit{Coh}_C(\ring{Y}_f)\right)\right)$ is isomorphic to $\mathbb{Z}^n$, the image of $\pi_0\mathit{Symp}^\circ(\ring{W}_f)$ in $\mathit{GL}\left(K_\mathit{num}\left(\mathcal{F}(\ring{W}_f;\mathbb{K})\right)\right)$ is also isomorphic to $\mathbb{Z}^n$.

Theorems \ref{theorem:MCG} and \ref{theorem:MCG1} will follow from the theorem below.

\begin{theorem}
We have the following.
\begin{itemize}
	\item[(i)] For the graded symplectic mapping class group, we have
	\begin{equation}
	\begin{split}
	\mathit{PBr}_{n+2}&\hookrightarrow\pi_0\mathit{Symp}^\mathit{gr}(\ring{W}_f)\twoheadrightarrow\mathrm{Auteq}_{\mathit{SH}^0}D^\mathit{perf}\mathcal{W}(\ring{W}_f;\mathbb{K})\cong\mathrm{Auteq}D^b\mathit{Coh}_C(\ring{Y}_f) \\
	&\cong\mathit{PBr}_{n+2},
	\end{split}
	\end{equation}
    where the generators $R_{ij}$, where $0\leq i<j\leq n$ are mapped to $\tau_{Q_{ij}}$, $R_{j(n+1)}$, where $1\leq j\leq n$ are mapped to $\lambda_j$, and the central generator $R_{0(n+1)}\cdot R_{n(n+1)} R_{(n-1)n}\cdots R_{01}$ is mapped to the degree shift $[1]$. Moreover, these group homomorphisms compose to the identity.
    \item[(ii)] Exact symplectomorphisms which preserve $(T,\zeta)\in\mathrm{Ob}\left(\mathcal{W}(\ring{W}_f;\mathbb{K})\right)$ fit into the following commutative diagram:
    \begin{equation}
	\begin{tikzcd}
	\mathbb{Z}^{\ast\infty} \arrow[d,hook] \arrow[r] &\mathbb{Z}\ast\mathbb{Z} \arrow[d,hook] \arrow[r] &\mathbb{Z} \arrow[d,hook] \\
	\mathit{PBr}_{n+2}^c \arrow[d,hook] \arrow[r] &\mathit{PBr}_{n+2}^\circ \arrow[r] \arrow[d,hook] &\mathbb{Z}^n \arrow[d,"="] \\
	\pi_0\mathit{Symp}_c(\ring{W}_f) \arrow[d,two heads] \arrow[r] &\pi_0\mathit{Symp}^\circ(\ring{W}_f) \arrow[d,two heads] \arrow[r] &\mathbb{Z}^n \arrow[d,"="] \\
	\mathrm{Auteq}^\circ_\mathrm{Tor}D^b\mathit{Coh}_C(\ring{Y}_f) \arrow[d,"\simeq"] \arrow[r] &\mathrm{Auteq}^\circ D^b\mathit{Coh}_C(\ring{Y}_f) \arrow[d,"\simeq"] \arrow[r] &\mathbb{Z}^n \arrow[d,"="] \\
	\mathit{PBr}_{n+2}^c \arrow[r] \arrow[d,two heads] &\mathit{PBr}_{n+2}^\circ \arrow[r] \arrow[d,two heads] &\mathbb{Z}^n \arrow[d,two heads] \\
	\mathbb{Z}^{\ast\infty} \arrow[r] &\mathbb{Z}\times\mathbb{Z} \arrow[r] &\mathbb{Z}
	\end{tikzcd}
	\end{equation}
    where the vertical compositions from the top to the bottom, and from the second line to the fifth line are identities. The map $\pi_0\mathit{Symp}^\circ(\ring{W}_f)\rightarrow\mathbb{Z}^n$ is given by taking the action on $K_\mathit{num}\left(\mathcal{F}(\ring{W}_f;\mathbb{K})\right)\cong K\left(D^b\mathit{Coh}_C(\ring{Y}_f)\right)$, and noting that the image lies in the $\mathbb{Z}^n$ subgroup.
\end{itemize}
\end{theorem}
\begin{proof}
(i) follows from Theorem \ref{theorem:Toda+}, (i), Toda's computation of the autoequivalence group implies that we have an isomorphism $\mathrm{Auteq}D^b\mathit{Coh}_C(\ring{Y}_f)\cong\mathit{PBr}_{n+2}$.

For (ii), the map from the first row to the second row is defined in Theorem \ref{theorem:Toda+}, (ii). The map from the second row to the third row is defined by combining the constructions in Section \ref{section:braid} with Lemma \ref{lemma:circ}. By Corollary \ref{corollary:trivial}, the image of $\pi_0\mathit{Symp}_c(\ring{W}_f)$ lies in the subgroup $\mathrm{Auteq}_\mathrm{Tor}^\circ D^b\mathit{Coh}_C(\ring{Y}_f)\subset\mathrm{Auteq}^\circ D^b\mathit{Coh}_C(\ring{Y}_f)$, which explains how the third row maps to the fourth row. The isomorphism between the fourth and fifth rows is a consequence of Theorem \ref{theorem:Toda+}. To map the fifth row to the sixth row, we project the generators $R_{j(n+1)}^{-k}R_{ij}R_{j(n+1)}^k$ of $\mathit{PBr}_{n+2}^c$, where $0\leq i<j\leq n$ and $k\in\mathbb{Z}$ to a fixed pair of $(i,j)$, which has already been chosen in defining the map $\mathbb{Z}^{\ast\infty}\hookrightarrow\mathit{PBr}_{n+2}^c$ from the first row to the second row; and the generators $R_{ij}$ and $R_{k(n+1)}$ of $\mathit{PBr}_{n+2}^\circ$, where $0\leq i<j\leq n$ and $1\leq k\leq n$ to $R_{ij}$ and $R_{j(n+1)}$ for fixed $(i,j)$.

It remains to check that the composition from the second row to the fifth row is an isomorphism. For the first column, note that Lemma \ref{lemma:functor} allows us to identify the generalized Dehn twists $\tau_{Q_{ij}^{(k)}}$ with the fat spherical twists $T_{\mathcal{E}_{ij}^{(k)}}$ for $0\leq i<j\leq n$ and $k\in\mathbb{Z}$.

For the second column, recall that we have $\lambda_1,\cdots,\lambda_n\in\pi_0\mathit{Symp}^\circ(\ring{W}_f)$. With our choices of gradings, $\lambda_j$ maps $Q_j^{(k)}$ to $Q_j^{(k+1)}$ for all $k\in\mathbb{Z}$ as graded Lagrangian branes. On the Grothendieck group, this is the same action as tensoring with the line bundle $\mathcal{L}_j$. This implies that $\pi_0\mathit{Symp}^\circ(\ring{W}_f)$ surjects onto $\mathbb{Z}^n$. Surjectivity of the second column now follows from the short exact sequence
\begin{equation}
0\rightarrow\mathrm{Auteq}_\mathrm{Tor}^\circ D^b\mathit{Coh}_C(\ring{Y}_f)\rightarrow\mathrm{Auteq}^\circ D^b\mathit{Coh}_C(\ring{Y}_f)\rightarrow\mathbb{Z}^n\rightarrow0
\end{equation}
from Theorem \ref{theorem:Toda+}, (ii). 

Since the composition from the second row to the fifth row is an isomorphism, it is clear that the maps from the second row to the third row are injective, and the maps from the third row to the fourth row are surjective.
\end{proof}

\Addresses


\begin{thebibliography}{99}
\bibitem{maa2}M. Abouzaid, \textit{A topological model for the Fukaya categories of plumbings}, J. Differ. Geom. 87(1), 1--80 (2011).
\bibitem{aah}M. Abouzaid and D. Auroux, \textit{Homological mirror symmetry for hypersurfaces in $(\mathbb{C}^\ast)^n$}, Geom. Topol. 28 (2024) 2825--2914.
\bibitem{aakl}M. Abouzaid, D. Auroux and L. Katzarkov, \textit{Lagrangian fibrations on blowups of toric varieties and mirror symmetry for hypersurfaces}, Publ. Math. Inst. Hautes \'{E}tudes Sci. 123, 199--282 (2016).
\bibitem{ase}M. Abouzaid and I. Smith, \textit{Exact Lagrangians in plumbings}, Geom. Funct. Anal. 22 (2012), no. 4, 785--831.
\bibitem{jas}J. Asplund, \textit{Simplicial descent for Chekanov-Eliashberg dg-algebras}, J. Topol. 16(2): 489--541, 2023.
\bibitem{alp}J. Asplund and Y.Li, \textit{Persistence of unknottedness of clean Lagrangian intersections}, J. Topol. 18(4): Paper No. e70053, 2025.
\bibitem{mbd}M. Ballard, \textit{Derived categories of singular schemes with an application to reconstruction}, Adv. Math. 227(2), 895--919 (2011).
\bibitem{mbn}M. Booth, \textit{Noncommutative deformation theory, the derived quotient, and DG singularity categories}, arXiv:1810.10060.
\bibitem{beee}F. Bourgeois, T. Ekholm, Y. Eliashberg, \textit{Effect of Legendrian surgery}, with an appendix by S. Ganatra and M. Maydanskiy, Geom. Topol. 16, (2012), 301--389.
\bibitem{fbw}F. Bro\'{c}i\'{c}, \textit{Wrapped Floer homology and subcritical handle attachment}, arXiv:2507.23290.
\bibitem{cdrgg}B. Chantraine, G. Dimitroglou-Rizell, P. Ghiggini and R. Golovko. \textit{Geometric generation of the wrapped Fukaya category of Weinstein manifolds and sectors}, Ann. Sci. \'{E}c. Norm. Sup\'{e}r. (4), 57(1):1--85, 2024.
\bibitem{cpul}K. Chan, D. Pomerleano and K. Ueda, \textit{Lagrangian torus fibrations and homological mirror symmetry for the conifold}, Comm. Math. Phys. 341 (2016), 135--178.
\bibitem{dw}W. Donovan and M. Wemyss, \textit{Noncommutative deformations and flops}, Duke Math. J. 165, no. 8 (2016): 1397--1474.
\bibitem{vdd}V. Drinfeld, \textit{DG quotients of DG categories}, J. Algebra 272 (2004), no.2, 643--691.
\bibitem{evle}J. Evans and Y. Lekili, \textit{Noncommutative crepant resolutions of $cA_n$ singularities via Fukaya categories}, Doc. Math. 31 (2026), 1--26.
\bibitem{tem}T. Ekholm, \textit{Morse flow trees and Legendrian contact homology in $1$-jet spaces}, Geom. Topol. 11: 1083--1224, 2007.
\bibitem{ekle}T. Ekholm and Y. Lekili, \textit{Duality between Lagrangian and Legendrian invariants}, Geom. Topol. 27 (2023), no. 6, 2049--2179.
\bibitem{etlep}T. Etg\"{u} and Y. Lekili, \textit{Fukaya categories of plumbings and multiplicative preprojective algebras}, Quantum Topol. 10 (2019), no. 4, 777--813.
\bibitem{gtwc}A. Gadbled, A.-L. Thiel, and E. Wagner, \textit{Categorical action of the extended braid group of affine type A}, Commun. Contemp. Math., 19(3):1650024, 39, 2017.
\bibitem{sgs}S. Ganatra, \textit{Symplectic cohomology and duality for the wrapped Fukaya category}, arXiv:1304.7312.
\bibitem{sgc}S. Ganatra, \textit{Cyclic homology}, $S^1$-\textit{equivariant Floer cohomology}, \textit{and Calabi-Yau structures}, Geometry \& Topology 27(9) (2023), 3461--3584.
\bibitem{gps}S. Ganatra, J. Pardon and V. Shende, \textit{Sectorial descent for wrapped Fukaya categories}, J. Amer. Math. Soc. 37 (2024), 499--635.
\bibitem{hwf}Y. Hirano and M. Wemyss, \textit{Faithful actions from hyperplane arrangements}. Geom. Topol. 22 (6), 3395--3433 (2018).
\bibitem{hkc}Z. Hua and B. Keller, \textit{Cluster categories and rational curves}, Geom. Topol. 28(6): 2569--2634, 2024.
\bibitem{kih}K. Irie, \textit{Handle attaching in wrapped Floer homology and brake orbits in classical Hamiltonian systems}, Osaka J. Math. 50(2): 363--396, 2013.
\bibitem{iua}A. Ishii, K. Ueda and H. Uehara, \textit{Stability conditions on $A_n$-singularities}, J. Differential Geom. 84 (2010), no. 1, 87--126.
\bibitem{ky}M. Kalck and D. Yang, \textit{Relative singularity categories I: Auslander resolutions}, Adv. Math. 301:973--1021, 2016.
\bibitem{klw}D. Karabas and S. Lee, \textit{The wrapped Fukaya category of plumbings}, J. Symplectic Geom. 23 (4) (2025), 853--949.
\bibitem{jkd}J. Karmazyn, \textit{Deformations of algebras defined by tilting bundles}, J. Algebra 513: 388--434, 2018.
\bibitem{sks}S. Katz, \textit{Small resolutions of Gorenstein threefold singularities}, In Algebraic geometry: Sundance 1988, volume 116 of Contemp. Math., pages 61--70. Amer. Math. Soc., Providence, RI, 1991.
\bibitem{kss}A. Keating and I. Smith, \textit{Symplectomorphisms and spherical objects in the conifold smoothing},
Compos. Math. 160 (11): 2738--2773, 2025.
\bibitem{ksws}A. Keating, I. Smith and M.Wemyss, \textit{Splitting symplectic monodromy}, arXiv:2601.20438.
\bibitem{lllm}S. Lee, Y. Li, S. Liu, and C.Y. Mak, \textit{Fukaya categories of hyperplane arrangements}, Geometry \& Topology 29(9) (2025), 4841--4909.
\bibitem{lese}Y. Lekili and E. Segal, \textit{Equivariant Fukaya categories at singular values}, arXiv:2304.10969. To appear in Glasgow Mathematical Journal.
\bibitem{yle}Y. Li, \textit{Exact Calabi-Yau categories and odd-dimensional Lagrangian spheres}, Quantum Topology 15 (2024), 123--227.
\bibitem{yla}Y. Li, \textit{Aspherical Lagrangian submanifolds, Audin's conjecture and cyclic dilations}, arXiv:2308.05086. To appear in Selecta Mathematica (N.S.).
\bibitem{dmr}D. Mamaev, \textit{Relative wrapped Fukaya categories of surfaces}, PhD Thesis, University College London, 2025.
\bibitem{lou}V.A. Lunts and D.O. Orlov, \textit{Uniqueness of enhancement for triangulated categories}, J. Amer. Math. Soc. 23 (2010), 853--908.
\bibitem{mmg}D. Margalit and J. McCammond, \textit{Geometric presentations for the pure braid group}, J. Knot Theory Ramifications 18 (2009), 1--20.
\bibitem{mw1}C.Y. Mak and W. Wu, \textit{Dehn twist exact sequences through Lagrangian cobordism}. Compos. Math. 154 (12), 2485--2533 (2018).
\bibitem{mwd}C.Y. Mak and W. Wu, \textit{Dehn twists and Lagrangian spherical manifolds}, Selecta Math. (N.S.) 25(5): 68 (2019).
\bibitem{otc}Y.-G. Oh and H.L. Tanaka, \textit{Continuous and coherent actions on wrapped Fukaya categories}, arXiv:1911.00349.
\bibitem{psc}T. Perutz and N. Sheridan, \textit{Constructing the relative Fukaya category}, J. Symplectic Geom. 21 (2023), no. 5, 997--1076.
\bibitem{arf}A.F. Ritter, \textit{Floer theory for negative line bundles via Gromov-Witten invariants}, Adv. Math. 262 (2014), 1035--1106.
\bibitem{pspi}P. Seidel, \textit{$\pi_1$ of symplectic automorphism groups and invertibles in quantum homology rings}, Geom. Funct. Anal. 7 (1997) 1046--1095.
\bibitem{psa}P. Seidel, \textit{A long exact sequence for symplectic Floer cohomology}, Topology 42 (2003) 1003--1063.
\bibitem{psf}P. Seidel, \textit{Fukaya categories and Picard--Lefschetz theory}, Zurich Lectures in Advanced Mathematics, European Mathematical Society (EMS), Z\"{u}rich, 2008.
\bibitem{psh}P. Seidel, \textit{Homological mirror symmetry for the genus $2$ curve}, J. Alg. Geom. 20: 727--769, (2011).
\bibitem{psd}P. Seidel, \textit{Disjoinable Lagrangian spheres and dilations}, Invent. Math. 197 (2014), no. 2, 299--359.
\bibitem{stb}P. Seidel and R.P. Thomas, \textit{Braid group actions on derived categories of sheaves}, Duke Math.
J. 108 (2001), 37--108.
\bibitem{nso}N. Sheridan, \textit{On the homological mirror symmetry conjecture for pairs of pants}, J. Diff. Geom. 89: 271--367, (2011).
\bibitem{sw}I. Smith and M. Wemyss, \textit{Double bubble plumbings and two-curve flops}, Sel. Math. New Ser. 29, 29 (2023).
\bibitem{yto}Y. Toda, \textit{On a certain generalization of spherical twists}, Bull. Soc. Math. France 135(1), 119--134 (2007).
\bibitem{yt1}Y. Toda, \textit{Stability conditions and crepant small resolutions}, Trans. Amer. Math. Soc. 360 (2008), 6149--6178.
\bibitem{yt2}Y. Toda, \textit{Stability conditions and Calabi-Yau fibrations}, J. Algebraic Geom. 18 (2009),101--133.
\bibitem{VdB}M. Van den Bergh, \textit{Three-dimensional flops and noncommutative rings}, Duke Math. J. 122(3):423--455, 2004.
\bibitem{mwa}M. Wemyss, \textit{A lockdown survey on cDV singularities}, from ``McKay correspondence, mutation and related topics” (Y Ito, A Ishii, O Iyama, editors), Adv. Stud. Pure Math. 88, Math. Soc. Japan, Tokyo (2023), 47--93.
\bibitem{ww}K. Wehrheim and C.T. Woodward, \textit{Exact triangle for fibered Dehn twists}, Res. Math. Sci. 3, 17 (2016).
\bibitem{bxh}B. Xie, \textit{Homological mirror symmetry for crepant resolutions of compound $A_n$ singularities}, Master Thesis, Uppsala University, in preparation.
\bibitem{jzp}J. Zhao, \textit{Periodic symplectic cohomologies and obstructions to exact Lagrangian immersions}, PhD thesis, Columbia University, 2016.
\end{thebibliography}
\end{document}